\documentclass[12pt,a4paper,reqno]{amsart} 
\usepackage{bbm,enumitem}

\usepackage[utf8]{inputenc}

\usepackage{amsmath}%,bbm}
\usepackage{amssymb}
\usepackage{graphicx}
\usepackage{color}
\usepackage{latexsym}
\usepackage[T1]{fontenc}
\usepackage{stmaryrd,comment}
\usepackage{url}
\usepackage[plainpages=false,pdfpagelabels,colorlinks=true,citecolor=blue,hypertexnames=false]{hyperref}

\usepackage[normalem]{ ulem }
\usepackage{soul}
\def\tsing{t_{\mathrm{c}}}

\def\ow{\overleftarrow{w}}
\def\oS{\overline{S}}
\def\vl{W} % active vertex left
\def\vr{E} % active vertex right 
\def\xstart{x_{\mathrm{start}}} % notations for starting/ending position, minimal value of x and y 
\def\ystart{y_{\mathrm{start}}}
\def\xend{x_{\mathrm{end}}}
\def\yend{y_{\mathrm{end}}}
\def\xmin{x_{\mathrm{min}}}
\def\ymin{y_{\mathrm{min}}}

\newcommand{\ns}{{\mathbb N}} % Natural numbers
\newcommand{\zs}{{\mathbb Z}} % Integers
\newcommand{\qs}{{\mathbb Q}}  % Rationals
\newcommand{\cs}{{\mathbb C}} % Complex numbers
\newcommand{\bx}{\bar x}
\newcommand{\bX}{\bar X}

\newcommand{\by}{\bar y}
\newcommand{\bz}{\bar z}

\def\cF{\mathcal{F}}
\def\cN{\mathcal{N}}
\def\cH{\mathcal{H}}
\def\cQ{\mathcal{Q}}
\newcommand{\GA}{\mathbb{A}}

\DeclareMathOperator{\Tpol}{T}

\newcommand{\gz}{\boldsymbol z}

\newcommand{\Wb}{{\sf{W}}}
\newcommand{\Ab}{{\sf{A}}}

\def\Ng{N_{\geq}}
\def\Nl{N_{<}}

\def\bxk{\hat{\mathbf{x}}}

\def\Sab{S^{(a,b)}}
\def\tiSab{\tilde{S}^{(a,b)}}
\def\Rab{R^{(a,b)}}
\def\Rpab{R_{\geq}^{(a,b)}}

\def\boldP{\mathbb{P}}

\def\dd{\mathrm{d}}

\def\tq{\widetilde{q}}

\newcommand{\cD}{\mathcal D}
\newcommand{\cE}{\mathcal E}

\newcommand{\cL}{\mathcal L}

\newcommand{\cW}{\mathcal W}

\newcommand{\cS}{\mathcal S}

\newcommand{\PP}{\mathcal P}

\def\cS{\mathcal{S}}
\def\bu{\bar{u}}
\def\bv{\bar{v}}
\def\bx{\bar{x}}
\def\bX{\bar{X}}
\def\by{\bar{y}}

\def\A{A}

\newtheorem{Theorem}{Theorem}[section]
\newtheorem*{Theorem*}{Theorem~\ref{thm:asympt} (with
  explicit constants)}
\newtheorem{Proposition}[Theorem]{Proposition}

\newtheorem{Corollary}[Theorem]{Corollary}

\newtheorem{Definition}[Theorem]{Definition}
\newtheorem{Lemma}[Theorem]{Lemma}

\theoremstyle{Definition}

\theoremstyle{remark}
\newtheorem*{remark}{Remark}
\newtheorem*{remarks}{Remarks}

\newcommand{\beq}{\begin{equation}}
\newcommand{\eeq}{\end{equation}}

\newcommand{\gf}{generating function}
\newcommand{\gfs}{generating functions}
\newcommand{\fps}{formal power series}

\def\EE{\mathbb{E}}
\def\PP{\mathbb{P}}

\def\emm#1,{{\em #1}}

\catcode`\@=11
\def\section{\@startsection{section}{1}%
 \z@{.7\linespacing\@plus\linespacing}{.5\linespacing}%
 {\normalfont\bfseries\scshape\centering}}

\def\subsection{\@startsection{subsection}{2}%
 \z@{.5\linespacing\@plus\linespacing}{.5\linespacing}%
 {\normalfont\bfseries\scshape}}

\def\subsubsection{\@startsection{subsubsection}{3}%
 \z@{.5\linespacing\@plus\linespacing}{-.5em}%{.5\linespacing}%
 {\normalfont\bfseries\itshape}}
\catcode`\@=12

\def\qed{$\hfill{\vrule height 3pt width 5pt depth 2pt}$}
\def\qee{$\hfill{\Box}$}

\renewcommand {\le}{\leqslant}
\renewcommand {\ge}{\geqslant}
\renewcommand {\leq}{\leqslant}
\renewcommand {\geq}{\geqslant}

\begin{document}
\title{Plane bipolar orientations and quadrant walks}

\author{Mireille Bousquet-M\'elou} 
\address{CNRS, Laboratoire Bordelais de Recherche en Informatique, Universit\'e de Bordeaux, France} \email{mireille.bousquet-melou@u-bordeaux.fr}

\author{\'Eric Fusy} 
\address{CNRS, Laboratoire d'Informatique de l'\'Ecole Polytechnique, France} \email{fusy@lix.polytechnique.fr}

\author{Kilian Raschel} 
\address{CNRS, Institut Denis Poisson, Universit\'e de Tours et Universit\'e d'Orl\'eans, France} \email{raschel@math.cnrs.fr}

\thanks{MBM and \'EF were partially supported by the French ``Agence Nationale
de la Recherche'', respectively via grants GRAAL, ANR-14-CE25-0014 and GATO, ANR-16-CE40-0009-01. This project has received funding from the European Research Council (ERC) under the European Union's Horizon 2020 research and innovation programme under the Grant Agreement No 759702.}

%%%%%%%%%%%%%%%%%%%%%%%%%%%%%%%%%%%%%%%%%%%
   

%*********************** Abstract ******************************
\begin{abstract}
Bipolar orientations of planar maps have recently attracted some interest in
combinatorics, probability theory and theoretical physics. Plane bipolar
orientations with $n$ edges are known to be counted by the $n$th
Baxter number $b(n)$, which can be defined by a linear recurrence
relation with polynomial coefficients. Equivalently, the associated \gf\
$\sum_n b(n) t^n$ is \emm D-finite,. In this
paper, we address a much refined enumeration problem, where we record
for every $r$ the number of faces of degree~$r$. When these degrees
are bounded, {we show that} the associated \gf\ is given as the constant term of a
multivariate rational series, and  thus is still D-finite. 
We also provide detailed asymptotic estimates for the corresponding
numbers.

The methods used earlier to count all plane bipolar orientations, regardless of their
face degrees, do not generalize easily to record face
degrees. Instead, we start from a recent bijection, due to Kenyon \emph{et
al.}, that sends bipolar orientations onto certain lattice walks confined to
the first quadrant. Due to this bijection, the study of bipolar
orientations meets the study of walks confined to a cone, which has
been extremely active in the past~15 years. Some of our proofs rely on
recent developments in this field, while others are purely
bijective. Our asymptotic results also involve probabilistic arguments.
\end{abstract}

\date{\today}
\maketitle

\begin{flushright} \emph{\`A Christian Krattenthaler \`a l'occasion de son
    $60^{{e}}$ anniversaire, \\avec admiration, reconnaissance et amiti\'e}\end{flushright}

\thispagestyle{myheadings}
\font\rms=cmr8 
\font\its=cmti8 
\font\bfs=cmbx8

\markright{\its S\'eminaire Lotharingien de
Combinatoire \bfs 81 \rms (2020), Article~B81l\hfill}
\def\thepage{}

%%%%%%%%%%%%%%%%%%%%%%%%%%%%%%%%%%%%%%%%%%%%%%%%%%%%%%%%%%%%%
\section{Introduction}
%%%%%%%%%%%%%%%%%%%%%%%%%%%%%%%%%%%%%%%%%%%%%%%%%%%%%%%%%%%%%
A \emm planar map, is a connected planar multigraph
embedded in the plane, and taken up to orientation preserving
homeomorphism (Figure~\ref{fig:defs}).  
The enumeration of planar maps is a venerable topic in combinatorics, which was born in the early sixties with the pioneering work of William
Tutte~\cite{tutte-triangulations,tutte-census-maps}. It is also
studied  in theoretical
physics, where planar maps are seen as a discrete model of \emm quantum
gravity,~\cite{BIPZ,BIZ}. The enumeration of maps also has connections with
factorizations of permutations, and hence representations of the
symmetric group~\cite{Jackson:Harer-Zagier,jackson-visentin}. Finally, 40 years after the first enumerative results of Tutte, planar maps  crossed the border between combinatorics and probability theory, where they are now studied as
 random metric spaces~\cite{angel-schramm,chassaing-schaeffer,le-gall-topological,marckert-mokkadem}. The limit behaviour of large
 random planar maps is now well understood, and gave birth to a
variety of  limiting objects, either continuous
% like the Brownian map~
\cite{curien-legall-plane,legall,legall-miermont-large-faces,miermont}, or
 discrete  %like the UIPQ (uniform infinite planar quadrangulation)~
\cite{angel-schramm,chassaing-durhuus,curien-miermont,menard-same}.

\begin{figure}[ht]
  \centering
 \includegraphics[width=8cm]{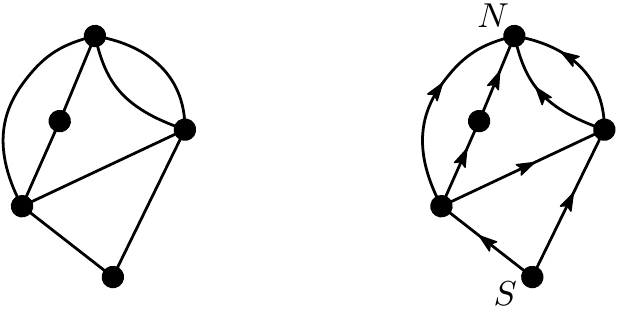}
  \caption{Left: a  planar map. Right: the same map, equipped with a bipolar
    orientation with source $S$ and sink $N$.}
  \label{fig:defs}
\end{figure}

The enumeration of maps equipped with some additional structure (\emph{e.g.}, a
spanning tree, a proper colouring, a self-avoiding-walk, a configuration of the Ising model \dots) has attracted the interest of both combinatorialists and theoretical physicists since the early days of this
study~\cite{DK88,Ka86,mullin-boisees,lambda12,tutte-dichromatic-sums}. 
This paper is devoted to the enumeration of planar maps equipped with
a \emm bipolar orientation,:  an acyclic orientation of its edges, having a unique source and a unique sink, 
both incident to the outer face (Figure~\ref{fig:defs}). 

The number of bipolar orientations of a given  multigraph  is an
important invariant in graph theory~\cite{FOR:95,OdM:1994}. Given a multigraph $G$ with a directed edge $(S,N)$, the number of bipolar orientations of $G$ with
source $S$ and sink $N$ is (up to a sign) the derivative of the chromatic polynomial $\chi_G(\lambda)$, evaluated at $\lambda=1$. It is also the
coefficient of $x^1y^0$ in the Tutte polynomial
$\Tpol_G(x,y)$~\cite{greene-zaslavsky,lass-orientations}. 

%First page headline in AmS-LaTeX for S\'eminaire Lotharingien de Combinatoire
%--restoring the headers and pagenumbering
\pagenumbering{arabic}
\addtocounter{page}{1}
\markboth{\SMALL MIREILLE BOUSQUET-M\'ELOU, \'ERIC FUSY, AND
KILIAN RASCHEL}{\SMALL PLANE BIPOLAR ORIENTATIONS AND QUADRANT WALKS}

In fact, the first enumerative result on plane bipolar orientations,
due to Tutte in 1973, was stated in terms of the derivative of the chromatic
polynomial~\cite{lambda12} (the interpretation in terms of orientations was only
discovered 10 years later). One of Tutte's main results  gives the
number of bipolar orientations of 
triangulations of a digon having $k+2$ vertices  
(equivalently, $2k$ inner faces, or $3k+1$ edges), as 
\beq\label{Tutte-bip}
a(k)= \frac {2(3k)!}{k!\,(k+1)!\,(k+2)!}\sim \frac {\sqrt 3}\pi  27^k k^{-4}.
\eeq
For instance, the 5  oriented triangulations explaining Tutte's result for $k=2$
are the following ones, where all edges are implicitly oriented
upwards.
\smallskip
\begin{center}
\includegraphics[width=10cm]{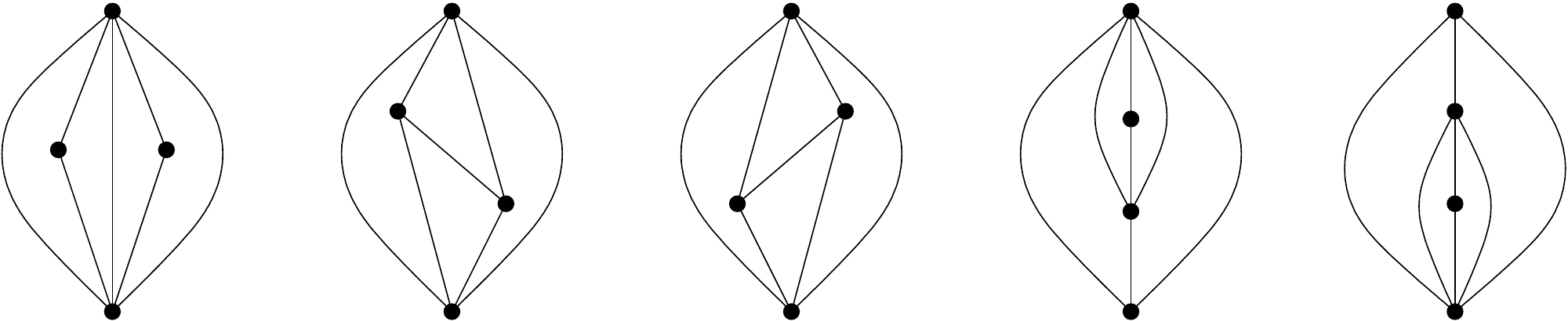}  
\end{center}
\smallskip

A more recent result, due to (Rodney) Baxter~\cite{baxter}, gives the number of
bipolar orientations of general planar maps with $n$ edges as
\beq\label{baxter-bip}
b(n)= \frac 2{n(n+1)^2}\sum_{m=1}^n\binom {n+1}{m-1}\binom {n+1}{ m}\binom {n+1
}{ m+1} \sim \frac{32}{\sqrt 3 \pi} 8^n n^{-4},
\eeq
where the asymptotic estimate can be obtained by standard
  techniques for the asymptotics of sums~\cite{odlyzko-handbook}.
For instance, the 6 bipolar orientations obtained  for $n=3$
are the following ones, where all edges are implicitly oriented
upwards.
\smallskip
\begin{center}
\includegraphics[width=10cm]{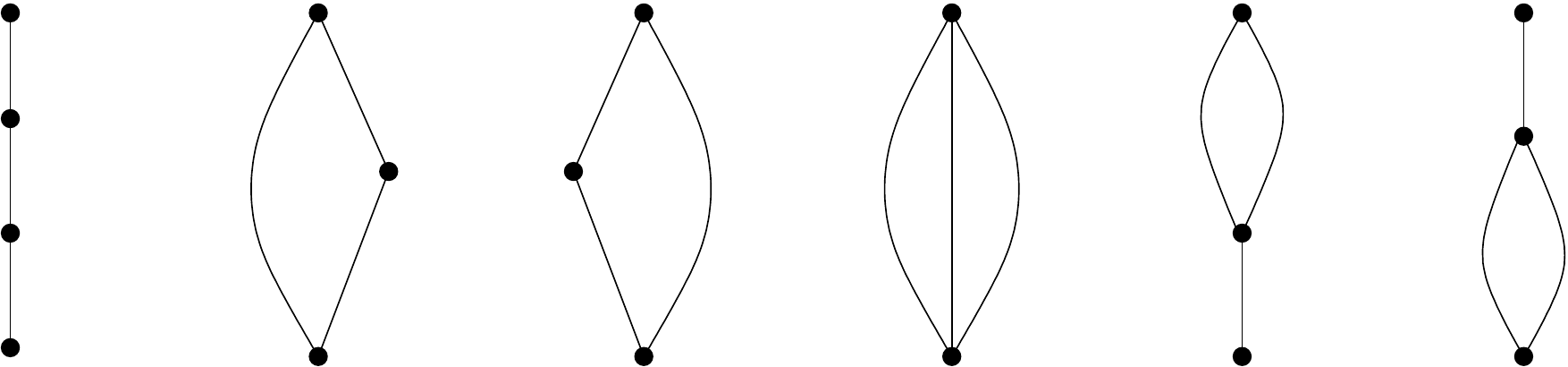}  
\end{center}
\smallskip
Baxter stated his result in terms of the Tutte polynomial, and  was apparently
unaware of its interpretation in terms of bipolar
orientations. Amusingly, he was also unaware of the fact that the
above numbers $b(n)$ were  known in the combinatorics literature as \dots 
Baxter numbers (after another Baxter, \emm Glen, Baxter). 

 Tutte's and Baxter's proofs both rely  on a  recursive
description of the chromatic (or Tutte) polynomial, which gives a
functional equation defining the \gf\ of maps equipped with a bipolar
orientation. Both solutions were  based on a guess-and-check approach, but these equations can now be solved in a more systematic
way~\cite{mbm-bcc,bousquet-motifs}. Moreover, several bijective proofs of Baxter's result
have been found, by constructing bijections between plane bipolar
orientations and various objects known to be counted by Baxter
numbers, like Baxter permutations, pairs of twin trees, or configurations
of three non-intersecting lattice
paths~\cite{AlPo15,bonichon-mbm-fusy,felsner2011bijections,fusy-bipolar}.

Tutte's and Baxter's results share some common features, for instance the
exponent~$-4$ occurring in the asymptotic estimate. Moreover, both
sequences $a(k)$ and $b(n)$ are \emm polynomially recursive,, that is,
 they satisfy a linear recurrence relation with polynomial
coefficients:
\begin{align*}
(k+1)(k+2)a(k)&= 3(3k-1)(3k-2)a(k-1),
\\
  (n+2)(n+3)b(n) &= (7n^2+7n-2)b(n-1) + 8(n-1)(n-2)b(n-2).
\end{align*}
Equivalently, the associated \gfs, namely $A(t)=\sum_{k\ge 0} a(k)
t^k$ and $B(t)=\sum_{n\ge 0} b(n) t^n$ are \emm D-finite,, meaning
that they satisfy a linear differential equation with polynomial
coefficients.

In this paper, we prove  universality of these features (detailed %precise
  statements will be   given in Section~\ref{sec:counting}): for any finite
  set $\Omega\subset\{2,3,\ldots\}$
    % not reduced to $\{2\}$},
    and for any integer~$e\geq 2$, the \gf\ of plane bipolar orientations such that all inner faces have their degree in $\Omega$ and the outer face has degree $e$ is a D-finite series, given as the
constant term of an explicit multivariate rational function (for
maps not carrying an orientation, the corresponding series are known
to be systematically
algebraic~\cite{bender-canfield,BDG-planaires}).  For
instance, if we consider bipolar orientations of quadrangulations of a
digon {(with the above notation, $\Omega=\{4\}$ and $e=2$)}, having $k+2$ vertices (equivalently, $k$  inner faces, or $2k+1$ edges), then the corresponding numbers $c(k)$ satisfy
\[
(k+2)(k+1)^2c(k)=4(2k-1)(k+1)(k-1)c(k-1)+12(2k-1)(2k-3)(k-1)c(k-2),
\]
their asymptotic behaviour is
\beq\label{asympt4}
c(k)\sim \frac 9{4\sqrt 3 \pi} 12^k k^{-4},
\eeq
and their \gf\ $\sum_{k\ge 0} c(k) t^{2k}$ is the constant term (in
$x$ and $z$) of the rational series
\beq\label{quad}
\frac{(1-\bx^2 z^2-2\bx z^3)(1+3\bx^4-\bx^2/t)}{1-t(x\bz +\bx^2+\bx z+z^2)},
\eeq
expanded as a series in $t$ whose coefficients are Laurent polynomials
in $x$ and $z$ (where we used the short notations $\bx:=1/x$ and $\bz:=1/z$). 
The  counterpart of the latter result for triangulations is 
\beq\label{tri}
\sum_{k\ge 0} a(k) t^{3k}= [x^0 z^0]
\frac{(1-\bx z^2)(1+2\bx^3-\bx^2/t)}{1-t(x\bz+\bx+z)},
\eeq
{where the operator $[x^0 z^0]$ extracts the constant term in $x$ and $z$.} Constant terms (or diagonals) of
multivariate rational functions form an important subclass of D-finite
series, for which specific methods have been developed, for instance
to determine their asymptotic
behaviour~\cite{MelczerMishna2016,pemantle2013analytic}, {possibly automatically~\cite{Melczer2017,melczer2016symbolic}},
or the recurrence relations satisfied by their coefficients~\cite{bostan-lairez-salvy-issac,bostan-lairez-salvy,lairez}.

\medskip
However, this paper is not only a paper on the enumerative properties of
(decorated) maps. It is also a paper on the enumerative and
probabilistic properties of lattice walks confined to a cone. The
reason for that is that 
the first ingredient in our approach is a recent beautiful bijection
by Kenyon, Miller, Sheffield and Wilson~\cite{kenyon2015bipolar} {(denoted KMSW)}, which encodes plane bipolar
orientations by  lattice walks confined to the first
quadrant of the plane. Among all known bijections that transform bipolar orientations into different objects~\cite{AlPo15,bonichon-mbm-fusy,felsner2011bijections,fusy-bipolar,kenyon2015bipolar}, the KMSW one seems to be the only one that naturally keeps track of the degree distribution of the faces. For instance, the above numbers $c(k)$ that count oriented quadrangulations also count quadrant walks  starting and ending at the origin,
and consisting of $2k$ steps taken in $\{(-2,0), (-1, 1), (0, 2), (1,
-1)\}$  (Figure~\ref{fig:ex-bijection}).

\begin{figure}[ht]
  \centering
 \includegraphics[width=8cm]{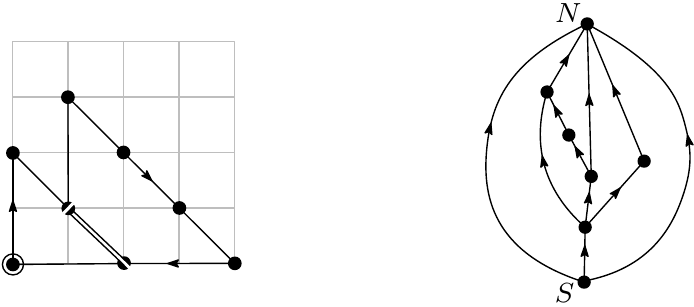}
  \caption{A walk in the quadrant and the corresponding bipolar
    orientation, through the {Kenyon--Miller--Sheffield--Wilson} (KMSW) bijection.}
  \label{fig:ex-bijection}
\end{figure}

As it happens, the enumeration of lattice walks confined to a cone is
at the moment a very active topic in enumerative
combinatorics~\cite{BeBMRa-17,BoKa08,BoRaSa14,mbm-mishna,DHRS-17,KuRa12,raschel-unified}.
These efforts have in the past 15 years led to a very good understanding of quadrant walk enumeration --- provided that all allowed steps are \emm small,,
that is, belong to $\{-1, 0, 1\}^2$. As shown by the above example of
quadrangulations, this is not the case for walks
coming from bipolar orientations (unless all faces have degree 2 or~3). 
It is only very recently that an
approach  was designed for arbitrary steps, by the first author and
two collaborators~\cite{BoBoMe18}. This approach will not work with \emm any,
collection of steps, but it \emm does, work for the well structured
step sets involved in the KMSW bijection, {which we call \emm tandem steps,}.
In fact, the enumeration of bipolar orientations provides a beautiful application, with arbitrarily large steps, of the method of~\cite{BoBoMe18}.
Hence this paper solves an enumerative problem on maps, \emm and, a
quadrant walk problem. Moreover, in
order to work out the asymptotic behaviour of the number of bipolar
orientations, we  go  through a probabilistic study of the
corresponding quadrant walks, for which we derive local limit theorems
and harmonic functions.

\begin{figure}[htb]
  \centering
  \scalebox{0.9}{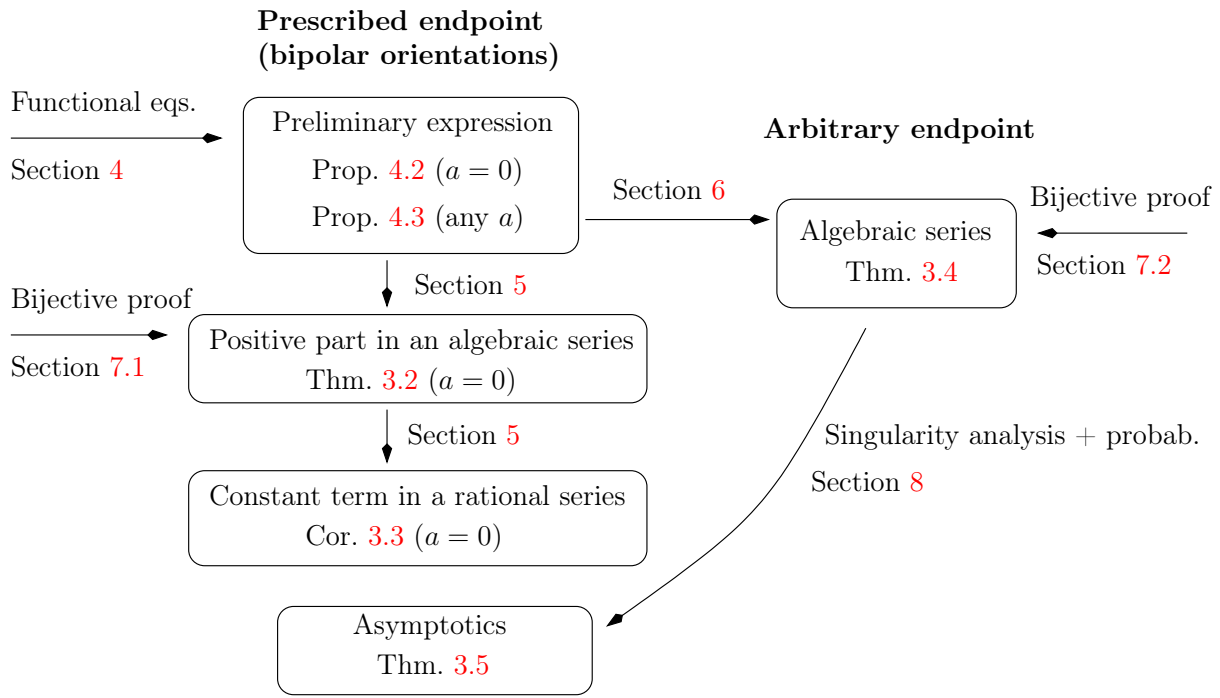}
  \caption{Enumeration of tandem walks in the quadrant starting at $(a,b)$: our results and proofs.}
  \label{fig:roadmap}
\end{figure}

\medskip\noindent{\bf{Outline of the paper {(see also Figure~\ref{fig:roadmap}).}}}
Our main enumerative
results (both exact, and asymptotic) are stated in Section~\ref{sec:counting}, after a
preliminary section where we {describe the KMSW bijection and  recall its main
properties} (Section~\ref{sec:KMSW}). We prove our exact
results in Sections~\ref{sec:eqfunc} to~\ref{sec:algeq}, using the
general approach of~\cite{BoBoMe18}. Section~\ref{sec:bij_proofs} is a bijective \emm
intermezzo,, where we provide a combinatorial explanation of our
results in terms of bipolar orientations, using the KMSW
bijection. These combinatorial proofs are more elegant than the
algebraic approach used in the earlier sections, but they are also
completely \emm ad hoc,, while the approach of~\cite{BoBoMe18} is far more
robust. In
Section~\ref{sec:asymptotic} we are back to quadrant walks, this time in a
probabilistic setting. By combining our enumerative results and
probabilistic tools inspired by a recent paper of Denisov and
Wachtel~\cite{denisov2015random}, we obtain detailed global and local limit theorems for random walks (related to bipolar orientations) conditioned to stay
in the first quadrant. We also determine explicitly 
{the} associated discrete harmonic function. This allows us to prove the asymptotic results stated in Section~\ref{sec:counting}. We conclude in Section~\ref{sec:final} with some complements --- among others, a combinatorial proof of Baxter's result~\eqref{baxter-bip} based on the KMSW walks, and a
  discussion on random generation, which leads to the uniform random
  bipolar orientation of Figure~\ref{fig:random}. This figure suggests
  that drawing planar maps at random according to the number of their
  bipolar orientations creates a bias in favour of ``fatter''
  maps. This has been recently confirmed by Ding and
  Gwynne~\cite[Fig.~2]{ding}, who showed that the  number of points in
  a ball of radius $r$ of a large bipolar-oriented map grows like~$r^d$, with $2.8\le d\le 3.3$, instead of $r^{4}$ for uniform maps. One expects the corresponding diameter of maps of size $n$ to scale like  $n^{1/d}$.

\begin{figure}[htb]
  \centering
\includegraphics[scale=0.6]{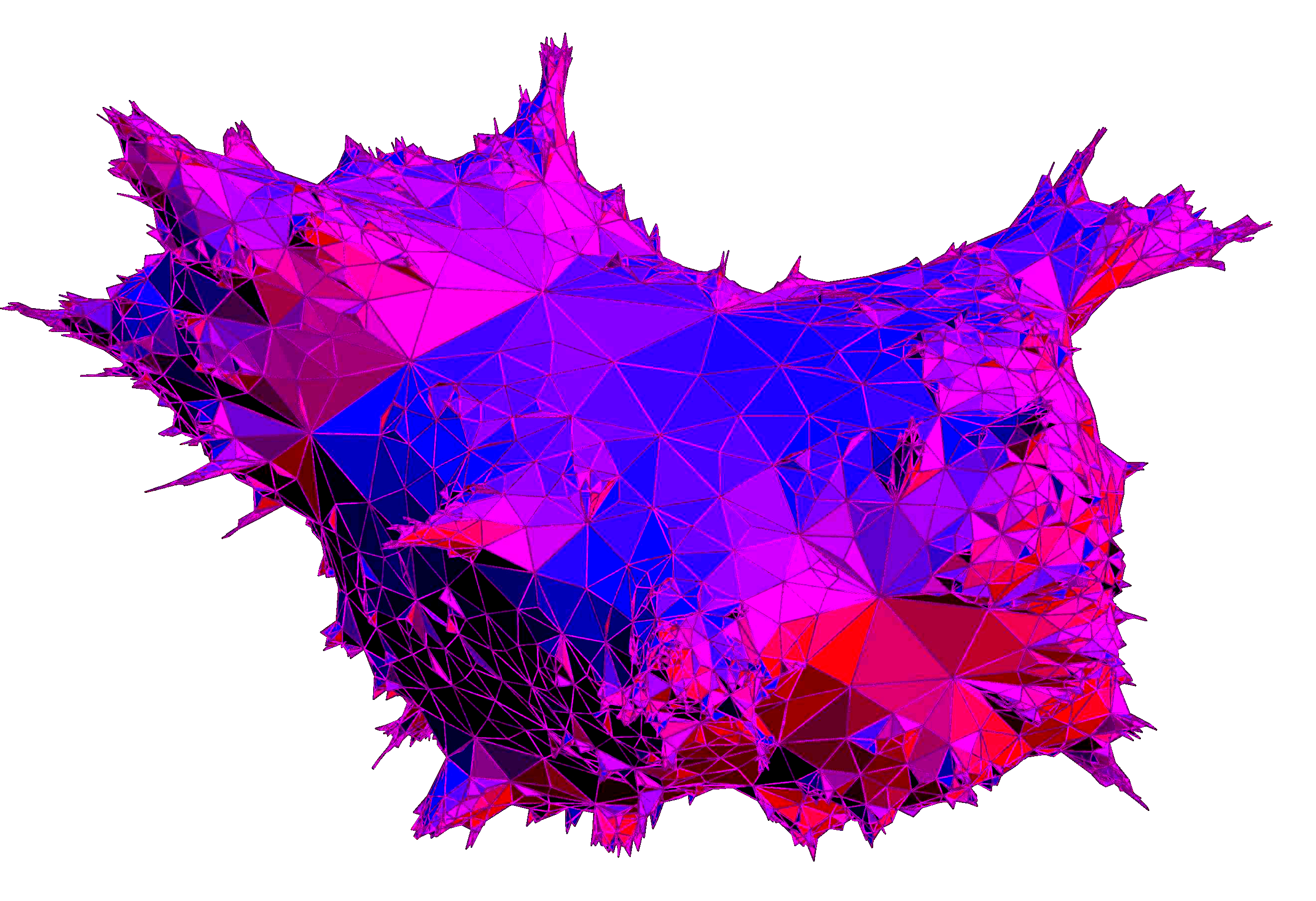}  
  \caption{A random bipolar orientation of {a} triangulation with 5000 faces (\copyright
    \ J\'er\'emie Bettinelli). The orientations of the edges are not shown.}
  \label{fig:random}
\end{figure}

%%%%%%%%%%%%%%%%%%%%%%%%%%%%%%%%%%%%%%%%%%%%%%%%%%%%%%%%%%%%%
\section{The Kenyon--Miller--Sheffield--Wilson bijection}
\label{sec:KMSW}
%%%%%%%%%%%%%%%%%%%%%%%%%%%%%%%%%%%%%%%%%%%%%%%%%%%%%%%%%%%%%
In this section, we recall a few definitions on planar maps, and
describe the KMSW bijection between bipolar orientations and certain
lattice walks.

A \emm planar map, is a proper embedding of a connected
multigraph in the plane, taken up to % orientation preserving
{orientation-preserving} 
homeomorphism. A (planar) map has naturally vertices and edges, but defines
also \emm faces,, which are
the connected components of the complement of the underlying
multigraph. One of the faces, surrounding the map, is unbounded. We
call it the \emm  outer, face. The other faces are called  \emm
inner, faces.  The \emph{degree} of a vertex or face is the number
of edges incident to it, counted with multiplicity. The degree of the
outer face is the \emm outer degree,.
A (plane) \emm bipolar orientation, 
is a planar map endowed with 
an acyclic orientation of its edges, having a unique source and a unique sink, 
both incident to the outer face.
We denote them by $S$ and $N$ respectively, as illustrated in
Figure~\ref{fig:example_bipolar}(a).
  We will usually draw the source $S$ at the bottom of the map, the sink $N$ at
the top, and orient all edges upwards. 

It is known~\cite{FOR:95} 
that bipolar orientations are characterized
by two local properties (see Figure~\ref{fig:example_bipolar}(b)):  
\begin{itemize}
\item 
the edges incident to a given
 vertex $v\notin\{S,N\}$ are partitioned
into a non-empty sequence of consecutive outgoing edges and a non-empty 
sequence of consecutive ingoing edges (in cyclic order around $v$), 
\item  in a dual way, the contour of 
each inner face $f$ is partitioned into a non-empty sequence 
of consecutive edges {oriented} clockwise around $f$ and a non-empty
sequence of consecutive edges {oriented}  counterclockwise; these are 
called the \emph{left boundary} and \emph{right boundary} of $f$, respectively. 
The edges of the outer face form two oriented paths going from $S$ to $N$,
called  \emm left, and \emm right outer boundaries, 
{following an obvious convention.} 
\end{itemize}
If an inner face $f$ has $i+1$  clockwise edges
 and $j+1$ counterclockwise edges, then $f$ is said to be \emph{of type $(i,j)$}.  

\begin{figure}
\begin{center}
\includegraphics[width=14cm]{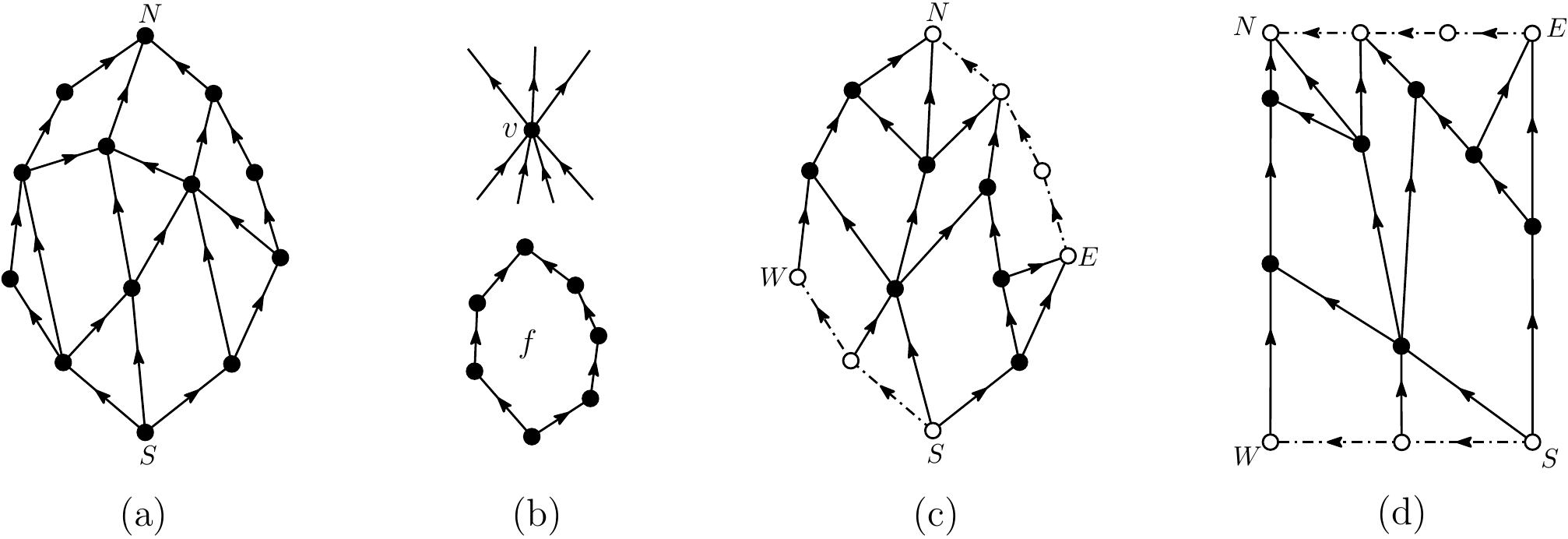}
\end{center}
\caption{(a) A plane bipolar orientation. (b)  Local 
properties of  plane bipolar orientations. (c)--(d) A marked 
bipolar orientation, drawn in two different ways.}
\label{fig:example_bipolar}
\end{figure}

A \emph{marked bipolar orientation} is a bipolar orientation
where the right (respectively left) outer boundary carries a distinguished vertex 
$\vr\neq S$  (respectively $\vl\neq N$), such that:
\begin{itemize}
\item 
 each vertex between $\vr$ and $N$
along the right outer boundary ($N$ excluded) has outdegree~$1$
 and the unique outgoing edge has an inner face on its left, 
\item similarly each vertex between $\vl$ and $S$
along the left outer boundary ($S$ excluded) has indegree $1$
and the unique ingoing edge has an inner face on its right.
\end{itemize}
 See Figure~\ref{fig:example_bipolar}(c) for an example.  
Note that a bipolar orientation can be identified with a marked bipolar
orientation, upon declaring $\vr$ to be $N$ and $\vl$ to be $S$.
The \emph{upper right boundary} (respectively \emph{lower right boundary})    
is the path from $\vr$ to $N$ (respectively from $S$ to $\vr$) along the right outer boundary; and similarly the \emph{lower left boundary} (respectively \emph{upper left boundary})  
is the path from $S$ to $\vl$ (respectively from $\vl$ to $N$) along the left outer boundary. Note that 
the upper right and lower left boundaries do not share any vertex.
A vertex or edge is called \emph{plain} if it does not belong  
to the upper right nor to the lower left boundary. In our figures,
plain vertices  are shown in black  and plain edges in solid
lines. The non-plain vertices are  white, and the non-plain edges are dashed.

At the end of 2015, Kenyon \emph{et al.}~\cite{kenyon2015bipolar} 
 introduced a bijection between certain 2-dimension\-al walks and marked bipolar orientations.
These walks, which we call \emph{tandem walks}\footnote{The name
  \emm tandem, originates from the case where the only   steps  are $(1,-1)$,
  $(-1,0)$ and $(0,1)$. Then it can be seen that such walks, starting at $(0,0)$ and restricted to the quadrant,  describe the evolution of two queues in series --- or in \emm tandem,.}, are defined as sequences  of steps of two types: {south-east steps  $(1,-1)$ (called SE steps for short)}  
and steps of the form
$(-i,j)$ with $i,j\geq 0$, which we  call \emph{face steps}; 
the \emph{level} of such a step is the integer $r=i+j$.
Note that a walk, being defined as a sequence of steps, has no prescribed starting point nor endpoint. In other words, it is defined up to translation. We often say that it is \emm non-embedded,.  When we will embed walks in the plane, we will specify a starting point explicitly.
Given a tandem walk~$w$ with successive steps $s_1,\ldots,s_n$ in~$\zs^2$, the bijection
builds a marked bipolar orientation as follows. We start with the marked bipolar orientation $O_0$ consisting of a single
edge $e=\{S,N\}$ with $\vr=N$ and $\vl=S$. The marked vertex $\vl$
will remain the same
all along the construction, but the source $S$ will move from vertex
to vertex. 
Then, for $k$ from $1$ to $n$, we construct a 
marked bipolar orientation $O_k$ from $O_{k-1}$
and the $k$th step $s_k$. Two cases may occur:
\begin{itemize}
\item
If $s_k$ is a SE step (Figure~\ref{fig:situationsb}), we push $\vr$ one step up; if $\vr\not= N$ in $O_{k-1}$,
this means that one dashed edge {of $O_{k-1}$} becomes plain in $O_k$; 
otherwise we also push $N=E$ one step up, thereby 
creating a new (plain) edge that is both on the left and right outer boundary
of the orientation. {In this case,} we still have $\vr=N$ in $O_k$.  
\item
If $s_k$ is a face step $(-i,j)$, we first glue a new inner face $f$ 
of type $(i,j)$ to the right outer boundary 
in  such a way {that} the  upper vertex  of $f$ is  $\vr$ and 
the lower vertex of $f$ lies on the right outer boundary of $O_k$
(Figure~\ref{fig:situationsa}); more precisely, if
$i+1$ does not exceed the length of the lower right boundary of $O_{k-1}$,
then the $i+1$ left edges of $f$ are identified with the top $i+1$
edges on this boundary; otherwise, the lowest left edges of $f$
are dashed and become
part of the lower left boundary in $O_k$, while the lower vertex of $f$
becomes the source of $O_k$. We finally choose %update
$\vr$ to be the end of the first edge along the right boundary of $f$.
\end{itemize}

\begin{figure}[h!]
\begin{center}
\includegraphics[width=0.8\linewidth]{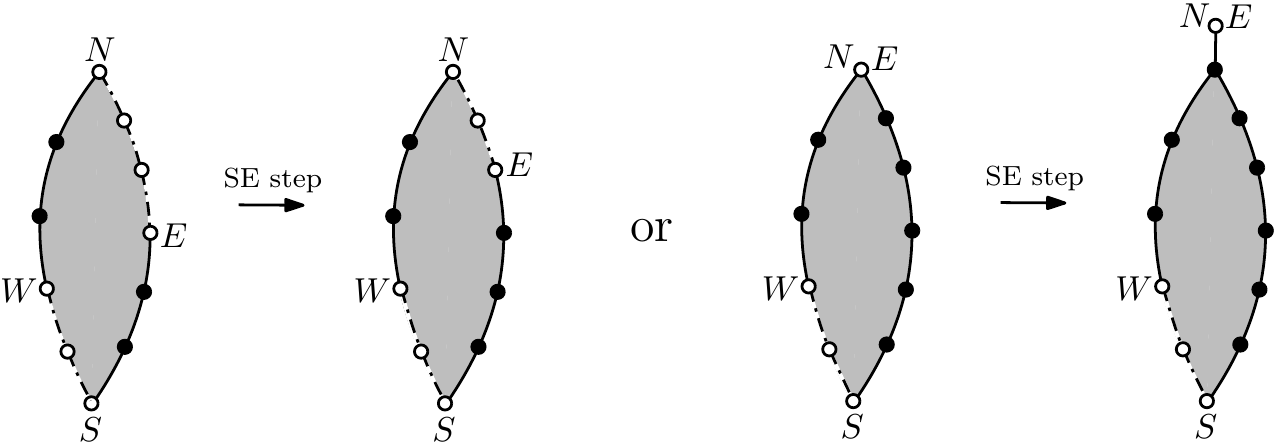}
\end{center}
\caption{Updating the marked bipolar orientation when a SE step is read.}
\label{fig:situationsb}
\end{figure}
\begin{figure}[h!]
\begin{center}
\includegraphics[width=0.8\linewidth]{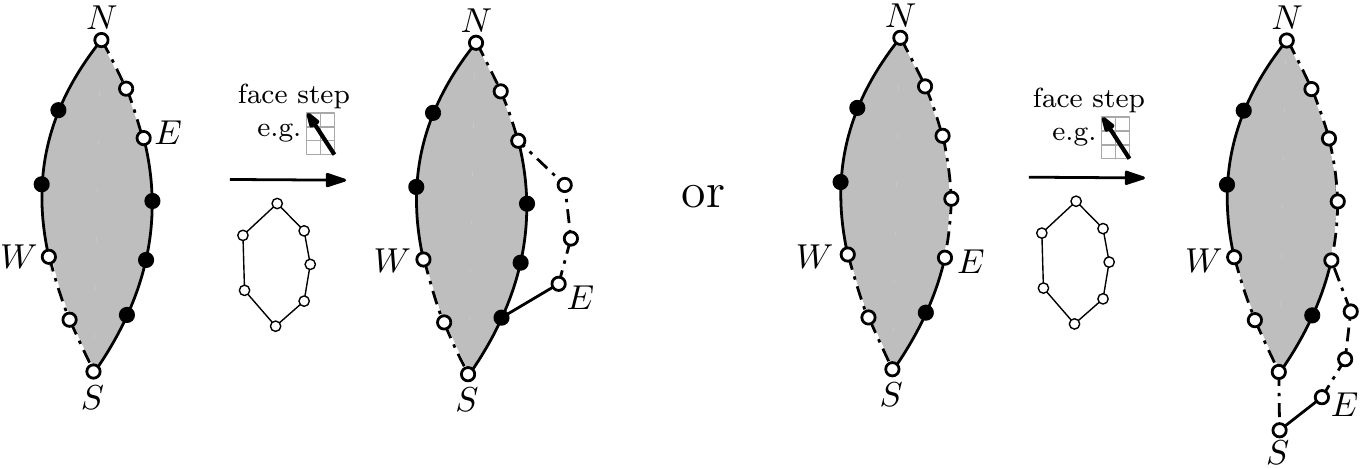}
\end{center}
\caption{Updating the marked bipolar orientation when a face step
  $(-2,3)$ is read.}
\label{fig:situationsa}
\end{figure}

We denote the marked bipolar orientation 
constructed from $w$ by $\Phi(w):=O_n$. 
A complete example is detailed in  Figure~\ref{fig:kenyon_bij}.

\begin{figure}[htb]
\begin{center}
\includegraphics[width=14cm]{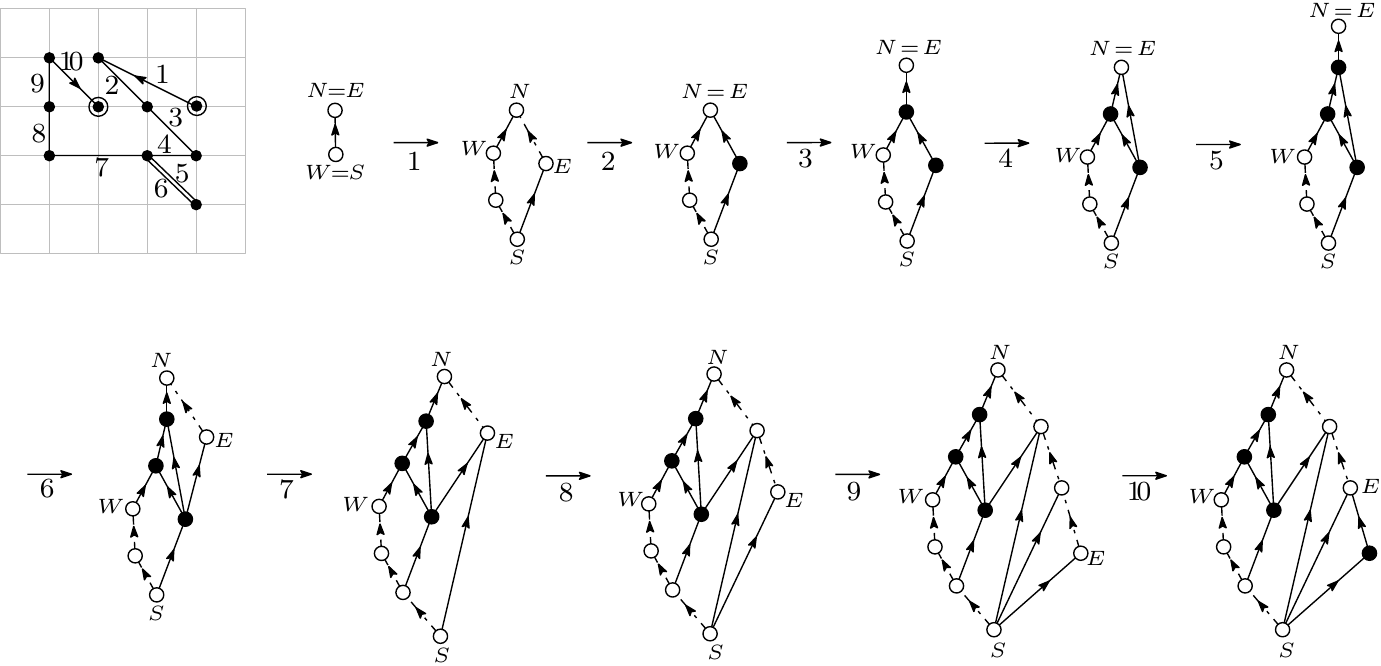}
\end{center}
\caption{A tandem walk of length $10$ and the associated
marked bipolar orientation, which is constructed step by step.}
\label{fig:kenyon_bij}
\end{figure}

\begin{Theorem}[\sc Kenyon et al.~\cite{kenyon2015bipolar}]\label{thm:KMSW}
The mapping $\Phi$ is a bijection between (non-embedded) tandem walks with~$n$ steps  and marked bipolar orientations with $n+1$ plain edges. It transforms
 SE steps into plain vertices, and face steps of level $r$ into inner
 faces of degree $r+2$.
\end{Theorem}

\begin{figure}[b]
\begin{center}
\includegraphics[width=0.8\linewidth]{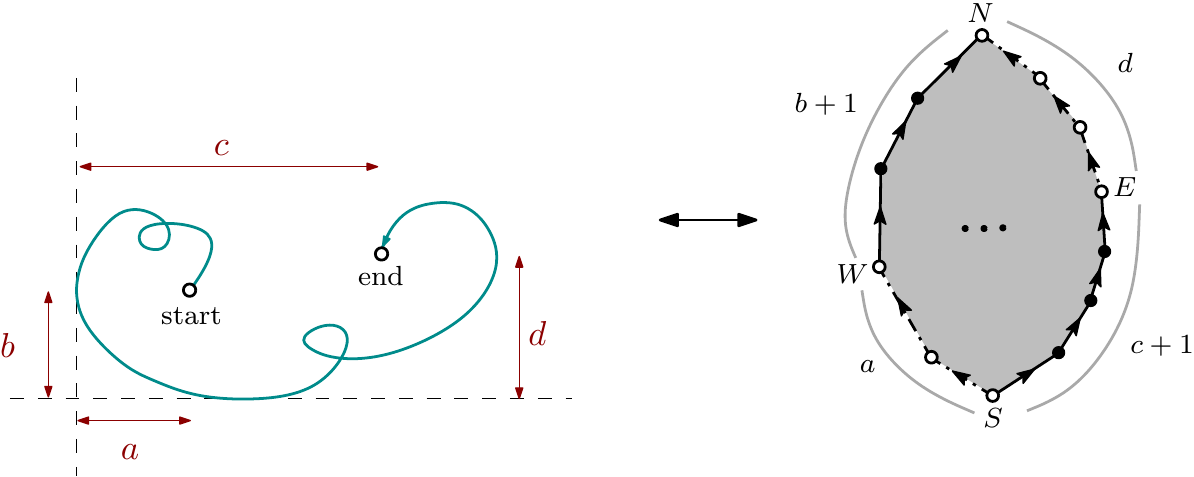}
\end{center}
\caption{The correspondence between the coordinates of the endpoints and 
the signature in the KMSW  bijection.}
\label{fig:corresp}
\end{figure}

The boundary lengths of the orientation $\Phi(w)$ are also conveniently translated through this
bijection. Let us denote by $a$ (respectively $b+1$, $c+1$, $d$) the
length of the lower left (respectively upper left, lower right, upper right)
boundary of $\Phi(w)$ (see Figure~\ref{fig:corresp}, right). We call the $4$-tuple $(a,b;c,d)$ the \emph{signature} of the marked bipolar orientation. 
Let us embed the walk $w$ in the plane so that it starts at some point
  $(\xstart,\ystart)$.  Let $\xmin$ and $\ymin$ be the
minimal $x$- and $y$-coordinates along the walk, and let $\xend$ and
$\yend$ be the $x$- and $y$-coordinates of the 
final point of $w$. Then one easily checks that 
\begin{align}
  a&=\xstart-\xmin, %=-\xmin,
  & b&=\ystart-\ymin,
  % =-\ymin,
       \nonumber\\
c&=\xend-\xmin, & d&=\yend-\ymin, \label{sign}
\end{align}
as illustrated in Figure~\ref{fig:corresp}.
Indeed these quantities are initially all equal to $0$ when we
{start constructing} $\Phi(w)$ (that is, for the initial orientation $O_0$ and the empty walk),
and then {the parameters in each pair} (e.g., $a$ and $\xstart-\xmin)$
change in the same way at each step of the construction 
(see Figures~\ref{fig:situationsb} and~\ref{fig:situationsa}).  

If we embed $w$ in the plane so that $\xmin=\ymin=0$, then it becomes a tandem walk in the quadrant $\{x\geq 0, y\geq 0\}$
starting at $(a,b)$, ending at $(c,d)$, constrained to visit at least
once the $x$-axis and  the $y$-axis. 
For unmarked bipolar orientations ($a=d=0$), the constraint holds
automatically, and the following corollary~\cite[Thm.~2.2]{kenyon2015bipolar} is obtained.

\begin{Corollary}\label{cor:bij-bip}
  The mapping $\Phi$ specializes to  a bijection between tandem
  walks of length $n$   in the quadrant,
starting at $(0,b)$ and ending at $(c,0)$, and bipolar orientations
with $n+1$ edges, having $b+1$ edges on %its
the left outer boundary and $c+1$ edges on the %its
 right outer boundary.

Specializing further to \emm excursions,, that is, walks starting and ending
at $(0,0)$, we obtain, upon erasing the  two outer edges, a
bijection between excursions of length $n$ and bipolar orientations
with $n-1$ edges.
\end{Corollary}

\begin{figure}[b]
\begin{center}
\includegraphics[width=\linewidth]{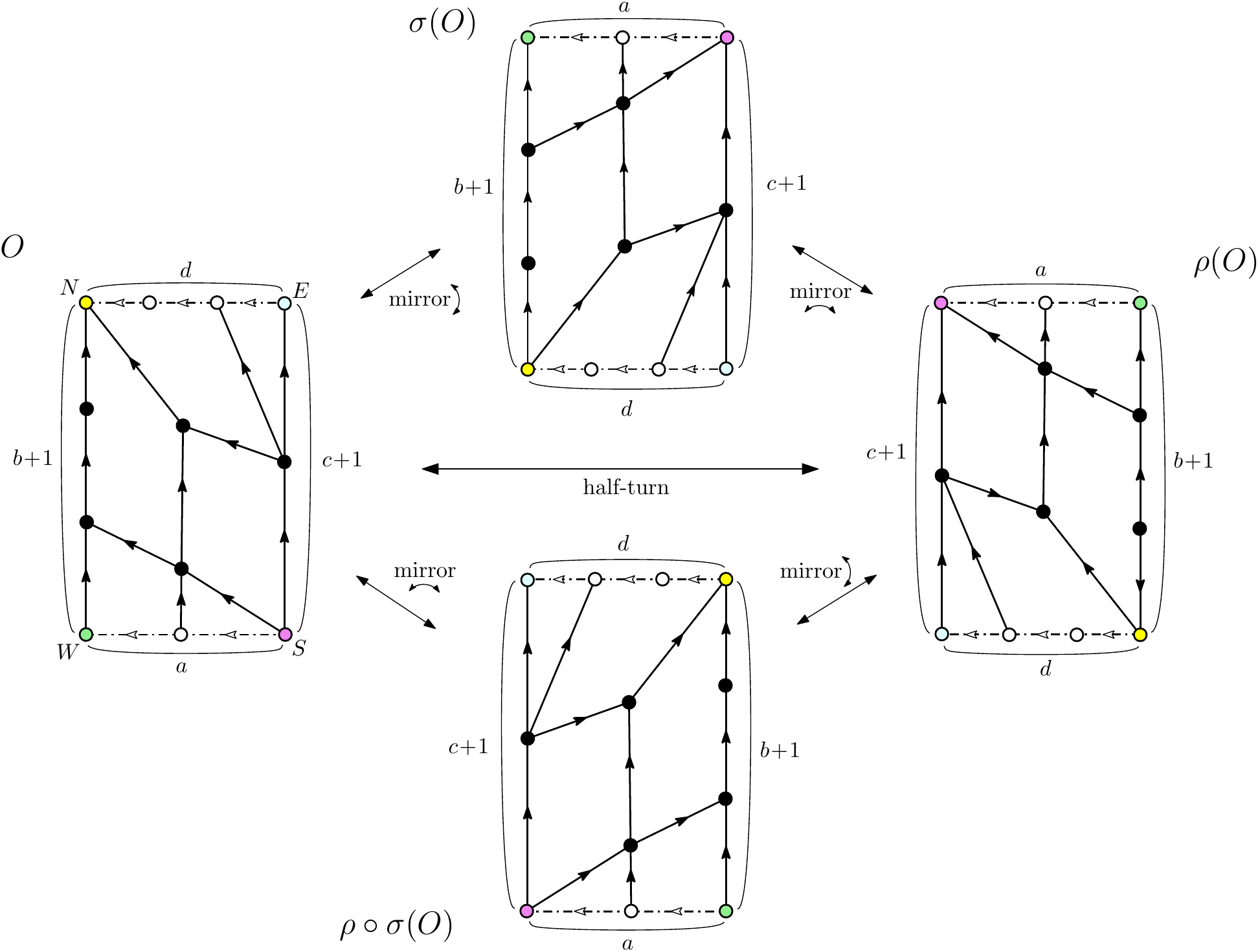}
\end{center}
\caption{The orbit of a marked bipolar orientation under the action of the two involutions $\sigma$ and $\rho$; the dashed edges are drawn as horizontal segments, which makes it easier to see the mirror-effect of $\sigma$ and
$\rho\circ\sigma$.}
\label{fig:involutions}
\end{figure}

We now define two involutions on marked bipolar orientations. 

\begin{Definition}\label{def:symmetries}
  Let $O$ be a marked bipolar orientation of signature $(a,b;c,d)$. We
  define $\rho(O)$ as the marked bipolar orientation obtained by
  reversing all edge directions in $O$. This exchanges the roles of $N$ and $S$ on the one hand, of $E$ and $W$ on the other hand. The signature of $\rho(0)$ is   $(d,c;b,a)$.

  We     define $\sigma(O)$ by first reflecting $O$ in a
  mirror, then reversing the edge directions of plain edges only.
 The new points $S', N', W'$ and $E'$ in
$\sigma(O)$ correspond to $E, W, N$ and $S$, respectively. 
The signature of  $\sigma(O)$ is $(d,b;c,a)$. 
\end{Definition}
This description clearly shows that $\rho$ and $\sigma$ are
involutions. Moreover, the marked bipolar orientations $(\rho \circ \sigma)(O)$ and $(\sigma \circ
\rho) (O)$ are both obtained by reflecting $O$ in a mirror and reversing
the directions of all dashed edges, and thus they coincide. Hence
  $\rho$ and $\sigma$ generate  a dihedral group of
order $4$. Their effect is perhaps better seen if we draw  marked
orientations with the rectangular convention adopted on the right of
Figure~\ref{fig:example_bipolar}: all plain edges go upward, while dashed edges go left. Then we can forget edge directions,  $\rho$ corresponds to a half-turn rotation, $\sigma$ to a reflection in a horizontal mirror,  and $\rho\circ\sigma$ to a reflection in a vertical mirror. This is illustrated in Figure~\ref{fig:involutions}. 

It is easy to describe the involution on tandem walks induced by
$\rho$. This description (which we will not exploit)
is used in the proof of Theorem~2.2 in~\cite{kenyon2015bipolar}.

\begin{Proposition}
  Let $w=s_1, \ldots, s_n$ be a tandem walk, and $O=\Phi(w)$ the
  corresponding marked bipolar orientation. Let $\widetilde s_k$ be
  $(-j,i)$ if $s_k=(-i,j)$, for any $i, j \in \zs^2$, and define $\widetilde
  w=\widetilde s_n , \ldots, \widetilde s_1$. Then $\Phi(\widetilde w)= \rho(O)$.
\end{Proposition}

It seems more difficult to directly describe the involution on tandem walks induced by~$\sigma$. This involution will be used  
in Section~\ref{sec:prop_alg_combi} to prove bijectively some of our
enumerative results.

%%%%%%%%%%%%%%%%%%%%%%%%%%%%%%%%%%%%%%%%%%%%%%%%%%%%%%%%%%%%%%%%%%%%%%%%%%%
\section{Counting tandem walks in the quadrant}
\label{sec:counting}
%%%%%%%%%%%%%%%%%%%%%%%%%%%%%%%%%%%%%%%%%%%%%%%%%%%%%%%%%%%%%%%%%%%%%%%%%%%

The KMSW bijection described in the  previous section relates two
topics that are actively studied at the moment in combinatorics and
probability theory: planar maps, here equipped with a bipolar orientation,
and walks confined to a cone, here the first quadrant. In this
section, we state our main results on the   enumeration of these
objects. 

The enumeration of walks confined to the quadrant is %now
 well understood when the walk consists of \emm small steps,, that is, %steps 
when the steps are   taken in $\{-1, 0, 1\}^2$. This is not the case here, unless we only consider
orientations with inner faces of degree $2$ and $3$. 
 Recently, the first author and two of her collaborators developed an approach to count quadrant walks with larger steps, generalizing in particular the
definition of a certain  group that plays a key role in the small step
case~\cite{BoBoMe18}. This approach does not apply to all possible
step sets; in particular, it requires  that the group (or what has
replaced it for large steps, namely a certain \emm orbit,) is
finite. This is the case for tandem walks, and  we will count them
using the approach of~\cite{BoBoMe18}. 

Given two points $(a,b)$ and $(c,d)$ in the first quadrant, we denote
by 
$$Q_{c,d}^{a,b} \equiv Q_{c,d}^{a,b}(t,z_0, z_1, \ldots)$$ 
the
\gf \ of tandem walks going from $(a,b)$ to $(c,d)$ in the
quadrant, where every edge is weighted
by $t$, and every face step of level $r$ by $z_r$ (which we take as an
indeterminate). For instance, the
walk of Figure~\ref{fig:kenyon_bij}, once translated so that it becomes a quadrant
walk visiting both coordinate axes, contributes 
$t^{10}z_1^3 z_2^2z_3$  to the series $Q^{3,2}_{1,2}$.

Returning to bipolar orientations, it follows from Corollary~\ref{cor:bij-bip} that $tQ^{0,b}_{c,0}$ counts bipolar
orientations with left (respectively right) outer boundary of length $b+1$
(respectively $c+1$)  with a weight~$t$ per edge, and $z_r$ per inner face
of degree $r+2$.  Also, $\frac 1 t (Q^{0,0}_{0,0}-1)$, specialized to
$z_r=1$ for all~$r$, simply counts bipolar
orientations by edges. As recalled in the introduction,  the number of
bipolar orientations having $n$ edges is 
the $n$th Baxter number $b(n)$, given by~\eqref{baxter-bip}.
We will recover this result using the bijection with tandem walks in
Section~\ref{sec:Baxter}. 

If we want to count \emm marked, bipolar orientations of signature
$(a,b;c,d)$, we must recall that they are in bijection with tandem
walks in the quadrant, {joining $(a,b)$ to $(c,d)$ and}  constrained to visit both coordinate axes. An
inclusion-exclusion argument gives their \gf\ as
\[
t \left(Q_{c,d}^{a,b}-Q_{c,d-1}^{a,b-1}
-Q_{c-1,d}^{a-1,b}+Q_{c-1,d-1}^{a-1,b-1}\right).
\]
Here, every \emm plain, edge is weighted by $t$, and every 
inner  face of degree $r+2$ by $z_r$.

Keeping $a$ and $b$
fixed, we group all the $ Q_{c,d}^{a,b}$ into a bigger \gf \ that
counts quadrant tandem walks starting at $(a,b)$:
\[
Q^{(a,b)}(x,y):=\sum_{c,d\geq 0}Q_{c,d}^{a,b}\, x^c y^d.
\] 
 By Corollary~\ref{cor:bij-bip},  we are especially interested in the
series $Q^{(0,b)}(x,0)$, since $txQ^{(0,b)}(x,0)$ counts %marked
bipolar orientations with a left boundary  of length $b+1$, by
edges ($t$),  face degrees ($z_r$ for each inner face of degree
$r+2$) and length of the lower right boundary ($x$).

%==============================================
\subsection{Preliminaries}
%==============================================

\subsubsection{Walk generating functions.} It may be a bit unusual to
involve  infinitely many variables in \gfs, as we do with the
$z_r$'s. Hence let us clarify in which  ring these series  live. 

{Many of the series that we consider count (sometimes implicitly) embedded tandem walks, not necessarily confined
to the quadrant, and record with variables $x$ and $y$ the coordinates
of their endpoint.} Then a natural option is to work with formal power series
in infinitely  many variables $t, z_0, z_1, \ldots$ with coefficients
in $\qs[x,1/x,y,1/y]$, the ring of Laurent polynomials in $x$ and~$y$. However, it will sometimes be convenient to handle a \emm
finite, collection of steps, and moreover to assign real values to
the $z_r$'s. This is why we usually assume that $z_r=0$ for 
$r>p$, for
some arbitrary $p$, and take our series in the ring of \fps\ in $t$
with coefficients in $\qs[x,1/x,y,1/y,z_0, z_1, \ldots, z_p]$. We call
this specialization the \emm $p$-specialization,, and the
corresponding walks, \emm $p$-tandem walks,. Both points of view
can be reconciled by letting $p\rightarrow \infty$. Indeed, if a walk
starting at $(a,b)$ and ending at $(c,d)$ uses a face step of level~$r$, then $(d-c)-(b-a) \ge r-2(n-1)$ (look at the projection of the
walk on a line of slope~$-1$). That is, $r \le
(d-c)-(b-a)+2(n-1)$. Hence, when $a,b,c,d,n$ are fixed, 
a walk of length $n$ going from $(a,b)$ to
$(c,d)$, cannot use arbitrarily large
steps. This means that for~$p$ large enough, the coefficient of
$t^nx^cy^d$ in any walk \gf\ (with a fixed starting point) is a
polynomial in the $z_r$'s which is independent of $p$.

\subsubsection{Periodicities.}
Throughout the paper, we will meet periodicity conditions, describing which points can be reached from say, the origin, in a fixed
number of steps. So let us clarify this right now. For a step set
$\cS$, we call a walk consisting of steps taken
in~$\cS$ an \emm $\cS$-walk,. {The following terminology is borrowed from
  Spitzer~\cite[Chap.~1.5]{Spitzer}. Take a finite step set
  $\cS\subset \zs^2$,
  and denote the lattice spanned by $\cS$ in $\zs^2$ by  $\Lambda$. We
  say that~$\cS$ is \emm strongly aperiodic, if,
  for any $(i,j)\in \Lambda$, the lattice generated by $(i,j)+\cS$
  coincides with $\Lambda$. In this case, for  $(i,j)\in \Lambda$,
  there exists $N_0\in \ns$ such that for all $n>N_0$, there exists an
  $\cS$-walk of length $n$ going from $(0,0)$ to $(i,j)$. We say that $\cS$ has \emm period, 1. Otherwise,
  there exists an integer $p>1$ (the \emm period,\/), such that
  for all $(i,j) \in \Lambda$ there exists $r\in \llbracket
  0,p-1\rrbracket$  such that for $n$ large enough, there exists
  an $\cS$-walk of length $n$ from $(0,0)$ to $(i,j)$ if and only if
  $n$ equals $r$ modulo $p$.}

\begin{Lemma}\label{lem:per}
  Let $D$ be a non-empty finite subset of $\ns$, not reduced to $\{0\}$, and define 
\[ 
\iota:=\gcd(r+2,r\in D).
\] 
Let $\cS_D$ be the set of steps
\[
\cS_D=\{(1,-1)\} \cup \bigcup_{r \in D} \{(-r,0), (-r+1, 1), \ldots,
(0,r)\}.
\]
Then the lattice $\Lambda_D$ spanned by $\cS_D$ is $\zs^2$ if $\iota$ is
odd, and $\{(i,j): i+j \ \hbox{even}\}$ otherwise. 

If there exists an $n$-step walk from $(0,0)$ to
$(i,j)$ with steps in $\cS_D$, then $i-j\equiv 2n$~mod~$\iota$. 
Conversely, if $n$ satisfies this condition and is large
enough, there exists an $\cS_D$-walk from $(0,0)$ to
$(i,j)$. {This means that the step set $\cS_D$ has period $\iota$ if $\iota$ is odd, $\iota/2$  otherwise.  In particular,  $\cS_D$ is  strongly aperiodic if and only if $\iota\in \{1,2\}$.}
\end{Lemma}
\begin{proof}
  If $\iota$ is odd, then there exists an odd $r$ in $D$, say $r=2s+1$
  with $s\ge0$. Then $(-s,s+1)$ belongs to $\cS_D$, and
  \[
    (-s,s+1)+ s(1,-1)=(0,1).
  \]
  Hence $\Lambda_D$ contains the vectors $(0,1)$ and $(1,-1)$ (which is
  always in $\cS_D$), and thus coincides with~$\zs^2$.

  If $\iota$ is even, then every $r\in D$ is even, and for every step
  $(i,j) \in \cS_D$ the difference $i-j$ is even: equal to $2$ for a
  SE step, to $-r$ for a step $(-i,r-i)$. Hence the same holds
  necessarily for any point  $(i,j)$ of $\Lambda_D$, which is thus included in
  $\{(i,j): i+j \ \hbox{even}\}$. Now take $r=2s\in D$ with $s\ge 1$. Then
  $(-s-1,s-1)\in \cS_D$, and
  \[
    (-s-1,s-1)+ s(1,-1)= (-1,-1).
  \]
  Hence $\Lambda_D$ contains $(-1,-1)$ and $(1,-1)$, and thus all
  points $(i,j)$ such that $i+j$ is even. We have thus  proved the first
    statement of the lemma.

\medskip
  Now consider an $\cS_D$-walk of length $n$ going from $(0,0)$ to $(i,j)$, and let $(i_k,j_k)$  be the point reached after $k$ steps.  Then 
$(i_k-j_k)- (i_{k-1}-j_{k-1})$ equals $2$ mod $\iota$ 
for every $k$.
Hence after $n$ steps, we find $i-j \equiv 2n$~mod~$\iota$.

\medskip
Let us now prove the  next result for $(i,j)=(0,0)$.  The
set $G$ of lengths $n$ such that there exists an $n$-step walk starting and ending at  $(0,0)$ (we
call such walks \emm excursions,) is an additive semi-group of~$\ns$. The
structure of semi-groups of~$\ns$ is well understood: there exists an integer
$p$ (the  period), such that $G\subset p \ns$ and $mp\in G$
for all large enough~$m$. Clearly $p=\gcd (G)$. By the
previous result, all elements $n$ of $G$ satisfy $2n\equiv 0$~mod~$\iota$, that is, $\iota \mid2n$.  Hence
the period $p$ is a multiple of $\iota$ if $\iota$ is odd, and of
$\iota/2$ otherwise. Now, saying that for any large enough $n$  such that $\iota \mid 2n$ there exists an $n$-step excursion, is equivalent to
  saying  that $p$ \emm  equals, $\iota$ if  $\iota$ is odd, and   $\iota/2$
  otherwise.  So let us first prove that $p\mid \iota$.
For each $r\in D$, there exists  an
excursion of length $r+2$ (consisting of the steps $(0,r)$ and $(-r,0)$
followed by~$r$ SE steps). Hence $D+2 \subset G$, and thus $p:=\gcd (G)$
divides~$\iota:=\gcd (D+2)$. This proves that $p=\iota$ if $\iota$ is odd, but if
$\iota$ is even, we can still have $p=\iota$ or $p=\iota/2$. So assume
 that $\iota$ is even. Then each $r\in D$ is even, and there exists
an excursion of length $1+r/2$ (consisting of the step $(-r/2,r/2)$
followed by $r/2$ SE steps). Hence $1+D/2 \subset G$, and thus
$p=\gcd(G)$ divides $\iota/2= \gcd(1+D/2)$. This concludes the proof
when $(i,j)=(0,0)$. 

\medskip
Once the period $p$ is determined, the extension to general points
$(i,j)$  is standard. See for instance the proof
of~\cite[Prop.~9]{BoBoMe18}, and references therein.
\end{proof}

\begin{remark}
 The period was already determined in the
original paper~\cite[Thm.~2.6]{kenyon2015bipolar}, where it is 
described as
\[
\gcd\left( \{ r+1: 2r\in D\}\cup \{2r+3: 2r+1 \in D\}\right).
\]
Both descriptions are of course equivalent. The reason why we prefer
to introduce $\iota$ is that this is the quantity that naturally
arises in asymptotic estimates (see for instance Corollary~\ref{coro:bipolar}).
\end{remark}

%==========================================================
\subsubsection{Some definitions and notation on \fps.}
%==========================================================

Let $\GA$ be a commutative ring and let $x$ be an indeterminate. We denote by
$\GA[x]$ (respectively $\GA[[x]]$) the ring of polynomials (respectively \fps) in $x$
with coefficients in $\GA$. If $\GA$ is a field,  then $\GA(x)$ denotes the field
of rational functions in $x$, and $\GA((x))$ the
set of Laurent series in $x$, that is, series of the form
\[ \sum_{n \ge n_0} a_n x^n,
\]
with $n_0\in \zs$ and $a_n\in \GA$.  The coefficient of $x^n$ in a   series $F(x)$ is denoted by
$[x^n]F(x)$. 

This notation is generalized to polynomials, fractions
and series in several indeterminates. For instance, the  \gf\ of
bipolar orientations, counted by edges (variable $t$) and faces
(variable $z$)  belongs to
$\qs[z][[t]]$. For a multivariate series, say $F(x,y) \in \qs[[x,y]]$,
the notation $[x^i]F(x,y)$ stands for the \emm series, $F_i(y)$ such
that $F(x,y)=\sum_i F_i(y) x^i$. It should not be mixed up with the
coefficient of $x^iy^0$ in $F(x,y)$, which we denote by $[x^i
y^0]F(x,y)$.
If $F(x,x_1, \ldots, x_d)$ is a series in the $x_i$'s whose
coefficients are Laurent series in $x$, say
\[
F(x,x_1, \ldots, x_d)= \sum_{i_1, \ldots, i_d} x_1^{i_1} \cdots
x_d^{id}
\sum_{n \ge n_0(i_1, \ldots, i_d)} a(n, i_1, \ldots, i_d) x^n,
\]
then the \emm nonnegative part of $F$ in, $x$ is the
following \fps\ in $x, x_1, \ldots, x_d$:
\[
[x^{\ge }]F(x,x_1, \ldots, x_d)= \sum_{i_1, \ldots, i_d} x_1^{i_1} \cdots
x_d^{id}
\sum_{n \ge 0} a(n, i_1, \ldots, i_d) x^n.
\]
We denote the reciprocals of variables using bars: to be precise, $\bx=1/x$,
so that $\GA[x,\bx]$ is the ring of Laurent polynomials in $x$ with
coefficients in $\GA$.

If $\GA$ is a field,  a power series $F(x) \in \GA[[x]]$
  is \emm algebraic, (over $\GA(x)$) if it satisfies a
non-trivial polynomial equation $P(x, F(x))=0$ with coefficients in
$\GA$. It is \emm differentially finite, (or \emm D-finite,) if it satisfies a non-trivial linear
differential equation with coefficients  in $\GA(x)$. For
multivariate series, D-finiteness requires the
existence of a differential equation \emm in each variable,.  {We
  refer to~\cite{Li88,lipshitz-df} for general results on D-finite series.}

For a series $F$ in several variables, we denote the derivative of $F$ with 
respect to the $i$th variable by $F'_i$.

\medskip
 In the next three subsections we state our main enumerative results, both exact and asymptotic.

%========================================================
\subsection{Quadrant tandem walks with prescribed endpoints}
\label{sec:main_result_start0b}
%========================================================
We give here an explicit expression for the generating function 
$Q^{(0,b)}(x,y)$ that counts tandem walks starting at height $b$ on the $y$-axis.
We define the \emph{step \gf} $S(x,y)$, {which counts all tandem steps}, as 
\beq\label{S:def}
S(x,y):=x\by+\sum_{r\geq 0}z_r\sum_{i=0}^r\bx^{r-i}y^i,
\eeq
and we let $K(x,y):=1-tS(x,y)$. In the $p$-specialization, $S(x,y)$ is
a (Laurent) polynomial.

We let $Y_1\equiv Y_1(x)$ be the unique power series in $t$ satisfying
$K(x,Y_1)=0$, that is,
\beq \label{eq:Y}
Y_1=t\Bigl(x+Y_1\sum_{r\geq 0}z_r\sum_{i=0}^r\bx^{r-i}Y_1^{i}\Bigr).
\eeq
This series has coefficients in $\qs[x, \bx, z_0, z_1, \ldots]$, and starts by
\[
  Y_1=tx + t^2x \sum_{r\geq 0}z_r\bx^{r}+ O(t^3).
\]
In the $p$-specialization, this series is algebraic. We observe that $H(x):=\frac{Y_1(x)}{tx}$ is the generating function of tandem walks starting at the origin, 
ending on the $x$-axis and \emph{staying in the upper half-plane $\{y\geq 0\}$}, 
where as usual $t$ marks the length, $x$ marks the final abscissa
and~$z_r$ marks the number of {face} steps of level $r$. Indeed, upon
considering the first step, say $(-r+i, i)$, of such a walk, and the
first time it comes back to the $x$-axis, it is
standard~\cite[Ch.~11]{lothaire1983combinatorics} to establish
\[
H= 1+ \sum_{r\ge0} \sum_{i=0}^r  (t z_r \bx ^{r-i}) (tx)^i H^{i+1},
\]
which is equivalent to~\eqref{eq:Y} with $txH=Y_1$.

\begin{Theorem}\label{thm:Q0}
Let $Y_1\equiv Y_1(x)$ and $K(x,y)$ be defined as above. The generating function
$Q^{(0,b)}(x,y)$ can be expressed as the nonnegative part in $x$ of
an explicit series\footnote{Throughout the paper, we 
 use the notation $u^a+\cdots+v^a$ for $\sum_{k=0}^a u^{k}v^{a-k}$.}:
\beq\label{eq:Qby}
Q^{(0,b)}(x,y)=[x^\geq ] \frac{-Y_1}{yK(x,y)}(Y_1^b+\cdots+\bx^b)\Bigl(1-\frac1{tx^2}+\sum_{r\geq 0}z_r(r+1)\bx^{r+2}\Bigr),
\eeq
where the argument of $[x^{\geq}]$ is expanded as an element of $\qs[x, \bx, z_0, z_1, \ldots]((t))[[y]]$.
In particular, the \gf\ of bipolar orientations of  left %outer
boundary length $b+1$ is $txQ^{(0,b)}(x,0)$, where
\beq\label{eq:Qb0}
Q^{(0,b)}(x,0)=[x^\geq ] \frac{Y_1}{tx}(Y_1^b+\cdots+\bx^b)\Bigl(1-\frac1{tx^2}+\sum_{r\geq 0}z_r(r+1)\bx^{r+2}\Bigr).
\eeq
In the $p$-specialization, these series are D-finite in all their variables.
\end{Theorem}

We will provide two different proofs of~\eqref{eq:Qb0} (which easily implies~\eqref{eq:Qby} as explained in Remark~2 below): first in
Sections~\ref{sec:eqfunc} and~\ref{sec:Qxy-simple} using the 
 method developed in~\cite{BoBoMe18} for quadrant walks with large steps,
 and then in Section~\ref{sec:bij_proofs_0b} using the KMSW bijection and local 
operations on  marked bipolar orientations. This second approach  explains
 combinatorially why the enumeration of tandem walks in the quadrant
is related to
 the enumeration of {tandem} walks in the upper half-plane, that is, to the
 series~$Y_1(x)$.

\vspace{.2cm} 

\begin{remarks}
{\bf 1.}
We will give another D-finite expression for $Q^{(0,b)}(x,y)$,
and more generally of $Q^{(a,b)}(x,y)$,
in Section~\ref{sec:Qab} (Propositions~\ref{prop:DF} and~\ref{prop:Qab}),   
again as the positive part of an algebraic generating function. In
this alternative expression, the expansion has to be
done (more classically)  in~$t$ first. 

 \medskip
\noindent
{\bf 2.} Since $yK(x,y)$ is a formal power series in $y$,
with constant term $(-tx)$, the expression~\eqref{eq:Qb0} is clearly
the special case $y=0$ of~\eqref{eq:Qby}. Conversely, a simple
argument involving factorizations of walks allows us to
derive~\eqref{eq:Qby} from~\eqref{eq:Qb0}. Indeed,  upon expanding the
right-hand side of~\eqref{eq:Qby} in $y$, what we want to prove is
that, for all $d\geq 0$,
\beq\label{eq:equivQby}
Q^b_d(x):=[y^d]Q^{(0,b)}(x,y)=[x^\geq ] \frac{Y_1}{tx}(Y_1^b+\cdots+\bx^b)\Big(1-\frac1{tx^2}+\sum_{r\geq 0}z_r(r+1)\bx^{r+2}\Big)P_d,
\eeq
where 
\[
P_d=[y^d]\frac{-tx}{yK(x,y)}=[y^d]\frac1{1-y(\bx/t-\sum_{r\geq 0}z_r\sum_{i=0}^r\bx^{r-i+1}y^i)}.
\]
Note that $P_d$ is a polynomial in $1/t$, $\bx$, and the $z_r$'s,
which can alternatively be described 
by the recurrence relation
\beq\label{P-rec}
P_0=1,\qquad  P_{d+1}=\frac{\bx}{t}P_d-\sum_{r\geq 0}z_r\sum_{i=0}^r\bx^{r-i+1}P_{d-i}\ \ \mathrm{for\ all}\ d\geq 0,
\eeq
where $P_d=0$ for $d<0$. We now prove~\eqref{eq:equivQby} by induction
on $d\ge 0$.  The case $d=0$  is
precisely~\eqref{eq:Qb0}. Assume that~\eqref{eq:equivQby} holds for $Q^b_0,
\ldots, Q^b_d$, and let us prove it for $Q^b_{d+1}$. A last step
decomposition of quadrant {tandem} walks ending at height $d$ gives
\[
Q^b_d = \mathbbm{1}_{d=b} + txQ^b_{d+1}+t [x^\geq] \sum_{r\geq
  0}z_r\sum_{i=0}^r  \bx^{r-i}Q^b_{d-i}, 
\]
with $Q^b_d=0$ for $d<0$. Extracting $Q^b_{d+1}$, and observing that
$[x^\geq](\bx[x^{\geq}] G(x))=[x^{\geq}](\bx G(x))$,  we see that
\[
Q^b_{d+1}=[x^{\geq}]\Big(\frac{\bx}{t}Q^b_d-\sum_{r\geq
  0}z_r\sum_{i=0}^r\bx^{r-i+1}Q^b_{d-i}\Big) .
\]
We now use the induction hypothesis~\eqref{eq:equivQby} to replace $Q^b_0, \ldots, Q^b_d$ by their respective expressions in
terms of $P_0, \ldots, P_d$. Observe that for $e\geq 0$ we have  
$[x^\geq](\bx^e[x^{\geq}] G(x))
=[x^{\geq}](\bx^e G(x))$. Finally we use the recurrence
relation~\eqref{P-rec} to conclude that~\eqref{eq:equivQby} holds for $Q^b_{d+1}$. 
\end{remarks}

It is well known that algebraic series, in particular  $Y_1$ and its powers, can
be expressed as constant terms of rational functions. Hence we can also express
$Q^{(0,b)}(x,y)$ in terms of a  rational function,
this time in three variables $x, y, z$. The following result is proved
in Section~\ref{sec:Qxy-simple}.

\begin{Corollary}\label{cor:Qthird}
As above, let $S(x,y)$ be defined
by~\eqref{S:def}, and let $K(x,y)=1-tS(x,y)$.
The series $Q^{(0,b)}(x,y)$ can alternatively be expressed as
\[
Q^{(0,b)}(x,y)=[x^{\geq }][z^0]\frac{tz^2}{y} \frac{S'_2(x,z)}{K(x,y)K(x,z)}(z^b+\cdots+\bx^b)\Big(1-\frac{\bx^2}{t}+\sum_{r\geq
  0}z_r(r+1)\bx^{r+2}\Big),
\]
where the argument of $[x^{\geq}][z^0]$ is expanded as a series in $\qs[x, \bx, z, \bz, z_0, z_1, \ldots]((t))[[y]]$.

In particular, 
\beq\label{eq:Qx0-rat}
Q^{(0,b)}(x,0)=-[x^{\geq}][z^0]\frac{z^2}{x}\frac{S'_2(x,z)}{K(x,z)}(z^b+\cdots+\bx^b)\Big(1-\frac{\bx^2}{t}+\sum_{r\geq 0}z_r(r+1)\bx^{r+2}\Big).
\eeq
\end{Corollary}

This result, specialized to $p$-tandem walks, yields  an expression for
$[x^c]Q^{(0,b)}(x,0)$ as the constant term in $x$ and $z$ of a rational expression of
$t$, $x$, $z$ and the $z_r$'s. 
From expressions of this form,  recent algorithms based on ``creative
telescoping'' can efficiently construct polynomial recurrences satisfied
by the
coefficients~\cite{bostan-lairez-salvy-issac,bostan-lairez-salvy,lairez}. For
instance, let us specialize~\eqref{eq:Qx0-rat} to $x=0$, $b=0$, $z_p=1$ and $z_r=0$ if $r\not =p$. We obtain 
\beq\label{Q0000}
Q^{(0,0)}(0,0)=-[x^0][z^0]\frac{z^2}{x}\frac{S'_2(x,z)}{K(x,z)}\Big(1-\frac{\bx^2}{t}+(p+1)\bx^{p+2}\Big).
\eeq
By Corollary~\ref{cor:bij-bip}, the series $tQ^{(0,0)}(0,0)$ counts (by edges) bipolar orientations of outer degree~$2$ with all inner faces
of degree $p+2$. By Lemma~\ref{lem:per}, such  orientations have $n+1$
edges, where $(p+2)$ divides $2n$. By counting adjacencies between edges and faces, it is easy to see  that these orientations have $\frac{2n}{p+2}$ inner faces.  If $p$ is odd, this number is necessarily even.
Retaining only non-zero coefficients in $Q^{(0,0)}(0,0)$, we write
\[
Q^{(0,0)}(0,0) = \sum_{k \ge 0} a(k) t^{c k(p+2)/2}, 
\]
where $c=2$ if $p$ is odd, and $c=1$ otherwise. In this way, $a(k)$
counts orientations with $ck$ 
inner  faces. In particular, when $p=3$ and $p=4$, we recover from~\eqref{Q0000} the  expressions~\eqref{tri} and~\eqref{quad} given in the
introduction. One can also derive the following recurrence relations, 
which were computed for
us by Pierre Lairez (in all cases,  $a(0)=1$), from the above expression.
\begin{itemize}
\item
For $p=1$ (triangulations):
\[
(k+3)(k+2)a(k+1)=3(3k+2)(3k+1)a(k).
\]
This gives the number of bipolar
triangulations with outer degree 2 and $2k$ inner 
 faces (equivalently, $k+2$ vertices) as 
\[
a(k)= \frac {2(3k)!}{k!\,(k+1)!\,(k+2)!},
\]
which is Tutte's result~\eqref{Tutte-bip}.

\item
For $p=2$ (quadrangulations), $a(k)$ gives the number of bipolar
orientations of a quadrangulated digon with $k$ inner faces (denoted
  $c(k)$ in the introduction), and
\[
(k+4)(k+3)^2a(k+2)=4(2k+3)(k+3)(k+1)a(k+1)+12(2k+3)(2k+1)(k+1)a(k),
\]
{as announced in the introduction.} Using Petkov\v{s}ek's algorithm~\cite{AB},
we proved that this recurrence relation has no hypergeometric solution. Still, one can derive from~\eqref{Q0000}, specialized to $p=2$, an expression for $a(k)$ as a single sum involving multinomial coefficients.

\item
For $p=3$ (pentagulations), $a(k)$ gives the number of bipolar
orientations of a pentagulated  digon with $2k$ inner faces, and
\begin{multline*}
27(3k\!+\!8)(3k\!+\!4)(5k\!+\!3)(3k\!+\!5)^2(3k\!+\!7)^2(k\!+\!2)^2a(k\!+\!2)=\\
60(5k\!+\!7)(3k\!+\!5)(5k\!+\!9)(5k\!+\!6)(3k\!+\!4)(8\!+\!5k)(145k^3\!+\!532k^2\!+\!626k\!+\!233)a(k\!+\!1)\\
-800(5k+6)(5k+1)(5k+7)(5k+2)(5k+3)(5k+9)(5k+4)(8+5k)^2 a(k).
\end{multline*}
Again, there is no hypergeometric solution.
\end{itemize}
Starting from~\eqref{eq:Qx0-rat}, similar constructions can be performed for a prescribed starting point $(0,b)$
and a prescribed endpoint $(c,0)$, in order to count bipolar
orientations of signature $(0,b+1; c+1, 0)$.

%================================================================
\subsection{Quadrant tandem walks ending anywhere}
\label{sec:anywhere}
%=============================================================

We now consider  the specialization $Q^{(a,b)}(1,1)$, which counts
tandem walks in the quadrant starting at $(a,b)$, and 
 records the length (variable $t$),  the number
of face steps of each level $r$ (variable $z_r$), but not the
coordinates of the endpoint. For this problem, we can either consider
$Q^{(a,b)}(1,1)$ as a series in infinitely many variables $t, z_0,
z_1, \ldots$, or apply the $p$-specialization and count $p$-tandem
walks only.

Let $W$ be the unique formal power series in $t$ satisfying 
\beq\label{W-eq}
W=t\,\Big(1+\sum_{r\ge0}z_r\left(W+\cdots+W^{r+1}\right)\Big).
\eeq
Note that $W=Y_1(1)$, where $Y_1\equiv Y_1(x)$ is given by~\eqref{eq:Y}.
In the $p$-specialization,  this series  is algebraic.

\begin{Theorem}\label{thm:alg} 
For $a,b\geq 0$, {the \gf\ of quadrant tandem walks starting at $(a,b)$
  and ending anywhere in the quadrant is}
\[
Q^{(a,b)}(1,1)=\frac{W}{t}\cdot\sum_{i=0}^a\A_i\cdot\sum_{j=0}^bW^j,
\]
where $\A_i$ is a  series in $W$ and the $z_r$'s:
\begin{align}
\A_i&= [u^i]\,\frac 1 W\,  \frac{\bu-1}{S(\bu,W)-S(1,W)} \nonumber\\
    &=[u^i]\,\frac1
      {1-uW \sum_{i,k\ge 0} u^i W^k \sum_{r>i+k}z_r},\label{Ai-def}
 \end{align}
with $S(x,y)$ given by~\eqref{S:def}. In particular $Q^{(0,0)}(1,1)=W/t$.

In the $p$-specialization, each $A_i$, and thus the whole series $tQ^{(a,b)}(1,1)$, becomes a polynomial in $W$ and $z_0, \ldots, z_p$.
\end{Theorem} 

We will provide a first proof in Section~\ref{sec:algeq} using functional equations and algebraic manipulations. 
A bijective proof will then be given in Section~\ref{sec:prop_alg_combi}. It involves  
the KMSW  bijection and the involution $\sigma$, both described in
Section~\ref{sec:KMSW}.    

\medskip

\begin{remarks}
{\bf 1.} In our combinatorial proof, the term
$\frac{W}{t}\A_iW^j$ will be interpreted as the generating
function of  tandem walks that start at $(0,j)$, remain
in the upper half-plane $\{y\geq 0\}$, touch the $x$-axis at least
once and end on the line $\{y=i\}$ (see Lemma~\ref{lem:interpretation}).
In  particular, for $a=b=0$,  this  proof gives a length preserving bijection 
between tandem walks in the quadrant that start at the origin,
and tandem walks in the upper half-plane that start at the origin
and end on the $x$-axis. Moreover, this bijection 
preserves the number of SE steps.

In the case where
$ z_p=1$ and $z_r=0$ for $r\not = p$, three such bijections already appear in the literature. The first two are only valid for $p=1$:
 one is due to Gouyou-Beauchamps~\cite{gouyou-tableaux},
and uses a simple correspondence between  1-tandem walks
and standard Young tableaux with at most $3$ rows, and then the Robinson--Schensted correspondence; the second,  more recent one is due to Eu~\cite{eu10} (generalized in~\cite{eu13} to Young tableaux with at most $k$ rows). 
The third  bijection, due to Chyzak and Yeats~\cite{chyzak-yeats}, is very recent and holds for any $p$. It relies on certain automata rules to 
build (step by step) a half-plane walk ending on the $x$-axis from a quarter plane excursion.
These three constructions do not seem to be equivalent to the
correspondence presented in Section~\ref{sec:prop_alg_combi}. 

\medskip

\noindent{\bf 2.}  Let us define \emm double-tandem walks, as walks
with steps $N,W,SE,E,S,NW$: these are  the three steps involved in
1-tandem walks, and their reverses. With these steps too, it is known that
walks in the quadrant that start at the origin are
equinumerous with  walks (of the same length) in the upper half-plane that start at the origin and end on the  $x$-axis~\cite[Prop.~10]{mbm-mishna}; see~\cite{mortimer} for an intriguing refinement involving walks confined to a triangle.  A bijection between these two families of walks was recently given 
 by Yeats~\cite{yeats2014bijection}, and then reformulated using automata 
 in~\cite{chyzak-yeats}. We do not know of any bijection for these walks that would generalize the KMSW map,
   but we conjecture that there exists an involution  on double-tandem walks having the same properties as the involution $\sigma$ of Section~\ref{sec:KMSW} (once defined on 1-tandem walks).
    See the remark at the end of Section~\ref{sec:prop_alg_combi} for details. 
\end{remarks}

%======================================
   \subsection{Asymptotic number of quadrant
     % {tandem}
     walks with prescribed  endpoints}\label{sec:main_asymp}
%=========================================
We now fix $p\ge 1$, and focus on the asymptotic enumeration of 
 $p$-tandem walks with prescribed endpoints confined to the quadrant. To be precise, we aim at finding an asymptotic estimate of
 the coefficients $[t^n]Q_{c,d}^{a,b}(t,z_0, \ldots, z_p)$ as $n\to\infty$, for any prescribed $a,b,c,d$ and nonnegative weights $z_0,\ldots,z_p$ with $z_p>0$. 
As it turns out, a detailed estimate can be derived by combining
 recent asymptotic results by Denisov and Wachtel~\cite{denisov2015random} (or
rather, a variant of these results that applies to our \emm periodic,
walks) and the algebraic expression for $Q^{(a,b)}(1,1)$ given in
Theorem~\ref{thm:alg}.

Let 
\beq\label{def:iota}
D:=\{r\in\llbracket 0,p\rrbracket,\ z_r>0\},\qquad \hbox{ and } \qquad \iota:=\mathrm{gcd}(r+2,r\in D).
\eeq
 It follows from Lemma~\ref{lem:per} that there can only exist a walk
 of length $n$ from $(a,b)$ to $(c,d)$ if $ c-d\equiv a-b +2n$~mod~$\iota$ 
(and $n$ is large enough).
Our main asymptotic result is the following.

\begin{Theorem}\label{thm:asympt}
Fix $p\ge 1$. Let $a,b,c,d$ be nonnegative integers and let $z_0,\ldots,z_p$ be
nonnegative weights with $z_p>0$.
Let $\iota$ be defined by~\eqref{def:iota}.  
Then, as $n\to\infty$ conditioned on $c-d\equiv a-b +2n$~{\em mod}~$\iota$, we have 
\[
[t^n]Q_{c,d}^{a,b}\sim \kappa\ \!\gamma^nn^{-4},
\] 
where the growth rate $\gamma$ is explicit and depends only on  the
weights $z_r$, while the multiplicative constant $\kappa$, also
explicit, depends on these weights and on $a,b,c,d$ as well.
\end{Theorem}
The explicit values of $\kappa$ and $\gamma$ are  given in
Section~\ref{sec:weighted}, together with the proof of the theorem.
When specialized to bipolar orientations ($a=d=0$), this theorem
will give (once the constants are explicit) the following detailed asymptotic estimate for the number of
orientations with prescribed face degrees and boundary lengths.

\begin{Corollary}[{\sc Bipolar orientations with prescribed  face degrees}]\label{coro:bipolar}
Let $\Omega\subset\{2,3,4,\ldots\}$ be a finite set such that $\max (\Omega) \ge 3$, 
and let 
$\iota$ be the gcd of all elements in $\Omega$.   
Let $\alpha$ be the unique positive solution of the equation
\[
1=\sum_{s\in\Omega}\binom{s-1}{2}\alpha^{-s},
\]
 and let 
\[
\gamma=\sum_{s\in\Omega}\binom{s}{2}\alpha^{-s+2}.
\]
Then, for $2n\equiv b+c$~{\em mod}~$\iota$, the number $B_n^{(\Omega)}(b,c)$
of bipolar orientations with $n+1$ edges, left %outer
 boundary length $b+1$, right %outer
 boundary length $c+1$, and all
inner face degrees in $\Omega$,  satisfies 
\[
B_n^{(\Omega)}(b,c)\sim \kappa {\gamma^{n}}{n^{-4}}
\qquad \hbox{ as } n\to\infty,
\]  
where the constant $\kappa$ is 
\[
\kappa:=\frac{\iota \gamma^2}{4\sqrt{3}\pi\alpha^4 \sigma^4}(b+1)(b+2)(c+1)(c+2)\alpha^{-b-c}, 
\]
with  
$\sigma^2=\frac{\alpha^2}{\gamma}\sum_{s\in\Omega}\binom{s}{3}\alpha^{-s}.$
\end{Corollary}
(The above quantity $\sigma$, which will be seen as a variance in Section~\ref{sec:asymptotic}, should not be mixed up with the transformation $\sigma$ of Section~\ref{sec:KMSW}. We hope that this will not cause any inconvenience.)
The proof is given in Section~\ref{sec:weighted}. Specializing further to the case of bipolar $d$-angulations
($\Omega=\{d\}$, with $d\ge 3$),
 we have
\[
\iota=d,\qquad \alpha=\binom{d-1}{2}^{1/d}, \qquad
\gamma=\frac{d}{d-2}\binom{d-1}{2}^{2/d},
\]
 so that  $\sigma^2=({d-2})/3$.
Hence the number of bipolar orientations having $n+1$ edges (for
$2n-b-c$ divisible by $d$),  left (respectively right)
boundary length $b+1$ (respectively $c+1$), satisfies
\[
B_n^{(d)}(b,c)\sim \frac {9(b+1)(b+2)(c+1)(c+2)} {4\sqrt 3 \pi
  d}\left( \frac d{d-2}\right)^{n+4} {\binom {d-1}{ 2}}^f n^{-4},
\]
where $f= (2n-b-c)/d$ is the number of
inner faces.  When $b=c=0$ and $d=3$ or $d=4$, this estimate is in
agreement with~\eqref{Tutte-bip} and~\eqref{asympt4}.

%%%%%%%%%%%%%%%%%%%%%%%%%%%%%%%%%%%%%%%%%%%%%%%%%%%%%%%%%%%%%%%
\section{A functional equation approach}
\label{sec:eqfunc}
%%%%%%%%%%%%%%%%%%%%%%%%%%%%%%%%%%%%%%%%%%%%%%%%%%%%%%%%%%%
Let $p\ge 1$. In this section, we apply  the general approach to 
quadrant walk enumeration described in~\cite{BoBoMe18} to $p$-tandem walks confined to the quadrant. This approach consists of four steps, detailed here in Sections~\ref{sec:ef} to~\ref{sec:extract}. Let us recall that it is not systematic, and does not  work for all sets of steps. Moreover, the fact that our step set depends on~$p$ adds another difficulty.  The first  step of the approach (write a functional equation) is however simple and systematic. The second (compute the so-called \emm orbit, of the problem) is also systematic as long as the orbit is finite, and can even be performed automatically for a fixed (small) value of $p$~\cite[Sec.~3.2]{BoBoMe18}. The last two steps (construction of a \emm section-free, equation and extraction of the \gf) definitely require more inventiveness.

  At the end, the approach yields an
 expression for $Q^{(a,b)}(x,y)$ as the nonnegative part (in $x$ and
 $y$) of an algebraic series. We state it in the first subsection below. This expression is \emm not,\/ the one of  Theorem~\ref{thm:Q0}, which will be derived later in  Section~\ref{sec:Qxy-simple}.

%===============================================================
 \subsection{First expression for $\boldsymbol{Q^{(a,b)}(x,y)}$}
 %===============================================================

 The algebraic ingredient in this  expression is a new
 series $x_1$, {involving the indeterminates $x, y$ and~$z_r$}, and is defined as follows.

\begin{Lemma}\label{lem:xi}
Recall the definition~\eqref{S:def} of $S(x,y)$.  The equation $S(x,y)=S(X,y)$,
when solved for $X$, admits $p+1$ roots $x_0=x, x_1, \ldots, x_p$,
which can be taken as Laurent  series in $\by:=1/y$ with coefficients in
$\cs[z_1,  \ldots, z_{p}, 1/z_p, x, \bx]$.
Exactly one of these roots, say $x_1$,
contains {some} positive powers of $y$ in its series expansion. It has
coefficients in $\qs[
%z_0,
z_1,  \ldots, z_{p}, x, \bx]$ and reads $x_1=z_p \bx
y^{p}(1+O(\by))$. The other roots are \fps\ in $\by$ with no constant term.
\end{Lemma}
\noindent{\bf Examples.}  For $p=1$ we have 
\[
S(x,y)=x\by + z_0+ z_1 (\bx+y),
\]
and the equation $S(X,y)=S(x,y)$ has two solutions, $x_0=x$ and
$x_1=z_1 \bx y$. 

For $p=2$ we have
\[
S(x,y)=x\by + z_0+ z_1 (\bx+y) + z_2(\bx^2+\bx y +y^2),
\]
and the equation $S(X,y)=S(x,y)$ has three solutions. One of them is $x_0=x$ and
the other two satisfy a quadratic equation:
\[
x^2X^2-yX(xz_1+z_2(1+xy))-xyz_2=0.
\]
Hence
\[
x_{1,2}=\frac{xz_1+z_2(1+xy)\pm \sqrt{(xz_1+z_2(1+xy))^2+4x^3\by z_2}}{2x^2\by}.
\]
(We take $x_1$ to correspond to the $+$ sign.) We expand both
solutions as Laurent series in $\by$ (not $y$!), and find:
\begin{align*}
  x_1&=z_2  \bx y^2+ \bx(z_1+\bx z_2) y + \by - (z_1/z_2+\bx) \by ^2+ O(\by^3),
\\
x_2&= \phantom{z_2  \bx y^2+ \bx(z_1+\bx z_2) y}- \by +(z_1/z_2+\bx)
     \by ^2+ O(\by^3). 
\end{align*} \qee

We prove  Lemma~\ref{lem:xi} in Section~\ref{sec:orbit}. It implies that $y_1:=\bx_1=1/x_1$ is a power series in~$\by$ whose coefficients
lie in $\qs[ z_1,\ldots, z_p, 1/z_p, x, \bx]$
(this comes from the
monomial form of the first coefficient of $x_1$). We can now give our
first expression for $Q^{(a,b)}(x,y)$, which we first state in the case $a=0$. 

\begin{Proposition}\label{prop:DF}
  Fix $p\ge 1$, and let $y_1=1/x_1$, where $x_1$ is defined in Lemma~\ref{lem:xi}.  The \gf\ of $p$-tandem walks confined to the first quadrant
  and starting at $(0,b)$ is the nonnegative part (in $x$ and $y$)
  of an algebraic function:
\[
Q^{(0,b)}(x,y)=[x^\ge y^\ge ]{\frac {\left(1-\bx \by \right) S'_1(x,y)
  }{1-tS(x,y)}}\sum_{k=0}^b\left(y^{k+1}- y_1^{k+1} \right)\bx^{b-k}
,
\]
where the argument of $[x^{\geq}y^{\geq}]$ is expanded as a series of\/ $\qs[x, \bx, z_0, \ldots, z_p, 1/z_p]((\by))[[t]]$.

In particular, $Q^{(0,b)}(x, y)$ is D-finite.
\end{Proposition}

The general  case is more involved.

\begin{Proposition}\label{prop:Qab}
 Fix $p\ge 1$, and let $y_1=1/x_1$, where $x_1$ is defined in Lemma~\ref{lem:xi}.  The \gf\ of $p$-tandem walks confined to the first quadrant
  and starting at $(a,b)$ is the nonnegative part (in $x$ and $y$)
  of an algebraic function:
\beq\label{ex:Qab}
Q^{(a,b)}(x,y)=[x^\ge y^\ge ]\frac{(1-\bx\by)S'_1(x,y)}{1-tS(x,y)}\big(\eta_b\zeta_a-\eta_{b-1}\zeta_{a-1}\big),
\eeq
where $\eta_b:=\sum_{k=0}^b(y^{k+1}-y_1^{k+1})\bx^{b-k}$ and $\zeta_a$ is the Laurent polynomial in $x$ and $y$ (with coefficients that are polynomial in $z_1,\ldots,z_p$) defined by 
\begin{align*}
 \sum_{a\ge 0} \zeta_au^a&=\frac1{uy(1-u\by)(S(\bu,y)-S(x,y))}
  \\
 & =\frac1{(1-ux)(1-u\by)(1-u \bx y\sum_{i+j+k<p} z_{i+j+k+1}u^i \bx^j y^k)}.
\end{align*}
The  argument of $[x^\ge y^\ge ]$ in~\eqref{ex:Qab} is meant
as a series of $\qs[x, \bx, z_0, \ldots, z_p, 1/z_p]((\by))[[t]]$.
 \end{Proposition}

Propositions~\ref{prop:DF} and~\ref{prop:Qab} will be proved in Sections~\ref{sec:extract} and~\ref{sec:Qab}, respectively.

%===================================================
\subsection{A functional equation}
\label{sec:ef}
%======================================================

The starting point of our approach is 
a functional equation that characterizes the series $Q(x,y):=Q^{(a,b)}(x,y)$, and
simply relies on a step-by-step construction of quadrant tandem walks. 
It  reads
\[ 
Q(x,y)=x^ay^b+ t S(x,y)Q(x,y) -tx\by Q(x,0)- t\sum_{r=1}^p z_r \sum_{i=1}^r \bx ^i
y^{r-i}\bigl(Q_0(y)+\cdots + x^{i-1} Q_{i-1}(y)\bigr),
\] 
where $S(x,y)$ is the step polynomial given by~\eqref{S:def},
 and $Q_i(y)$ counts quadrant tandem walks starting at $(a,b)$ and ending at abscissa~$i$. We call $Q(x,0)$ and the series $Q_i(y)$ \emm sections, of $Q(x,y)$.
Equivalently,
\beq\label{eqfunc}
K(x,y) Q(x,y) =x^ay^b- tx\by Q(x,0)- \sum_{j=1}^p \bx^j G_j(y),
\eeq
where $K(x,y)=1-tS(x,y)$ and 
\[
G_j(y)= t \sum_{r=j}^p z_r \bigl(Q_0(y) y^{r-j}+ Q_1(y) y^{r-j-1}+\cdots + Q_{r-j}(y) y^0\bigr).
\]

%===========================================================
\subsection{The orbit of $\boldsymbol p$-tandem walks}
\label{sec:orbit}
%===========================================================
The aim of this subsection is to prove the following result.
\begin{Proposition}\label{prop:inv}
Let $x_0, \ldots, x_p$ be defined as the roots of $S(X,y)=S(x,y)$, as in Lemma~\ref{lem:xi}. Let us write  $x_{p+1}=\by$. For
$0\le i \le p+1$, denote moreover $y_i=\bx_i:=1/x_i$, so that  in particular,
$y_{p+1}=y$. Then, for $0\le i, j \le p+1$ and $i\not = j$, we have 
\[
S(x_i,y_j)=S(x,y).\]
\end{Proposition}
In the terminology of~\cite{BoBoMe18}, the pairs $(x_i,y_j)$ with $i\not = j$
form the \emm orbit, of $(x,y)$ for the step set $\cS_p:=\{(1,-1)\}
  \cup \{(-i,j): i, j \ge 0, i+j \le p\}$.

\medskip

Our first task is to prove Lemma~\ref{lem:xi}.

\begin{proof}[Proof of Lemma~\ref{lem:xi}.]
Recall the expression~\eqref{S:def} of the step polynomial
$S(x,y)$.  We refer the reader to~\cite[Ch.~6]{stanley-vol2} for generalities on
algebraic series. Clearly one of the roots of $S(X,y)=S(x,y)$ (solved for $X$)
is $x_0=x$. The others satisfy
\beq\label{Xi-eq}
0=\frac{S(x,y)-S(X,y)}{x-X}=\by-\bx \bX\sum_{\substack{i,j,k\geq 0\\ i+j+k<p}}z_{i+j+k+1}\bx^i\bX^jy^k.
\eeq
This expression is a polynomial in $\bX$, and is thus well suited to
determine the
roots $x_i$ such that $\bx_i=1/x_i$ is a formal power
series in $\by$ (or in a positive power of $\by$). 
The other roots~$\bx_i$ will involve positive powers of
$y$, hence their reciprocals will be formal power series in (a positive power of) $\by$. More precisely,    upon multiplying the above
identity by $x\by^{p-1}$ {and expanding in powers of $\by$}, we have
\beq\label{eq-x1}
0=x\by^p -z_p\bX- \sum_{\ell =1}^{p-1}\by ^\ell \sum_{i+j \le \ell}
z_{i+j+p-\ell}\,\bx^i \bX^{j+1}.
\eeq
The coefficient of $\by^0$ is $-z_p \bX$. It has degree $1$ in $\bX$, hence there is a unique
$\bx_i$, say $\bx_1$,
that only involves nonnegative powers of $\by$. This series can be
computed iteratively from the above equation. It is a power series in
$\by$, starts by
\[
{\bx_1}= \frac{x\by^p}{z_p} + O(\by^{p+1}),
\]
and its coefficients belong to $\qs[z_1, \ldots, z_p, 1/z_p, x,
\bx]$. This proves the claimed properties of~$x_1$.

To understand the nature of the other roots $x_2, \ldots, x_p$, we now
write $X=\by U$, and multiply~\eqref{Xi-eq} by $\by^p U^p$. Then
\beq\label{U-eq}
0= \by^{p+1} U^p -{\bx} \sum_{i+j+k<p} z_{i+j+k+1} \bx^i U^{p-1-j}
\by^{p-1-j-k}.
\eeq
The coefficient of $\by^0$ is 
\[
- {\bx}z_p \sum_{k=0}^{p-1} U^k.
\]
It has degree $p-1$ in $U$, hence~\eqref{U-eq} admits $p-1$ solutions
$u_2, \ldots, u_p$ that
expand in nonnegative powers of $\by$ only. Their constant terms
are the $p$th roots of unity distinct from~$1$. All of them are power
series in $\by$, and their expansions can be computed {recursively}
using~\eqref{U-eq}. Their coefficients lie in $\cs[z_1, \ldots, z_p,
1/z_p, x, \bx]$ (in fact, in the lemma, we could replace  $\cs$ by the extension of $\qs$
generated by $p$th roots of unity). 
\end{proof}

We then need the following  symmetry properties of $S(x,y)$.
\begin{Lemma}\label{lem:symm}
  The step polynomial $S(x,y)$, defined by~\eqref{S:def}, satisfies $S(x,y)=S(\by,\bx)$ and
\[
x\, \frac{S(x,y)-S(X,y)}{x-X}= -\bX \, \frac{S(x,y)-S(x,\bX)}{y-\bX}.
\]
\end{Lemma}
\begin{proof}
  The first point is easy, using
\[
S(x,y)= x\by + \sum _r z_r \frac{\bx^{r+1}-y^{r+1}}{\bx-y}.
\]
For the second, we recall from~\eqref{Xi-eq} that
\beq\label{Xi-eq-bis}
\frac{S(x,y)-S(X,y)}{x-X}=\by -\bx\bX\sum_{\substack{i,j,k\geq 0\\
    i+j+k<p}}z_{i+j+k+1} \bx^i\bX^jy^k,
\eeq
and we compute from~\eqref{S:def} that
\beq\label{Sdiff}
\frac{S(x,y)-S(x,\bX)}{y-\bX}=-xX\by +\sum_{\substack{i,j,k\geq 0\\ i+j+k<p}}z_{i+j+k+1}\bx^i\bX^jy^k.
\eeq
The result follows by comparing these two expressions.
\end{proof}

We can now prove Proposition~\ref{prop:inv}.

\begin{proof}[Proof of Proposition~\ref{prop:inv}.] By
  definition of the series $x_i$, for $0\le i \le p$, we have
  $S(x,y)=S(x_i,y)$. So the claimed identity holds for $j=p+1$. The
  first identity in Lemma~\ref{lem:symm} then gives
\[
S(x_i,y)=S(\by, {\bx_i}),
\]
hence the claimed identity holds as well for $i=p+1$.

Now we specialize the second identity in Lemma~\ref{lem:symm} to $x=x_i$,
$X=x_j$, with $0\le i \not = j \le p$. This reads
\[
x_i \frac{S(x_i,y)-S(x_j,y)}{x_i-x_j}=-y_j
\frac{S(x_i,y)-S(x_i,y_j)}{y-y_j}.
\]
Since the left-hand side is zero, we conclude that
\[
S(x_i,y_j)=S(x_i,y)=S(x,y)
\]
for $0\le i \not = j \le p$, which concludes the proof of Proposition~\ref{prop:inv}.  
\end{proof}

%==============================================================
\subsection{A section-free functional equation}
\label{sec:sec-free}
%==============================================================
In the functional equation~\eqref{eqfunc}, we can replace the pair
$(x,y)$ by any element $(x_i,y_j)$ of the orbit. The series that occur
in the resulting equation are series in $t$ with algebraic
coefficients in $x$ and $y$ (and the $z_r$'s). By Proposition~\ref{prop:inv}, the kernel $K(x,y)=1-tS(x,y)$ takes
the same value at all points of the orbit. We thus obtain
$(p+1)(p+2)$ equations. Our aim is to form a linear combination of
these equations 
in which the right-hand side does not contain any section
$Q(x_i,0)$ nor $G_k(y_j)$. As soon as $p>1$, the vector space of such linear
combinations has dimension larger than 1. We choose here a
section-free combination that only involves the pairs $(x_i,y_j)$ for
$j\in \{0, 1, p+1\}$. That is, $y_j$ will be either $y$, or $\bx$, or
$y_1:=1/x_1$. We focus on the case $a=0$ until Section~\ref{sec:Qab}.

\begin{Lemma}\label{lem:section-free}
Let $x_0, \ldots, x_{p}$ be defined by Lemma~\ref{lem:xi}, and take
$x_{p+1}=\by$ as before. For $ Q(x,y)\equiv Q^{(0,b)}(x,y)$,  the following identity holds:
\begin{multline*}
   Q(x,y)  
%+(-1)^p 
- \bx \sum_{i=1}^p \left({x_i^p \, Q(x_i,y)}
\prod_{j \not = 0, i, p+1}\frac{1-\bx x_j}{x_i-x_j}\right)
\\
- y_1 \by   \prod_{i=2}^{p+1}(1-\bx x_i) \sum_{i\not = 1}
\frac{x_i^p\, Q(x_i,y_1)}{\prod_{j \not = i, 1}(x_i-x_j)}
\\
+\bx^{2} \by (x_1-\by) \prod_{i=2}^{p}(1-\bx x_i) \sum_{i\not = 0}
\frac{x_i^p\, Q(x_i,\bx)}{\prod_{j \not = i, 0}(x_i-x_j)}
=
{\frac { \left(1-\bx \by \right) S'_1(x,y)
  }{1-tS(x,y)}} \sum_{k=0}^b (y^{k+1}-y_1^ {k+1})  \bx^{b-k}.
%-\frac{   (1-xy)(1-\bx_1\by)(1-x_1\bx)\prod_{j=2}^p(x-x_j)}{x^p y K(x,y)}.
\end{multline*}
\end{Lemma}

In order to prove this lemma,
we  need two other lemmas that involve classical symmetric functions. We
 recall the definition of complete  and elementary homogeneous symmetric
 functions of degree $k$ in 
 {$m$} variables $u_1, \ldots, u_m$: 
\[
h_k(u_1, \ldots, u_m) =\sum_{1\le i_1\le \cdots \le i_k \le m}u_{i_1}
\cdots u_{i_k}, \qquad
e_k(u_1, \ldots, u_m) =\sum_{1\le i_1< \cdots <i_k \le m}u_{i_1}
\cdots u_{i_k}.
\]
In this subsection, we only apply the   following lemma  to polynomials 
  $P(u,v_1, \ldots, v_m)\in \qs[u]$, but we
   use it in full generality in the next subsection. This lemma extends Lemma~13 in~\cite{mbm-chapuy-preville}.
\begin{Lemma}\label{lem:Lagrange}
Let $P(u, v_1, \ldots, v_m)  \in \qs[u, v_1,
\ldots, v_m]$ be a polynomial, symmetric in the $v_i$'s. 
Take $m+1$  variables  $u_0, u_1, \dots , u_m$,
and define
\beq\label{CL}
E(u_0, \ldots, u_m):= \sum_{i=0}^m \frac{P(u_i, u_0, \ldots, u_{i-1},
  u_{i+1}, \ldots, u_m)}{\prod_{j\neq i}{(u_i-u_j)}}.
\eeq
Then $E(u_0, \ldots, u_m)$ is a symmetric polynomial in $u_0, \ldots, u_m$, of degree at most
$\deg(P)-m$. In particular, $E(u_0, \ldots, u_m)=0$ 
if $P$ has degree less than $m$.

If  $P(u, v_1, \ldots,v_m)=u^{m+a}$ with $a\geq 0$, then
\beq\label{eq:ha}
E(u_0, \ldots, u_m)=\sum_{i=0}^m \frac{{u_i}^{m+a}}{\prod_{j\neq i}{(u_i-u_j)}} = h_a(u_0,\ldots,u_m),
\eeq
with $h_a$ the complete homogeneous symmetric function. 

Finally, for $a\geq 0$ we have
\[ 
\sum_{i=0}^m \frac{{u_i}^{-a-1}}{\prod_{j\neq i}{(u_i-u_j)}} = \frac{(-1)^m}{\prod_{i=0}^mu_i}h_{a}(1/u_0,\ldots,1/u_m).
\] 
\end{Lemma}
\begin{proof}
  Let $E(u)$ denote the expression~\eqref{CL}, where we use
the shorthand notation $u$ for the $(m+1)$-tuple $(u_0,\ldots,u_m)$.  
Multiplication of $E(u)$ by the   Vandermonde determinant
\[
  \Delta(u):= \prod_{0\le i <j \le m}(u_i-u_j)
\]
gives a  polynomial in the $u_i$'s
  which is antisymmetric in the $u_i$'s (that is, swapping $u_i$ and
  $u_j$ changes the sign of the expression): this comes from the fact
  that $E$ is symmetric, while $\Delta$ is antisymmetric. Hence $E(u)\Delta(u)$,
  as a polynomial,
  must be divisible by the Vandermonde determinant, and $E(u)$ itself is a polynomial. Its degree is obviously $\deg(P)-m$ (at most).

Next, in order to prove~\eqref{eq:ha}, note that $h_a$  
is the Schur function of the Ferrers diagram consisting
of a single row of length $a$, and therefore, by definition of Schur functions~\cite[Ch.~4]{sagan-book}, we have
\[
  h_a(u_0,\ldots,u_m)=\frac{\det\left(u_i^{a\cdot\delta_{j,0}+m-j}\right)_{0\leq i,j\leq m}}{\Delta(u_0,\ldots,u_m)}.
\]
Upon expanding the determinant according to the first column ($j=0$), we find 
\[
h_a(u)\Delta(u)=\det\left(u_i^{a\cdot\delta_{j,0}+m-j}\right)_{0\leq i,j\leq m}
=\sum_{i=0}^m(-1)^i u_i^{a+m}\Delta(u_0,\ldots,u_{i-1},u_{i+1},\ldots,u_m).
\]
But for any $0\leq i\leq m$ we have  
\[\Delta(u_0,\ldots,u_{i-1},u_{i+1},\ldots,u_m)=\frac{\Delta(u_0,\ldots,u_m)}{(-1)^i\prod_{j\neq i}(u_i-u_j)},\]
which yields~\eqref{eq:ha}. 

To prove the last statement, we let $v_i=1/u_i$ and note that
$\frac{1}{u_i-u_j}=-\frac{v_iv_j}{v_i-v_j}$, hence
\begin{align*}
\sum_{i=0}^mu_i^{-a-1}\prod_{j\neq i}\frac{1}{u_i-u_j}&=(-1)^m(v_0\cdots v_m)\sum_{i=0}^mv_i^{m+a}\prod_{j\neq i}\frac{1}{v_i-v_j}\\
&=(-1)^m(v_0\cdots v_m)h_{a}(v_0,\ldots,v_m) \qquad\qquad \hbox{ by \eqref{eq:ha}}.
\end{align*}
\end{proof}

\begin{Lemma}\label{lem:sym-ei}
 Let $x_1, \ldots, x_p$ be the series defined in Lemma~\ref{lem:xi}. Their elementary symmetric functions are:
\[ 
 e_\ell(x_1,\ldots,x_p)=
 \begin{cases}
   1, & \hbox{if } \ell=0,\\
 \displaystyle  (-1)^{\ell-1}\bx y
   \sum_{\substack{i, k \ge 0\\ i+k\le p-\ell}} z_{i+k+\ell} \,\bx^i y^k,
   & \hbox{for }  1\le \ell \le p.
 \end{cases}
\] 
In particular, they are $x$-nonpositive  and $y$-nonnegative (meaning that, in every monomial that they contain, $x$ has a nonpositive exponent and $y$ a nonnegative one). Moreover, every monomial $\bx^i y^j$ occurring in them satisfies $i\ge j/p$.

Finally, 
\beq\label{S-diff}
\prod_{i=1}^p (1-\bx x_i)=yS'_1(x,y).
\eeq
\end{Lemma}

\begin{proof}
It follows from~\eqref{Xi-eq-bis} that $(S(x,y)-S(X,y))/(x-X)$ is a polynomial in $\bX$ with constant term $\by$. Hence
\[
\frac{S(x,y)-S(X,y)}{x-X}= \by
\prod_{i=1}^p(1-x_i\bX)=\by\sum_{\ell=0}^p(-1)^\ell e_\ell(x_1,\ldots,x_p)\bX^{\ell}.
\]
Comparison with~\eqref{Xi-eq-bis} gives the expression for $e_\ell(x_1,\ldots,
x_p)$. The next statement is then obvious.
 Letting $X$ tend to $x$ in the
above identity finally gives~\eqref{S-diff}.
\end{proof}

\begin{proof}[Proof of Lemma~\ref{lem:section-free}.]
As already noted, in the basic functional
equation~\eqref{eqfunc} we can replace the pair $(x,y)$ by any element $(x_i,y_j)$ of its
orbit. By Proposition~\ref{prop:inv}, this does not change the value of
$K(x,y)=1-tS(x,y)$. Recall that for the moment, we take $a=0$.

Using the elements
$(x_i,y)$ of the orbit, with $0\le i \le p$, we first construct  a linear
combination avoiding all series $G_j(y)$:
\beq\label{pos}
K(x,y) \sum_{i=0}^p \frac{x_i^p \, Q(x_i,y)}{\prod_{j \not = i, p+1}(x_i-x_j)}
=
y^b- t\by \sum_{i=0}^p \frac{x_i^{p+1} \, Q(x_i,0)}{\prod_{j \not = i,
    p+1}(x_i-x_j)}.
\eeq
Lemma~\ref{lem:Lagrange} explains the simplicity of the right-hand
side, namely the fact that all sections $G_j(y)$ disappear, and that the constant term is just $y^b$. In the former case the lemma is applied with $P(u)=u^{p-j}$ and $m=p$, in the latter case with $P(u)=u^p$ and $m=p$ again.

More generally, for $0\le k\le p+1$, we can similarly eliminate all series $G_j(y_k)$ in
 the equations obtained from the elements $(x_i,y_k)$ of the
orbit, for $i\not = k$. After multiplying by~$y_k$, this gives
\[ 
y_k K(x,y) \sum_{i\not = k} \frac{x_i^p\, Q(x_i,y_k)}{\prod_{j \not = i, k}(x_i-x_j)}
=
y_k^{b+1} - t \sum_{i\not = k } \frac{x_i^{p+1} \, Q(x_i,0)}{\prod_{j \not = i,
    k}(x_i-x_j)},
\] 
where the natural range of the indices $i$ and $j$ is $0, \ldots, p+1$.
Note that the equation can be rewritten as
\beq\label{eq:ykK-bis}
y_k K(x,y) \sum_{i\not = k} \frac{x_i^p\, Q(x_i,y_k)}{\prod_{j \not = i, k}(x_i-x_j)}
=
y_k^{b+1} - t \sum_{i=0}^{p+1} \frac{x_i^{p+1} \, Q(x_i,0)}{\prod_{j \not = i}(x_i-x_j)}(x_i-x_k).
\eeq
We now take an appropriate linear combination of three of these $p+2$
equations, namely those obtained for $k=p+1$, $k=1$ and $k=0$, with respective weights
\beq\label{weights}
x_0-x_1, \qquad x_{p+1} -x_0, \qquad x_1-x_{p+1}.
\eeq
(Of course these weights can be written in a simpler way as $x-x_1,  \by -x$ and  $x_1-\by$, but the above notation makes the symmetry clearer.)
Then, writing the {three} equations as in~\eqref{eq:ykK-bis}, it is easy to show that  all terms involving $Q(x_i,0)$ vanish from
the right-hand side.
Hence the only remaining  term on the right-hand side is
\begin{multline*}
  (x_0-x_1)y_{p+1}^{b+1}+(x_{p+1}-x_0)y_1^{b+1}+(x_1-x_{p+1})y_0^{b+1}
  \\=x_0x_1x_{p+1}\left( (y_1-y_0) y_{p+1}^{b+2} +(y_0-y_{p+1})y_1^{b+2} +(y_{p+1}-y_1) y_0^{b+2}\right).
\end{multline*}
With the notation~\eqref{CL} and $P(u, v_1, v_2)= u^{b+2}$, this can be rewritten as
\[   -x_0x_1x_{p+1}\Delta(y_0,y_1,y_{p+1}) E(y_0, y_1, y_{p+1}) \hskip 74mm \
  \]
\begin{align}  \hskip 51mm \  &= -x_0x_1x_{p+1}\Delta(y_0,y_1,y_{p+1}) h_b( y_0,y_1, y_{p+1})
  \qquad \hbox{ by \eqref{eq:ha}} \nonumber
 \\   
  &= x_0x_1x_{p+1}(y_0-y_1)(y_0-y_{p+1}) \sum_{k=0}^b(y_{p+1}^{k+1}-y_1^{k+1})y_0^{b-k}\nonumber
  \\
  &= (1-\bx\by)(x-x_1)\sum_{k=0}^b(y^{k+1}-y_1^{k+1})\bx^{b-k}.\label {final}
\end{align}
The left-hand side in our linear combination is
\begin{multline*}
  K(x,y) \left( (x-x_1)y
\sum_{i\not = p+1} \frac{x_i^p \, Q(x_i,y)}{\prod_{j \not = i, p+1}(x_i-x_j)}
+
(\by-x) y_1 \sum_{i\not = 1} \frac{x_i^p\, Q(x_i,y_1)}{\prod_{j \not = i, 1}(x_i-x_j)}
\right.\\
\left. +
(x_1-\by) \bx \sum_{i\not = 0} \frac{x_i^p\,Q(x_i,\bx) }{\prod_{j \not = i, 0}(x_i-x_j)}
\right).
\end{multline*}
The term $Q(x,y)$ only  occurs in the first sum, with coefficient
\beq\label{div}
x^p y \frac{K(x,y)(x-x_1)}{\prod_{j=1}^p(x-x_j)}=\frac{K(x,y)(x-x_1)}{S'_1(x,y)},
\eeq
by~\eqref{S-diff}. Division of our linear combination by this expression gives Lemma~\ref{lem:section-free}.\end{proof}

%=================================================
\subsection{Extracting $\boldsymbol{Q^{(0,b)}(x,y)}$:
  proof of Proposition~\ref{prop:DF}}
\label{sec:extract}
%=================================================
We now derive the expression for $Q^{(0,b)}(x,y)$ given in  Proposition~\ref{prop:DF} from Lemma~\ref{lem:section-free}. In the identity of the lemma,
both
sides are  power series in $t$ whose coefficients are algebraic
functions of $x$ and $y$ (and the $z_r$'s). More precisely, these coefficients are
written as \emm polynomials, in $x$,
$\bx$, $y$, $\by$, $y_1$ and in  $x_1, \ldots, x_p$ (due to
Lemma~\ref{lem:Lagrange}), symmetric in $x_2, \ldots,
x_p$ (but not $x_1$).
We think of  them as Laurent series in
$\cs[z_0, \ldots, z_p, 1/z_p, x, \bx]((\by))$
(Lemma~\ref{lem:xi}). We extract from each coefficient the monomials
that are nonnegative in $x$ and $y$, and  show that $Q(x,y)$ is the only contribution on the left-hand side --- this is exactly what Proposition~\ref{prop:DF} says.  We proceed line by line
in the identity of Lemma~\ref{lem:section-free}.

In the first line, the term $Q(x,y)$ is clearly nonnegative in $x$ and $y$. Then the coefficient of $t^n$ in the sum over~$i$ is a polynomial in $y$, $\bx$ and $x_1, \ldots, x_p$, symmetric in the $x_i$'s. Since the symmetric functions of the $x_i$'s are $x$-nonpositive (Lemma~\ref{lem:sym-ei}), the second term of the first line only involves negative powers
of $x$ (because of  the factor $\bx$ before the sum), and thus the contribution of the first line reduces to $Q(x,y)$. 

 Let us show that the second line only involves negative  powers
of $y$. The coefficient of~$t^n$ in it is, up to a factor $\by$, a
polynomial in~$y_1$, $\bx$,  and $x_0=x, x_2, \ldots, x_p, x_{p+1}=\by$, symmetric in the latter $p$ variables $x_2, \ldots, x_p, x_{p+1}$. By Proposition~\ref{prop:inv}, the series $x_0=x, x_2, \ldots, x_p, x_{p+1}=\by$ are the solutions of the equation $S(X,y_1)=S(x, y_1)$ (solved for~$X$), $x_0=x$ being the trivial solution. By Lemma~\ref{lem:sym-ei}, applied with~$y$ replaced by $y_1$, the symmetric functions of $x_2, \ldots, x_{p+1}$ are polynomials in $\bx$
and $y_1$. In particular, they are $y$-nonpositive as~$y_1$ itself, and so is the whole second line.
It is even $y$-negative due to the factor $\by$. Hence the second line  does not contribute in the extraction.

Let us finally consider the last term on the left-hand side.
Things are a bit more delicate here: we are going to prove that every
$y$-nonnegative monomial that occurs there is $x$-negative. 
We need the following lemma.

\begin{Lemma}\label{lem:x1}
Let $\GA=\qs(z_0, \ldots, z_p)$. For a series $G(x, \by) \in \GA[x,
\bx]((\by))$, we say that $G$ satisfies  property $\mathcal P$ if
  all monomials  $x^k \by^\ell$ (with $k, \ell \in \zs$) that
  occur in $G(x, \by)$  satisfy $k\le \ell/p$. 
  Equivalently, $G(\bx^p, x\by) \in \GA [x]((\by))$.
  
Then the series~$x_1$ defined in Lemma~\ref{lem:xi} satisfies
$\mathcal P$, as well as all its 
 (positive or negative) powers.
\end{Lemma}
Note that Lemma~\ref{lem:sym-ei} implies that the symmetric functions of $x_1,
\ldots, x_p$ satisfy $\mathcal P$. Also, any sum or product of series
satisfying $\mathcal P$ still satisfies $\mathcal P$, and the
``series'' $\bx$ and $\by$ satisfy~$\mathcal P$. We delay the proof of  Lemma~\ref{lem:x1} to complete the proof of Proposition~\ref{prop:DF}. We get back to the identity of Lemma~\ref{lem:section-free}. The coefficient of $t^n$ in the sum
\[
   \sum_{i\not = 0}
   \frac{x_i^p\, Q(x_i,\bx)}{\prod_{j \not = i, 0}(x_i-x_j)}
\]
is a polynomial in $\bx$ and $x_1, x_2, \ldots, x_{p+1}=\by$, symmetric
in the latter $p+1$ variables. By  the
above observations, it satisfies $\mathcal P$.
Now consider the product $\prod_{i=2}^p(1-\bx x_i)$: it is a polynomial in $\bx$ and $x_2, \ldots, x_p$,
symmetric in the latter $p-1$ variables. If $p_k$ is the $k$th power
sum, we have of course
\[
p_k(x_2, \ldots, x_p)= p_k(x_1, \ldots, x_p)-x_1^k.
\]
We recall that power sums generate (as an algebra) all symmetric
polynomials. Hence the above product is a polynomial in $\bx, x_1$,
and the elementary functions of $x_1, \ldots, x_p$. By the above
observations and Lemma~\ref{lem:x1}, it satisfies $\mathcal P$. So
does the factor $(x_1-\by)$. Hence every monomial occurring in the last part
of the left-hand side in Lemma~\ref{lem:section-free} is of the form  $\bx^2 \by x^k \by^\ell$ with $k\le \ell/p$. If it is nonnegative in $y$, it is $x$-negative, and thus cannot contribute in the extraction.
\qed

\begin{proof}[Proof of Lemma~\ref{lem:x1}]
Recall that $X=x_1= z_p \bx y^p (1+O(\by))$ satisfies $S(x,y)=S(X,y)$, which we write
as~\eqref{eq-x1}. Then we see that  $G(x, \by):=\bX=1/x_1$
satisfies $\mathcal P$. Indeed, denoting $\tilde G= G(\bx^{p} , x\by)$, we have
\[
  z_p \tilde G= \by^{p}+\sum_{\ell=1}^p \sum_{\substack{i, j \ge 0 \\ i+j\le \ell } }z_{i+j+p-\ell}\, x^{ip+\ell} \by^{\ell} \tilde G ^{j+1},
\]
from which, by an inductive argument, it is clear that $\tilde G= \by^{p}/z_p(1+O(\by))$ is a series in $\by$ with polynomial coefficients in~$x$. Moreover, since the first coefficient of $\tilde G$, being $1/z_p$, does not depend on~$x$,
 property $\mathcal P$  holds as well for the
reciprocal of $G$, which is $X=x_1$.
\end{proof}

%========================================================
\subsection{Quadrant tandem walks starting at $\boldsymbol{(a,b)}$: the  series $\boldsymbol{Q^{(a,b)}(x,y)}$}
\label{sec:Qab}
 %========================================================

We finally generalize the expression for $Q^{(0,b)}(x,y)$ given in 
Proposition~\ref{prop:DF} to quadrant tandem walks starting at an arbitrary
position $(a,b)$.

\begin{proof}[Proof of Proposition~\ref{prop:Qab}]
We start from the functional equation~\eqref{eqfunc}, and  adapt  
the solution presented earlier in this section
to the case where $a$ is not necessarily zero. 

First, we generalize the section-free equation of Lemma~\ref{lem:section-free}. We follow step by step the proof of this lemma,  given in
Section~\ref{sec:sec-free}.
By Lemma~\ref{lem:Lagrange}, the linear combination~\eqref{pos} becomes
\beq\label{eq:orbit_Qab}
K(x,y) \sum_{i=0}^p \frac{x_i^p \, Q(x_i,y)}{\prod_{j \not = i,
    p+1}(x_i-x_j)}=h_a(x_0,\ldots,x_p)y^b- t\by \sum_{i=0}^p
\frac{x_i^{p+1} \, Q(x_i,0)}{\prod_{j \not = i,    p+1}(x_i-x_j)}.
\eeq
Hence, denoting  $\bxk_k=(x_0,\ldots,x_{k-1},x_{k+1},\ldots,x_{p+1})$
for $0\le k \le p+1$, the counterpart of~\eqref{eq:ykK-bis} is
obtained by replacing $y_k^{b+1} $ by $h_a(\bxk_k)y_k^{b+1} $. 
We now rewrite the homogeneous symmetric functions $h_a(\bxk_k)$ in a simpler fashion.
Recall that $x_{p+1}$ is defined to be $\by$, while  $x_0, x_1,
\ldots, x_p$ are the roots of $S(X,y)-S(x,y)$
(Lemma~\ref{lem:xi}). In particular,
\[
S(\bu,y)-S(x,y)=\bu\by \prod_{i=0}^p (1-ux_i),
\]
so that, for $k=0, \ldots, p+1$, we have
\beq\label{eq:bxk}
h_a(\bxk_k)=[u^a]\frac{1-ux_k}{\prod_{0\le i \le p+1} (1-ux_i)}
=[u^a](1-ux_k)C(u),
\eeq
where we have defined
\begin{align}
  C(u):=\prod_{i=0}^{p+1}\frac1{1-ux_i}&=\frac1{uy(1-u\by)(S(\bu,y)-S(x,y))}
                                         \label{Cu}
\\
                                       & =\frac1{(1-ux)(1-u\by)(1-u \bx y\sum_{i+j+k<p} z_{i+j+k+1}u^i \bx^j y^k)},
                                         \nonumber
   \end{align}
by~\eqref{Xi-eq}.

Then we take the same linear combination of three equations as
 in the case $a=0$, with weights given by~\eqref{weights}. The
 left-hand side keeps the same form, while the
 right-hand side reads 
\[
  (x_0-x_1)h_a(\bxk_{p+1}  )           y_{p+1}^{b+1}+(x_{p+1}-x_0)h_a(\bxk_1  )y_1^{b+1}+(x_1-x_{p+1})h_a(\bxk_0  )y_0^{b+1}
           \hskip 30mm \
 \]
 \begin{align*}
             &=[u^a]\left[\bigl((x_0-x_1)(1-ux_{p+1})y_{p+1}^{b+1}+(x_{p+1}-x_0)(1-ux_1)y_1^{b+1}\right.\\
     &\hskip 80mm     \left.           +(x_1-x_{p+1})(1-ux_0)y_0^{b+1}\bigr)C(u)\right]\\
&=\bigl((x_0-x_1)y_{p+1}^{b+1}+(x_{p+1}-x_0)y_1^{b+1}+(x_1-x_{p+1})y_0^{b+1}\bigr)[u^a]C(u)\\
&\hskip 30mm
 -\bigl((x_0-x_1)y_{p+1}^{b}+(x_{p+1}-x_0)y_1^{b}+(x_1-x_{p+1})y_0^{b}\bigr)[u^{a-1}]C(u).
           \end{align*}
Let us write $\zeta_a:=[u^a]C(u)$. We now return to the derivation~\eqref{final} and  conclude that the above expression is
\[
           (1-\bx\by)(x-x_1)\bigl(\eta_b\zeta_a-\eta_{b-1}\zeta_{a-1}\bigr),
 \]
where $\eta_b:=\sum_{k=0}^b(y^{k+1}-y_1^{k+1})\bx^{b-k}$ as defined
in Proposition~\ref{prop:Qab} (note that $\eta_{-1}=0$).

We then isolate $Q(x,y)$ by dividing the whole equation
by~\eqref{div}, and thus obtain the counterpart of
Lemma~\ref{lem:section-free}: the left-hand side is unchanged,
  while the right-hand side is
\[
\frac{(1-\bx\by)S'_1(x,y)}{K(x,y)}\big(\eta_b\zeta_a-\eta_{b-1}\zeta_{a-1}\big).
\]

\smallskip

It remains to apply the operator $[x^{\geq}y^{\geq}]$. 
Again,  the only term that survives
on the left-hand side is $Q(x,y)$, and this concludes the proof.
\end{proof}

%%%%%%%%%%%%%%%%%%%%%%%%%%%%%%%%%%%%%%%%%%%%%%%%%%%%%%%%%%%%%%%%%%%
\section{Final  expressions for $\boldsymbol{Q^{(0,b)}(x,y)}$}
\label{sec:Qxy-simple}
%%%%%%%%%%%%%%%%%%%%%%%%%%%%%%%%%%%%%%%%%%%%%%%%%%%%%%%%%%%%%%%
We still fix $p\ge 1$. Our aim now is to derive the expressions for $Q^{(0,b)}(x,y)$ given in
Theorem~\ref{thm:Q0} and Corollary~\ref{cor:Qthird} from Proposition~\ref{prop:DF}. 
We begin with Theorem~\ref{thm:Q0}. As explained in the second remark following
this proposition, it suffices to prove  the case $y=0$. By
linearity (and  Proposition~\ref{prop:DF}), it is enough to prove the following lemma.

\begin{Lemma}\label{lem:toprove}
  For $k\ge 0$, we have
\[
[y^0]  \frac{(1-\bx\by)S'_1(x,y)}{K(x,y)}(y^{k+1}-y_1^{k+1})= 
\frac {Y_1^{k+1}}{tx} \left(1- \frac 1 {tx^2} +\sum_{r= 0}^p
    z_r(r+1)\bx^{r+2}\right) +\frac 1{tx^2} \mathbbm{1}_{k=0}.
\]
\end{Lemma}

The proof that we will give is closely related to the proof of the
equivalence of Propositions~18 and~19 in~\cite{BoBoMe18}.

Recall that  $Y_1=xt+O(t^2)$ is the unique power series in $t$ that
cancels $K(x,Y)$. We will also need to handle the other roots of $K(x,Y)=1-tS(x,Y)$.

\begin{Lemma}\label{lem:Y}
The equation $tS(x,Y)=1$, when solved for $Y$, admits $p+1$ roots
$Y_1, Y_2, \ldots,\break Y_{p+1}$, taken as Puiseux series in $t$. Only $Y_1$ is a power series in $t$.
The other roots are Laurent series in $t^{1/p}$ that contain
some negative powers in $t$. They have coefficients  
in $\cs[z_0, \ldots, z_{p-1}, z_p^{1/p}, 1/z_p,x, \bx]$. 
\end{Lemma}
\begin{proof}
The equation $tS(x,Y)=1$ reads 
\[
Y=t\left( x+ \sum_{i+j\le p} z_{i+j} \bx^i Y^{j+1}\right).
\]
If $t=0$ this reduces to $Y=0$, hence  $Y_1$ is the unique power series solution. Its  expansion in $t$ can be computed iteratively from the equation, and its coefficients lie in $\qs[z_0, \ldots, z_p, x, \bx]$.
The other roots $Y_2, \ldots, Y_{p+1}$ thus involve negative powers of
$t$. Writing $V=1/Y$, the equation $tS(x,Y)=1$, once multiplied by $V^{p+1}$, reads
\beq\label{V-eq}
V^p= tx V^{p+1}+ tz_p +{t} \sum_{i+j \le p, j<p} z_{i+j} \bx^i V^{p-j}.
\eeq
The Newton polygon method~\cite[pp.~498--500]{flajolet-sedgewick} allows us to conclude that the $p$ solutions
$V_2, \ldots, V_{p+1}$ read 
\[ 
V_j= \xi^j z_p^{1/p} t^{1/p}\left(1+o(1)\right),
\] 
where $\xi$ is a primitive  $p$th root of unity,
and have coefficients in $\cs[z_0, \ldots, z_{p-1}, z_p^{1/p}, {1/z_p},\break x,
\bx]$.  The claimed properties of $Y_j=1/V_j$ follow.  
\end{proof}

Going back to Lemma~\ref{lem:toprove}, we need to extract the constant
term in $y$ from a series of the form $N(x,y)/K(x,y)$, where $N(x,y)$ is a Laurent series in $\by$. In our case
\beq\label{N-expr}
N(x,y)= (1-\bx\by)S'_1(x,y)(y^{k+1}-y_1^{k+1}),
\eeq
but, {in the following lemma}, we first  focus on the case where $N$ is a monomial in $y$.

\begin{Lemma}\label{lem:constK}
Upon expanding $1/K(x,y)$ as a power series in  $t$ with coefficients
in the ring $\qs[z_0, \ldots, z_p, x, \bx, y, \by]$, for $k\geq 0$ we have
\[
[y^0]\frac{y^k}{K(x,y)}=-\frac1{tz_p}Y_1^k\prod_{j\neq 1}\frac1{Y_1-Y_j},
\]
while for $k<0$ we have
\[
[y^0]\frac{y^k}{K(x,y)}=\frac1{tz_p}\sum_{i=2}^{p+1}Y_i^{k}\prod_{j\neq i}\frac1{Y_i-Y_j},
\]
where the natural  range of $j$ is $\llbracket 1,p+1\rrbracket$. 
\end{Lemma}
\begin{proof}
The partial fraction decomposition of ${1}/{K(x,y)}$ reads
\begin{align*}
\frac{1}{K(x,y)}=-\frac{y}{tz_p}\prod_j\frac1{y-Y_j}&=-\frac1{tz_p}\sum_i\frac{Y_i}{y-Y_i}\prod_{j\neq i}\frac1{Y_i-Y_j}\\
&=-\frac1{tz_p}\frac{\by Y_1}{1-\by Y_1}\prod_{j\neq 1}\frac1{Y_1-Y_j}+\frac1{tz_p}\sum_{i=2}^{p+1}\frac{1}{1-yY_i^{-1}}\prod_{j\neq i}\frac1{Y_i-Y_j}.
\end{align*}
Recall that $Y_1=O(t)$ while, for $i\ge 2$, $1/Y_i=O(t^{1/p})$. Hence, in the ring of series in $t^{1/p}$,
\[
\frac1{K(x,y)}=-\frac{1}{tz_p}\sum_{k\geq 1}Y_1^k\by^k\prod_{j\not = 1}\frac1{Y_1-Y_j}+\frac1{tz_p}\sum_{k\geq 0}\sum_{i= 2}^{p+1}Y_i^{-k}y^{k}\prod_{j\neq i}\frac1{Y_i-Y_j}.
\]
For $k\neq 0$ this gives the claimed expression 
for $[y^0]({y^k}/{K})=[\by^{k}](1/{K})$.  
For $k=0$ it gives
\[
[y^0]\frac{1}{K}=\frac1{tz_p}\sum_{i= 2}^{p+1}\prod_{j\neq i}\frac1{Y_i-Y_j}.
\]
However, by Lemma~\ref{lem:Lagrange}, 
we have $\sum_{i=1}^{p+1}\prod_{j\neq i}\frac1{Y_i-Y_j}=0$, hence the claimed expression also holds for $k=0$.  
\end{proof}

\begin{proof}[Proof of Lemma~\ref{lem:toprove}]
  For $N(x,y)$ a Laurent series in $\by$, we denote by  $\Nl(x,y):=\break [y^<]N(x,y)$
 the negative part of $N$ in $y$, and by $\Ng(x,y):=[y^\ge]N(x,y)$
 the nonnegative part. Then Lemma~\ref{lem:constK} gives
\beq\label{eq:const_in_N}
[y^0]\frac{N(x,y)}{K(x,y)}=-\frac1{tz_p}\Ng(x,Y_1)\prod_{j\not = 1}\frac1{Y_1-Y_j}+\frac1{tz_p}\sum_{i= 2}^{p+1}\Nl(x,Y_i)\prod_{j\neq i}\frac1{Y_i-Y_j}.
\eeq
For $N(x,y)$ given by~\eqref{N-expr}, it is easy to express $N_<$ and
$N_\ge$. Recall that $S'_1(x,y)$ has valuation $-1$ and
degree $p-1$  in $y$, while  $y_1= x\by^p/z_p(1+O(\by))$ by Lemma~\ref{lem:xi}. This gives
\begin{align*}
  N_\ge(x,y)&=y^{k+1}(1-\bx\by)S'_1(x,y)+\mathbbm{1}_{k=0}\, \bx\by,\\
N_<(x,y)&=-y_1^{k+1}(1-\bx\by)S'_1(x,y)-\mathbbm{1}_{k=0}\, \bx\by.
\end{align*}
According to~\eqref{eq:const_in_N}, we have to evaluate $\Nl(x,y)$ at
$y=Y_i$, for $i\ge 2$, and hence to evaluate the series
  $y_1=1/x_1$ at $y=Y_i$. In the following lemma, we emphasize the fact that $y_1$ depends on $y$ (it is a power series in $\by$) with the notation $y_1(y)$. 
\begin{Lemma}\label{lem:Yi}
Fix $i\in\llbracket 2,p+1\rrbracket$. The series $1/Y_i$ is a power
series in $t^{1/p}$ with no constant term. Hence $y_1(Y_i)$ is a
formal power series in $t^{1/p}$, which in fact  equals $Y_1$.
\end{Lemma}
\begin{proof}
  The first statement follows from Lemma~\ref{lem:Y}, so it remains to identify $y_1(Y_i)$. Recall that $S(x,y_1(y))=S(x,y)$. Replacing 
  $y$ by $Y_i$ gives $S(x,y_1(Y_i))=S(x,Y_i)$.   
Since $Y_i$ is a root (in $y$) of $K(x,Y)=1-tS(x,Y)$, it follows that
$K(x,y_1(Y_i))=1-tS(x,y_1(Y_i))=0$ as well. Hence $y_1(Y_i)$ is one of
the $Y_j$'s. But $Y_1$ is the only $Y_j$ that does not contain
negative powers of $t$ (Lemma~\ref{lem:Y}), and we conclude that $y_1(Y_i)=Y_1$.  
\end{proof}
We can now apply~\eqref{eq:const_in_N}. This gives
\[
[y ^0] \frac{N(x,y)}{K(x,y)} = - \frac{1}{tz_p} 
  \sum_{i=1}^{p+1} \left(Y_1^{k+1}(1-\bx/Y_i) S'_1(x,Y_i) 
+ \mathbbm{1}_{k=0}\, \bx /{Y_i} \right)\prod_{j\not = i} \frac
1{Y_i-Y_j}.
\]
We will  evaluate this sum using Lemma~\ref{lem:Lagrange}. The Laurent
polynomial $P(y):=(1-\bx/y) S'_1(x,y)$ has degree $p-1$ in $y$, and
valuation $-2$. Moreover,
\[
P_{-2}:=[\by^2]P(y)=-\bx, \qquad \hbox{and} \qquad P_{-1}:=[\by]P(y)=1+\sum_{r=0}^p
z_rr\bx^{r+2}.
\]
Hence, by Lemma~\ref{lem:Lagrange},
\[
[y ^0] \frac{N(x,y)}{K(x,y)} =- \frac 1 {tz_p}\frac{(-1)^p}{{\prod_i Y_i}}
\left(  Y_1^{k+1}P_{-1} +\bx  \mathbbm{1}_{k=0}    +
 Y_1^{k+1}P_{-2}  { h_1(1/Y_1, \ldots,
  1/Y_{p+1})}\right).
\]
The elementary symmetric functions of $1/Y_1, \ldots,  1/Y_{p+1}$ are easily computed using the fact that each of them is a
  root $V$ of~\eqref{V-eq}. One finds 
\[
  e_{p+1}(1/Y_1, \ldots,  1/Y_{p+1})= \frac 1 {\prod_{i}Y_i}= (-1)^{p-1} \bx z_p
\]
and
\[
  e_1(1/Y_1, \ldots,  1/Y_{p+1})=\frac 1 {tx}\left(
  1-t \sum_{r=0}^p z_r \bx^r\right).
\]
Since $e_1=h_1$, this gives the expression of Lemma~\ref{lem:toprove}.
\end{proof} 

\begin{proof}[Proof of Corollary~\ref{cor:Qthird}] We start from the
  expression for $Q^{(0,b)}(x,y)$ given in Theorem~\ref{thm:Q0}. By linearity, it suffices to prove that,   for $0\le k \le b$,
\beq\label{eq11}
- Y_1^{k+1}= [z^0] \frac{tz^{k+2} S'_2(x,z)}{K(x,z)}.
\eeq
The numerator occurring on the right-hand side is a polynomial in $z$, because
$S'_2(x,z)$ has valuation $-2$ in~$z$. Hence  the first part of
Lemma~\ref{lem:constK} tells us that
\beq\label{eq12}
  [z^0] \frac{tz^{k+2} S'_2(x,z)}{K(x,z)}= -\frac1 {z_p} Y_1^{k+2}
  S'_2(x,Y_1) \prod_{j\not = 1} \frac 1 {Y_1-Y_j}
.
\eeq
Upon writing
\[
S(x,z)-1/t= S(x,z)-S(x,Y_1) = -\frac 1 t K(x,z) =z_p\bz \prod_{i=1}^{p+1}(z-Y_i),
\]
we can compute
\[
S'_2(x,Y_1)= \frac{z_p}{Y_1} \prod_{i=2}^{p+1}(Y_1-Y_i),
\]
which, combined with~\eqref{eq12}, gives~\eqref{eq11}. An alternative, purely combinatorial proof of~\eqref{eq11} in terms of one-dimensional lattice walks is given in Section~\ref{sec:proof_eq_const}.
\end{proof}

%%%%%%%%%%%%%%%%%%%%%%%%%%%%%%%%%%%%%%%%%%%%%%%%%%%%%%%%
\section{Quadrant
  % {tandem}
  walks with arbitrary endpoint: algebraic solution}
\label{sec:algeq}
%%%%%%%%%%%%%%%%%%%%%%%%%%%%%%%%%%%%%%%%%%%%%%%%%%%%%%%%%%%%%%

We are now going to prove Theorem~\ref{thm:alg}, which gives an explicit
algebraic expression for the series $Q^{(a,b)}(1,1)$. We still work in
the $p$-specialization, for $p\ge 1$. We  write  $Q(x,y)\equiv Q^{(a,b)}(x,y)$.

The proof that we give generalizes the proof given for $p=1$ in~\cite[Sec.~5.2]{mbm-mishna}. In our basic functional equation~\eqref{eqfunc}, let us replace the
pair $(x,y)$ by  $(x_i,\bx)$, with $1\le i \le p$. By
Proposition~\ref{prop:inv},  $K(x_i, \bx)=K(x,y)$. We thus obtain
\beq\label{eqxi}
K(x,y)Q(x_i,\bx)=x_i^a\bx^b- t x_ix Q(x_i,0) - \sum_{k=1}^p \bx_i^k G_k(\bx).
\eeq
We now consider  the linear combination~\eqref{eq:orbit_Qab}.
It involves, on the
right-hand side, the series $Q(x,0)$, and then $Q(x_i,0)$ for $1\le i
\le p$. By taking a linear combination with~\eqref{eqxi}, for $1\le
i\le p$, we can eliminate the latter $p$ series --- upon introducing
the series $G_k(\bx)$. More precisely,
\begin{multline}\label{to-alg}
  K(x,y) \left( y \sum_{i=0}^p \frac{x_i^p \, Q(x_i,y)}{\prod_{j\not =
      i, p+1} (x_i-x_j)} - \bx \sum_{i=1}^p \frac{x_i^p \, Q(x_i,\bx)}{\prod_{j\not =
      i, p+1} (x_i-x_j)}\right)\\
=
h_a(x_0,\ldots,x_p)y^{b+1} -t\, \frac{x^{p+1} Q(x,0)}{\prod_{j\not =0,
       p+1} (x-x_j)}\\
 - \bx \sum_{i=1}^p \frac{x_i^p}{\prod_{j\not =
      i, p+1} (x_i-x_j)}\left(x_i^a\bx^b-  \sum_{k=1}^p \bx_i^k G_k(\bx)\right).
\end{multline}
We now simplify the right-hand side. The coefficient of $Q(x,0)$
can be rewritten in terms of $S'_1(x,y)$ due to~\eqref{S-diff}. The
sums over $i$ can be evaluated in closed form using~\eqref{S-diff}
again, and Lemma~\ref{lem:Lagrange}. First, since $a\ge 0$,
\begin{align*}
  \sum_{i=1}^{p} \frac{x_i^{p+a}}{\prod_{j\not = i, p+1} (x_i-x_j)}
&=
\sum_{i=0}^p \frac{x_i^{p+a}}{\prod_{j\not = i, p+1} (x_i-x_j)}-
\frac{x^{p+a}}{\prod_{j=1}^p (x-x_j)}\\
&=
h_a(x_0, x_1, \ldots, x_p)- \frac{x^{p+a}}{\prod_{j=1}^p (x-x_j)}  \qquad \hbox{by Lemma~\ref{lem:Lagrange}}
\\
&=  h_a(x_0, x_1, \ldots, x_p)- \frac{x^a\by}{S'_1(x,y)} \qquad  \quad
\qquad     \hbox{by~\eqref{S-diff}}. 
\end{align*}
Similarly, for $1\le k \le p$,
\begin{align*}
  \sum_{i=1}^p \frac{x_i^{p-k}}{\prod_{j\not = i, p+1} (x_i-x_j)}
&=
\sum_{i=0}^{p} \frac{x_i^{p-k}}{\prod_{j\not = i, p+1} (x_i-x_j)}-
\frac{x^{p-k}}{\prod_{j=1}^p (x-x_j)}\\
&=
- \frac{x^{p-k}}{\prod_{j=1}^p (x-x_j)}  \qquad \hbox{by Lemma~\ref{lem:Lagrange}}
\\
&= - \frac{\bx^k\by}{S'_1(x,y)} \qquad \qquad \quad \hbox{by~\eqref{S-diff}}. 
\end{align*}
Hence the right-hand side of~\eqref{to-alg}  simplifies to
\[
h_a(x_0,\ldots,x_p)\left( y^{b+1}-\bx^{b+1}\right) +\frac{\bx\by}{S'_1(x,y)}\left(x^{a-b}-tx^2Q(x,0)-\sum_{k=1}^p\bx^kG_k(\bx)\right).
\]
The rightmost term can be expressed in terms of $Q(x,\bx)$. Indeed,
specializing the main functional equation~\eqref{eqfunc} to the case $y=\bx$,
we obtain
\[
  K(x,\bx) Q(x,\bx)=x^{a-b}- tx^2Q(x,0) -\sum_{k=1}^p \bx^k G_k(\bx).
\] 
We can thus  rewrite the right-hand side of~\eqref{to-alg} as
\beq\label{rhs-new}
h_a(x_0,\ldots,x_p)
\left(y^{b+1}-\bx^{b+1}\right )+\frac{\bx\by}{S'_1(x,y)}
K(x,\bx)Q(x,\bx).
\eeq
Recall from~\eqref{eq:bxk} and~\eqref{Cu} that
\begin{align}
  h_a(x_0, \ldots, x_p)=h_a(\hat{\mathbf{x}}_{p+1}) &
        =[u^a]\frac1{uy(S(\bu,y)-S(x,y))} \nonumber
\\
  & =[u^a]\frac1{(1-ux)(1-u \bx y\sum_{i+j+k<p} z_{i+j+k+1}u^i \bx^j y^k)}\nonumber
  \\
 & :=D_a(x,y). \label{Da}
\end{align}

Equation~\eqref{to-alg}, with its right-hand side written as~\eqref{rhs-new}, holds for indeterminates $x$ and $y$, where
$x_1, \ldots, x_p$ are the roots of $S(X,y)=S(x,y)$ distinct from
$x$. By Lemma~\ref{prop:inv}, the $x_i$'s can also be described as the roots $X$ of $S(x, 1/X)=S(x,y)$ distinct from $\by$. Observe that we have not used the fact that we usually take them as Puiseux series in $\by$. We can choose them in any   algebraic  closure of $\qs(z_1, \ldots, z_p, x, y)$, and~\eqref{to-alg}  still holds. We now specialize  $y$ to $Y_1$, which is  the unique power series in $t$ satisfying
  $K(x,y)=1-tS(x,y)=0$. Then the corresponding values $x_1, \ldots, x_p$ are
  the roots $X$ of $S(x,1/X)=S(x,Y_1)$ distinct from $1/Y_1$, or
  equivalently the roots $X$ of $K(x,1/X)=0$ distinct from $1/Y_1$. We choose to take them as Puiseux series in $t$, hence
  they  are in fact the series  $1/Y_2, \ldots, 1/Y_{p+1}$, with the $Y_i$'s   defined in   Lemma~\ref{lem:Y}. Note that each $1/Y_i$, for $i\ge 2$, is a formal power series in $t^{1/p}$ with no constant term.

  With these values of $x_1, \ldots, x_p$, 
  the series $Q(x_i,Y_1)$ and $Q(x_i,\bx)$ occurring
in~\eqref{to-alg} (specialized to $y=Y_1$) are
well defined power series in $t^{1/p}$.  But since $K(x,Y_1)=0$, the
left-hand side of~\eqref{to-alg} vanishes, and so does its right-hand
side, which we have simplified to~\eqref{rhs-new}. 
We thus obtain
\beq\label{Qxbx}
{K(x,\bx)}Q(x,\bx)={D_a(x,Y_1)\left(\bx^{b+1}-Y_1^{b+1}\right)xY_1S'_1 (x,Y_1)}.
\eeq
We now want to express $S_1'(x,Y_1)$. Since the series $x_1, \ldots, x_p$ are $1/Y_2, \ldots, 1/Y_{p+1}$ when $y=Y_1$, the specialization of~\eqref{S-diff} at
$y=Y_1$ reads
\[
  Y_1  S'_1(x,Y_1)= \prod_{i=2}^{p+1} (1-\bx /Y_i).
\]
On the other hand, $K(x,y)$ factors as
\[
  K(x,y)= \frac{tx}{Y_1} (1-\by Y_1) \prod_{i=2}^{p+1}(1-y/Y_i),
\]
so that
\[
  K(x,\bx)= \frac{tx}{Y_1} (1-x Y_1) \prod_{i=2}^{p+1}(1-\bx/Y_i).
\]
With these two identities, Equation \eqref{Qxbx} gives
\[
  Q(x,\bx) = \frac{Y_1}{tx} D_a(x,Y_1) \frac{\bx^{b+1}-Y_1^{b+1}}{\bx -Y_1}.
\]
If $x=1$, the series $Y_1$ specializes to $W$, and we obtain
\[
  Q(1,1) = \frac{W}{t} D_a(1,W) \frac{1-W^{b+1}}{1 -W},
\]
where $D_a(x,y)$ is defined by~\eqref{Da}. This gives the expression
of  Theorem~\ref{thm:alg}.
\qed

\begin{remark}
We have first obtained an expression for the series $Q(x,\bx)$,
which counts tandem walks starting at $(a,b)$  
with a weight $x^{i-j}$ for walks ending at $(i,j)$. Then we have
specialized this expression to $x=1$.  There is no loss of information in 
this specialization. Indeed, since
every SE step lets $i-j$ increase by $2$ 
and each face step of level $r$ lets $i-j$ decrease by $r$, the series
 $x^{b-a}Q(x,\bx)$ is equal to $Q(1,1)$ where $t$ is replaced 
 by $tx^2$ and $z_r$  by $z_rx^{-r-2}$. 
\end{remark}

%%%%%%%%%%%%%%%%%%%%%%%%%%%%%%%%%%%%%%%%%%%%%%%%%%%%%%%%%%%%%
\section{Bijective proofs}\label{sec:bij_proofs}
%%%%%%%%%%%%%%%%%%%%%%%%%%%%%%%%%%%%%%%%%%%%%%%%%%%%%%%%%%%%%
In this section we give combinatorial proofs of the expression for $Q^{(0,b)}(x,0)$ given in {Theorem~\ref{thm:Q0}} and of the expression for $Q^{(a,b)}(1,1)$
given in {Theorem~\ref{thm:alg}}.
In both cases we use the KMSW bijection to interpret our series as
\gfs\ of marked bipolar orientations, and perform simple
transformations on these  orientations to establish the identities.

%============================================================
\subsection{The expression for $\boldsymbol{Q^{(0,b)}(x,0)}$}\label{sec:bij_proofs_0b}
%===========================================================

For $i, b\geq 0$ let us denote by $Q_i^{(b)}:=[x^i]Q^{(0,b)}(x,0)$ the \gf\ of tandem walks
starting at $(0,b)$, staying in the quadrant, and ending at $(i,0)$.   The variable~$t$ records the number of steps,~$x$  the final $x$-coordinate,
and~$z_r$  the number of face steps of level~$r$.  
Similarly, let $H_i^{(b)}$ be the series that counts tandem walks 
starting at $(0,b)$, staying in the half-plane $\{y\geq 0\}$, and
ending at $(i,0)$. As explained after the definition~\eqref{eq:Y} of
$Y_1$, we have
  \[
    Y_1= tx \sum_{i\in \zs} x^i H_i^{(0)},
  \]
  so that $Y_1$  counts tandem walks starting at the origin,
ending on the line $\{y=-1\}$, but staying on or above the $x$-axis
until this last step. By concatenating $b+1$ such walks, and
translating the resulting walk $b$ steps up, we see that  $Y_1^{b+1}$
counts tandem walks starting at $(0,b)$, ending on the line $\{y=-1\}$
and staying in $\{y\geq 0\}$ up to the last step. In other words, 
\beq\label{HY}
   Y_1^{b+1}= tx\sum_{i\in \zs} x^i H_i^{(b)}=: tx H^{(b)}(x).
\eeq
Through the KMSW bijection described in Section~\ref{sec:KMSW} (see in
particular Figure~\ref{fig:corresp}), the
series $H_i^{(b)}$ counts  marked bipolar orientations of signature  $(a, b; a+i, 0)$,
for some  $a\geq 0$, where~$t$ records the number
of plain edges minus $1$ --- called the \emph{size} of the orientation --- 
and~$z_r$  the number of inner faces of degree $r+2$.   These
orientations  have no dashed edge on their right boundary. The series $Q_i^{(b)}$ counts
those that, in addition, have no dashed edge on the left boundary.  That is,
those for which $a=0$.

We now consider an orientation $O$ counted by $H_i^{(b)}$. Let $e$
 be the bottom edge of the right outer boundary
of $O$, directed from $S$ to a vertex $v$. Then $e$ is necessarily plain. Let~$f$ be the face on its left. Three cases occur:
\begin{enumerate}[label=(\arabic{*}),ref={\arabic{*}}]
\setcounter{enumi}{-1}
\item\label{0-type}$f$ is the outer face,
\item\label{1-type}$f$ is an inner face and   $e$  is the unique ingoing edge at $v$,
\item\label{2-type}$f$ is an inner face and  there are several ingoing edges at $v$.
\end{enumerate}
 Accordingly, $O$ will be said to be of type \ref{0-type}, \ref{1-type} or \ref{2-type}. The associated
 \gfs\ are denoted $H_{0,i}^{(b)}$, $H_{1,i}^{(b)}$ and $H_{2,i}^{(b)}$. Clearly,
\beq\label{Hsum}
H_i^{(b)}=H_{0,i}^{(b)}+H_{1,i}^{(b)}+H_{2,i}^{(b)}.
\eeq

\begin{Lemma}\label{lem:UV}
For $i\geq 0$, the above defined series satisfy:
\begin{align}
H^{(b)}_{0,i}&=\mathbbm{1}_{i=b=0}+tQ_{i-1}^{(b-1)},  \label{H0}
\\
H^{(b)}_{1,i+2}&=t\bigl(H_i^{(b)}-Q_i^{(b)}\bigr),\label{H1}\\
H^{(b)}_{2,i}&=t \sum_{r\geq   0}(r+1)z_r H_{i+r}^{(b)}.\label{H2}
\end{align}
\end{Lemma}

\begin{proof}
We adopt the notation ($e$, $f$) used above in the
definition of the types.

Let $O$ be  an orientation of type \ref{0-type},  counted by $H^{(b)}_{0,i}$.  By definition $e$ is plain, so that
$a=0$. Erasing $e$ leaves a bipolar orientation of signature $(0, b-1;i-1,0)$, or just a
single point if $O$ is reduced to the edge $e$ (in which case
$b=i=0$). This gives the first identity.

\begin{figure}
\begin{center}
\includegraphics[width=\linewidth]{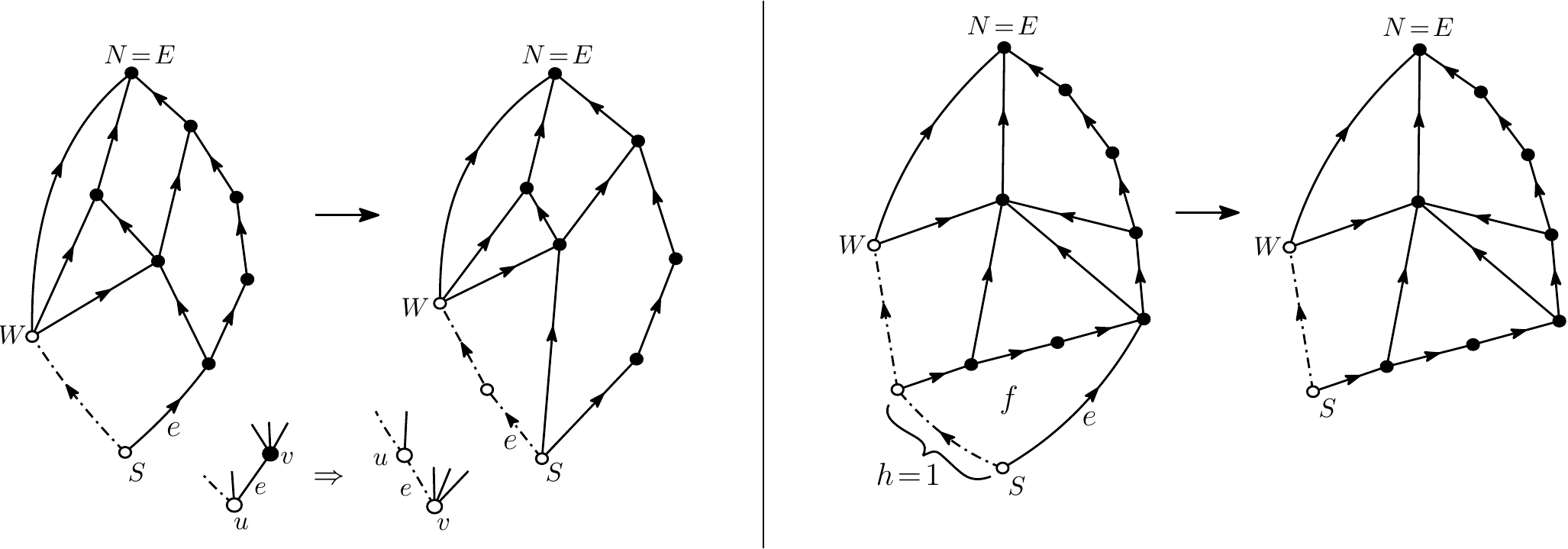}
\end{center}
\caption{Left: sliding the bottom-right  boundary edge from right to left. Right: erasing the inner face incident to the bottom-right edge.}
\label{fig:local}
\end{figure}

\medskip
Now let $O$ be a marked bipolar orientation of type \ref{1-type} counted by $H^{(b)}_{1,i+2}$.
Its signature is of the form  $(a,b; a+i+2,0)$ for some $a\geq 0$.  
Let $O'$ be obtained by reversing the orientation of $e$, and transforming
it into a dashed edge (Figure~\ref{fig:local}, left). Graphically, we
``slide'' the bottom edge $e$ from the right to the left boundary.  It is easy to see that  we thus obtain a marked bipolar orientation with source $v$, of signature $(a+1, b; a+1+i,0)$,
containing at least one dashed  edge (on the left boundary).
Conversely, let us start from  a marked 
bipolar orientation $O'$ of  signature  $(\alpha,b ; \alpha+i,0)$ for
 some $\alpha\geq 1$. Such orientations are counted by $H_i^{(b)}-Q_i^{(b)}$. Since the vertex lying just above the source on the left boundary has indegree $1$, we can reverse the direction of the bottom dashed edge, and make it plain: this gives a
marked bipolar orientation $O$  of type 1, of  
signature $(a,b;a+i+2,0)$ with  $a=\alpha-1$, and thus counted by $H^{(b)}_{1,i+2}$. 
This correspondence yields the second identity of the lemma.

\medskip
Finally let $O$ be a marked bipolar orientation of type \ref{2-type} counted by $H^{(b)}_{2,i}$, of signature  $(a,b;i+a,0)$ for some $a\geq 0$. 
  Let $r+2$ be the degree of the face $f$ and let $h\ge 0$ be the number
of dashed edges on the left boundary of $f$. Since $v$ has
  indegree at least $2$, the top edge of the left boundary of $f$
  cannot be dashed, by definition of marked
bipolar orientations. Hence $h\le r$. If $h>0$, then
 the $h$ dashed edges of $f$ form a path $P$ of length~$h$
on the left boundary of $f$, starting from~$S$; we call $P$
the \emph{bottom-left path} of $f$.     

Let $O'$ be obtained by  erasing  $e$ and $P$,
thereby choosing the top vertex of $P$
as the new source (Figure~\ref{fig:local}, right).  The size has decreased
by one,  and the signature of $O'$ is $(a-h, b; a-h+i+r,0)$. Thus $O'$
is {a marked bipolar orientation} counted by $H_{i+r}^{(b)}$.

Conversely, consider a marked bipolar orientation $O'$ counted 
by $H_{i+r}^{(b)}$, and let $(\alpha,b;\alpha+i+r,0)$ be its signature. 
For $0\leq h\leq r$ consider the operation of attaching a 
path $P$ of $h$ dashed edges  below the source  $S$ of $O'$, choosing
the bottom vertex of $P$ as the new source,
and then adding a new edge $e$ so that it becomes the bottom-edge
of the right outer boundary and encloses an inner face of degree $r+2$.  
We obtain a marked bipolar orientation $O$ counted by $H^{(b)}_{2,i}$, of  
 signature  $(a,b; a+i,0)$, where $a=\alpha+h$. The resulting
correspondence thus gives
\[
H^{(b)}_{2,i}=t\sum_{r\geq 0}\sum_{h=0}^{r}z_r H^{(b)}_{i+r}=t\sum_{r\geq 0}(r+1)z_r H^{(b)}_{i+r},
\]
which concludes the proof of the lemma.
\end{proof}

\begin{proof}[Proof of {Theorem~\ref{thm:Q0}}]
We now establish the expression~\eqref{eq:Qb0} of
$Q^{(0,b)}(x,0)$. As explained below {Theorem~\ref{thm:Q0}}, this
suffices to prove the entire proposition. For $i, b \ge0$,
\begin{align*}
Q_i^{(b)}=H_i^{(b)}-(H_i^{(b)}-Q_i^{(b)})&=
H_i^{(b)}-\frac1{t}H_{1,i+2}^{(b)}
  \hskip 65mm \hbox{by~\eqref{H1}}\\
&=H_i^{(b)}-\frac1{t}\Big(H_{i+2}^{(b)}-tQ_{i+1}^{(b-1)}-H_{2,i+2}^{(b)}\Big)
\hskip 12mm \hbox{by~\eqref{Hsum} and~\eqref{H0}}\\
&=Q_{i+1}^{(b-1)}+H_i^{(b)}-\frac1{t}H_{i+2}^{(b)}+\sum_{r\geq 0}z_r(r+1)H_{i+r+2}^{(b)}\hskip 10mm \hbox{by~\eqref{H2}}.
\end{align*}
By convention, $Q_{i+1}^{(-1)}=0$. Recall that
$Q^{(0,b)}(x,0)=\sum_{i\geq 0}x^iQ_i^{(b)}$. Recall the definition of $H^{(b)}(x)$ in~\eqref{HY}, and let
$A(x):= 1-\frac1{tx^2}+\sum_{r\geq 0}z_r(r+1)\bx^{r+2}$. Multiplying
the above identity by $x^i$, and summing over $i\ge 0$, we get
\begin{align*}
  Q^{(0,b)}(x,0)&=[x^{\geq}]\left(\bx
                  Q^{(0,b-1)}(x,0)+A(x) H^{(b)}(x) \right)
\\
&=[x^{\geq}]\left(\bx
                  Q^{(0,b-1)}(x,0)+ A(x)\frac{Y_1^{b+1}}{tx}
  \right).
\end{align*}
The case $y=0$ of {Theorem~\ref{thm:Q0}} then follows by induction on $b\ge 0$, due to the property, already used in Section~\ref{sec:counting}, that $[x^\geq](\bx[x^{\geq}] G(x))
=[x^{\geq}](\bx G(x))$ for any series $G$.
\end{proof}

\begin{remark}
There is an analogy between the combinatorial proof
  given above and an argument used by Bouttier
and Guitter in~\cite[Sec.~3.3]{bouttier2012planar}. Their aim is 
to determine the \gf\ $M_i$ of rooted planar maps with outer degree~$i$, with $t$ recording the number of edges and $z_r$ the number of inner faces of degree $r$ (the notation is ours). 
They consider the \gf\ $N_i$ of a larger class: rooted planar maps of
outer degree~$i$  with an additional marked
vertex $v$ (with the condition that the root-vertex minimizes the distance to $v$ among all outer   vertices).
The \gf\ $N_i$ is easy to compute using a bijection with
certain labelled trees (called \emm mobiles,). Then they
express $M_i$ as $N_i-(N_i-M_i)$, and determine 
$N_i-M_i$ using a local operation consisting
in ``opening'' the first edge on the leftmost geodesic path from the
root-vertex to the marked vertex (Figure~6 in~\cite{bouttier2012planar}). 

Similarly, to determine the \gf\ $Q_i^{(b)}$ of bipolar orientations, we consider the larger class of  \emm marked, bipolar orientations counted by $H_i^{(b)}$.
We express $Q_i^{(b)}$ as $H_i^{(b)}-(H_i^{(b)}-Q_i^{(b)})$. 
Then  $H_i^{(b)}$ is computable due to the bijection with tandem walks in the upper half-plane, 
while $H_i^{(b)}-Q_i^{(b)}$ is determined  via the local operation
consisting in ``sliding'' the lower left outer edge to the
  right boundary (Figure~\ref{fig:local}, left).
\end{remark}

%%%===================================================
\subsection{The expression for $\boldsymbol{Q^{(a,b)}(1,1)}$}
\label{sec:prop_alg_combi}
%%%===================================================
We will now give a combinatorial proof of {Theorem~\ref{thm:alg}}
using  the  involution $\sigma$   on marked
bipolar orientations defined in Definition~\ref{def:symmetries}, and 
 illustrated in Figure~\ref{fig:involutions}.  
We recall the expression of {Theorem~\ref{thm:alg}} here, but  for convenience 
we exchange the roles of the indices $(i,j)$ and $(a,b)$: for $i, j \ge 0$,
\beq\label{eq:Q11bis}
Q^{(i,j)}(1,1)=\frac{W}{t}\cdot\sum_{a=0}^i\A_a\cdot\sum_{b=0}^jW^b,
\eeq
where $W=Y_1(1)$ satisfies~\eqref{W-eq} 
and $\A_a$ is the series in $W$ and the $z_r$'s defined 
by~\eqref{Ai-def}.

Let $\cQ^{(i,j)}$ be the family of tandem walks 
starting at $(i,j)$ and staying in the quadrant. 
For $a,b\geq 0$, let $\cH^{b\rightarrow a}$ be the 
family of tandem walks starting at $(0,b)$, staying in the upper half-plane $\{y\geq 0\}$, reaching the $x$-axis at least once, and ending on the line $\{y=a\}$.

Recall from Section~\ref{sec:KMSW} that for a walk $w$, the signature
of the marked bipolar orientation $\Phi(w)$, where $\Phi$ is the KMSW
bijection, is given by~\eqref{sign}.
Recall also the definition of the involution $\sigma$  on marked
bipolar orientations (Definition~\ref{def:symmetries}).

\begin{Proposition}\label{prop:invol} 
 The mapping {$\Phi^{-1}\circ\sigma\circ \Phi$} is an involution on 
non-embedded tandem walks (seen as sequences of   steps),
which exchanges $a=\xstart-\xmin$ and  $d=\yend-\ymin$, while preserving  $b=\ystart-\ymin$ and $c=\xend-\xmin$. It also preserves
the length, the number of SE steps, and the number of 
face steps of each level $r$.

Upon embedding walks appropriately, it induces a bijection between 
$\cQ^{(i,j)}$ and\break $\bigcup_{0\leq a\leq i}\bigcup_{0\leq b\leq
  j}\cH^{b\rightarrow a}$,  preserving the same statistics.
\end{Proposition}
\begin{proof}
The first part directly follows from the properties of $\Phi$ and
$\sigma$ (see Theorem~\ref{thm:KMSW}, Eq.~\eqref{sign} and Definition~\ref{def:symmetries}).

For the second part, we simply fix $i$ and $j$ and restrict
$\Phi^{-1}\circ\sigma\circ \Phi$ to paths $w$ such that
$a=\xstart-\xmin\le i$ and $b=\ystart-\ymin\le j$. We embed them so
that they start at $(i,j)$: then they are exactly the walks of
$\cQ^{(i,j)}$. Then we embed the walks obtained by applying
$\Phi^{-1}\circ\sigma\circ \Phi$ so that they start at $(0,b)$. This
gives the announced result, illustrated by
Figure~\ref{fig:sigma_walks}.
\end{proof}

\begin{figure}[htb]
\begin{center}
\includegraphics[width=12cm]{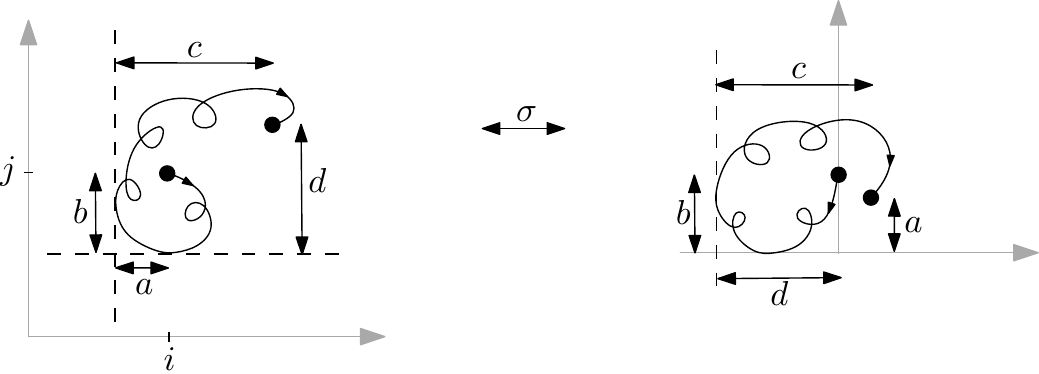}
\end{center}
\caption{The bijection of Proposition~\ref{prop:invol}
  (derived from $\sigma$) between $\cQ^{(i,j)}$ and $\bigcup_{0\leq a\leq i}\bigcup_{0\leq b\leq     j}\cH^{b\rightarrow a}$.
}
\label{fig:sigma_walks}
\end{figure}

Due to the above proposition, in order to prove~\eqref{eq:Q11bis} it now suffices to prove a half-plane result dealing with $\cH^{b\rightarrow
  a}$. But half-plane problems are in essence problems of walks
  on a half-line, and hence much simpler and perfectly understood.

\begin{Lemma}\label{lem:interpretation}
  For $a,b\geq 0$ the \gf\
  of $\cH^{b\rightarrow a}$ (with $t$ recording the   length and~$z_r$ 
  the number of face steps of level $r$) is $\frac{W}{t}\A_aW^b$, with $\A_a$ defined by~\eqref{Ai-def}.  
\end{Lemma}
\begin{proof}
Each walk of $\cH^{b\rightarrow a}$ can be uniquely factored into a
walk of $\cH^{b\rightarrow 0}$ hitting the $x$-axis \emm only once,,
followed by a walk of $\cH^{0\rightarrow a}$. With the notation
  used at the beginning of Section~\ref{sec:bij_proofs_0b},
walks of $\cH^{b\rightarrow 0}$ hitting the $x$-axis only once are counted by
$txH^{(b-1)}(x)=Y_1(x)^{b}$ if we keep track of the abscissa of the endpoint (with the variable $x$), and thus
by $W^b= Y_1(1)^b$ if we do not. It thus suffices to prove that
half-plane walks going from the origin to ordinate $a$ are counted by
$\frac{W}{t} \A_a$. However, it is a classical one-dimensional
result~\cite{banderier-flajolet,bousquet-petkovsek-recurrences,gessel-factorization},
obtained in one line using the so-called \emm kernel method,, that their \gf \ is
\beq\label{1D-res}
 \sum_{a\ge 0} H^{0\rightarrow a} u^a= \frac{1-\bu W}{K(1,u)}.
\eeq
However, we have 
\begin{align*}
  K(1,u)&= K(1,u)-K(1,W)
\\ &= t( S(1,W)-S(1,u))
\\ &= t (u-W)\left( \frac{\bu}{W} -\sum_{i, j, k \ge 0} z_{i+j+k+1} W^j u^k\right)
\hskip 10mm \hbox{by~\eqref{Sdiff}}.
\end{align*}
Combined with~\eqref{1D-res}, this gives
\beq\label{1D-res-bis}
  \sum_{a\ge 0} H^{0\rightarrow a} u^a=\frac W t \cdot \frac 1 {1-uW
    \sum_{j,k\ge 0} W^j u^k\sum_{r>j+k}z_r},
\eeq
which concludes the proof of the lemma and of {Theorem~\ref{thm:alg}}. We have relied on a classical one-dimensional result to obtain~\eqref{1D-res-bis}, but we also give a purely combinatorial proof of the latter identity in Section~\ref{sec:app1}.
\end{proof}

\begin{remark}
Let us  return to the \emm double-tandem walks, defined just before Section~\ref{sec:main_asymp}. We believe  that the first statement of Proposition~\ref{prop:invol} also holds for these walks, or more precisely, that there exists a length-preserving involution on (non-embedded) double-tandem walks % (step-set $\{N,W,SE,S,E,NW\}$) 
that exchanges  $a=\xstart-\xmin$ and $d=\yend-\ymin$, 
while preserving $b=\ystart-\ymin$ and $c=\xend-\xmin$. Moreover, this involution would preserve the total number of steps in $\{N,W,SE\}$.

We have tested this conjecture numerically as follows. First, we embed walks canonically so as to have $x_{\min}=y_{\min}=0$. That is,  walks are now confined to the quadrant and touch both boundaries. We let $\widetilde{D}[a,b,c,d;\ell,m]$ be the number of such walks, with double-tandem steps, that go   from $(a,b)$ to $(c,d)$, have $\ell$ steps in $\{N,W,SE\}$ and $m$ steps in $\{S,E,NW\}$. Our conjecture  translates into
\begin{equation}\label{eq:symmetry}
\widetilde{D}[a,b,c,d;\ell,m]=\widetilde{D}[d,b,c,a;\ell,m].
\end{equation}
In order to compute these numbers, we let $D[a,b,c,d;\ell,m]$ be the corresponding numbers for  quadrant walks that do not necessarily touch both boundaries. Then
  \begin{multline*}
    \widetilde{D}[a,b,c,d;\ell,m]= D[a,b,c,d;\ell,m]-D[a-1,b,c-1,d;\ell,m]
\\-D[a,b-1,c,d-1;\ell,m]+D[a-1,b-1,c-1,d-1;\ell,m].
\end{multline*}
The reflection principle~\cite{gessel-zeilberger}, or equivalently  the approach of~\cite{mbm-mishna} for quadrant walks with small steps (see in particular Prop.~10), gives
\begin{multline*}
  D[a,b,c,d;\ell,m] = \\
  [x^c y^d s^\ell t^m ]
  \frac{x^{a+1} y^{b+1} -\bx^{a+1} y^{a+b+2} +\bx^{a+b+2}y^{a+1}-\bx^{b+1}\by^{a+1} +x^{a+1} \by^{a+b+2}-x^{a+b+2}\by ^{b+1}}{ 1-s(\bx+y+x\by) -t(x+\by +\bx y)},
\end{multline*}
 which yields   an efficient computation method
 to test~\eqref{eq:symmetry}. It also yields a formula for $\widetilde{D}[a,b,c,d;\ell,m]$, as a  sum involving multinomial coefficients. We have not tried to derive~\eqref{eq:symmetry} from it,  but this may  be possible.

For $a=b=0$, this involution would give a new bijective proof for the equi-enumeration of double-tandem walks in the quadrant ending anywhere and double-tandem walks in the upper half-plane ending on the $x$-axis (see the final remark in Section~\ref{sec:anywhere}). It would also 
extend the involution  of Proposition~\ref{prop:invol} for $p=1$ (by taking all steps in  $\{N,W,SE\}$).

We have not found any counterpart of this conjecture  for $p\geq 2$.
\end{remark}

%%%%%%%%%%%%%%%%%%%%%%%%%%%%%%%%%%%%%%%%%%%%%%%%%%%%%%%%%%%%%%%%
\section{Asymptotic enumeration}
\label{sec:asymptotic}
%%%%%%%%%%%%%%%%%%%%%%%%%%%%%%%%%%%%%%%%%%%%%%%%%%%%%%%%%%%%%%%%

This section is devoted to the proof of Theorem~\ref{thm:asympt}
on the asymptotic enumeration of $p$-tandem walks with prescribed starting point $(a,b)$ and endpoint $(c,d)$, with a  weight $z_r$ for face steps of level $r$ (for $0\leq r\leq p$). 
We proceed by a reduction to a random walk model with zero drift. We
first compute (due to the expression for $Q^{(a,b)}(1,1)$ given in Theorem~\ref{thm:alg}) 
an asymptotic estimate of the probability that a random walk starting from $(a,b)$ stays in   
the quadrant at least up to time $n$.   
Then we adapt recent results of Denisov and
Wachtel~\cite{denisov2015random} to our setting 
to derive an asymptotic estimate 
of the probability that a random walk starting at $(a,b)$ 
stays in the quadrant at least up to time $n$ {\emph{and}} ends at a prescribed
point $(c,d)$. The  reason why an adaptation is required is that the
results of Denisov and Wachtel {require} an aperiodicity condition
which does not hold for $p$-tandem walks.

An alternative approach would be  to apply analytic combinatorics in several variables (ACSV)~\cite{pemantle2013analytic} to the explicit expression for $Q^{(a,b)}(x,y)$ given in Proposition~\ref{prop:Qab} as the nonnegative part of an algebraic series; this would yield, in theory,  full asymptotic expansions, but these techniques are highly involved, especially with the complicated algebraic expression that we have. If $a=0$, a more convenient starting point would be Corollary~\ref{cor:Qthird}, where $Q^{(0,b)}(x,y)$ is expressed in terms of a \emph{rational}  function. For instance, starting from the
  expression~\eqref{Q0000} for $Q^{(0,0)}(0,0)$, Marni Mishna was able to work out the asymptotic number of
  excursions when  all face steps have level $p$, for $p=3$ and $p=4$, and it is possible that this could be
  extended to arbitrary $p$   (personal communication). We refer to~\cite{BoBoMe18,MelczerMishna2016,melczer-wilson} for recent applications of ACSV to the enumeration of walks confined to cones.

In Section~\ref{sec:weighted_asymp} we {state}  our probabilistic
results. {We then derive from them asymptotic estimates} for the  weighted number of tandem walks in Section~\ref{sec:weighted}.
{We finally prove the results of Section~\ref{sec:weighted_asymp}} in Sections~\ref{subsec:computing_V(a,b)} and~\ref{subsec:LLT}.

%===============================================================
\subsection{Random tandem walks and discrete harmonic functions}
\label{sec:weighted_asymp}
%======================================================
In this section, we fix a $(p+2)$-tuple  $(z,z_0,\ldots,z_p)$ of nonnegative reals (with $z_p>0$) such that
\beq\label{eq:dist}
z+\sum_{r=0}^p(r+1)z_r=1. 
\eeq
We define  a step distribution in $\zs^2$ by 
\beq\label{distri}
\PP((X,Y)=(i,j)) = 
\begin{cases}
  z, & \hbox{if } (i,j)=(1,-1),
\\
z_r, & \hbox{if } j=i+r \hbox{ and } 0\le j \le r,
\\
0, & \hbox{otherwise.}
\end{cases}
\eeq
To avoid trivialities, we assume throughout the section that $p\ge 1$.
We then consider the \emph{random tandem  walk} that starts at $(a,b)$ and
takes each step independently under the above distribution. The point
reached by the walk after $n$ steps is denoted by $S^{(a,b)}(n)$,
and we let $\tau^{(a,b)}\in\mathbb{N}\cup\{\infty\}$ denote the first time  
{that} the random walk exits the quadrant.   
The \emm drift, $(\EE(X), \EE(Y))$ of this walk is given by
\[
\EE(X)=-\EE(Y)=z-\sum_{r=0}^pz_r\sum_{i=0}^ri=z-\sum_{r=0}^pz_r\binom{r+1}{2}.
\]
Hence the {drift vanishes} if and only if
\beq\label{eq:drift}
z=\sum_{r=1}^pz_r\binom{r+1}{2}.
\eeq

\begin{Lemma}\label{lem:covari}
Under the zero-drift assumption~\eqref{eq:drift}, the covariance
matrix of the step distribution is  
\begin{equation}
\label{eq:covariance_matrix}
M=\left(\begin{array}{cc}\EE (X^2)&\EE (XY)\\\EE (XY)&\EE (Y^2)\end{array}\right)=\sigma^2\left(\begin{array}{cc}2&-1\\-1&2\end{array}\right),
\end{equation}
where 
\beq\label{eq:sigma}
\sigma^2=\sum_{r=1}^pz_r\binom{r+2}{3}.
\eeq
\end{Lemma}

\begin{proof}
Using~\eqref{eq:drift}, we compute
\[
\EE (X^2)=\EE (Y^2)=z+\sum_{r=0}^pz_r\sum_{i=0}^ri^2=\sum_{r=0}^pz_r\sum_{i=0}^r(i+i^2)=2\sum_{r=0}^pz_r\binom{r+2}{3}=2\sigma^2,
\]
and
\[
-\EE (XY)=z+\sum_{r=0}^pz_r\sum_{i=0}^ri(r-i)=
\sum_{r=0}^pz_r\sum_{i=0}^ri(r+1-i)=\sum_{r=0}^pz_r\binom{r+2}{3}=\sigma^2.
\]
\end{proof}

Our first result is an estimate for the probability that the walk   
remains in the quadrant until time $n$ at least.
\begin{Proposition}\label{prop:asympz_enum}
Let $a, b \ge 0$. Under the zero-drift assumption~\eqref{eq:drift}, 
\[
\mathbb{P}\left(\tau^{(a,b)}>n\right)\sim\frac1{4\sqrt{\pi}}\, V(a,b)\,n^{-3/2}
\qquad \hbox{ as }  n\to\infty,
\]
where the constants $V(a,b)$ have \gf \ 
\beq
\label{eq:expression_H}
{\mathcal V(u,v):=}\sum_{a,b\geq 0}V(a,b)u^av^b=\frac{2}{\sigma}\cdot\frac{1-uv}{(1-u)^3(1-v)^3\Lambda(u)},
\eeq
with
\[
 \Lambda(u)=\sum_{k=0}^{p-1}u^k\sum_{r=k+1}^pz_r\binom{r-k+1}{2}
\]
and $\sigma >0$ defined by~\eqref{eq:sigma}.
\end{Proposition}
We will prove Proposition~\ref{prop:asympz_enum} in
Section~\ref{subsec:computing_V(a,b)}. The proof consists in applying
Flajolet and Odlyzko's \emm singularity analysis,~\cite{flajolet-odlyzko,flajolet-sedgewick} to  the  expression for $Q^{(a,b)}(1,1)$ given in Theorem~\ref{thm:alg}. 

\begin{remarks}
{\bf 1.}  We admit that the factor $2$ in the expression for
    $V(a,b)$ looks strange in sight of the denominator $4$ occurring
    in the estimate of $\mathbb{P}\left(\tau^{(a,b)}>n\right)$. However, this is the    right convention in terms of the limit behaviour of $V(a,b)$. See the discussion in Section~\ref{sec:proba-harmonic}.

\medskip
\noindent
{\bf 2.} By considering the first step of a tandem walk
  starting at $(a,b)$, we see that the function $V(a,b)$ of Proposition~\ref{prop:asympz_enum} has to satisfy 
\beq
\label{eq:recurrence_relation_V}
     V(a,b)=z\ \!V(a+1,b-1)+\sum_{r=0}^pz_r\sum_{i=0}^rV(a-i,b+ r-i),
\eeq
 with the convention that $V(a,b)=0$ if $a<0$ or $b<0$.  In other words,
$V(a,b)$ is equal to the expected value of $V(\cdot, \cdot)$ over
the neighbours of $(a,b)$ after one random step, i.e.,
\begin{equation*}
     V(a,b)=\EE \left(V((a,b)+(X,Y) )\mathbbm 1_{\tau^{(a,b)}>1}\right).
\end{equation*}
Such a function is called a \emph{discrete harmonic function} for the walk model
(here  random $p$-tandem walks in the quadrant). We discuss more
aspects of the above property in Section~\ref{sec:Tutte}.
\medskip

\noindent{\bf 3.} Except in a few particular cases, it is very rare to obtain an explicit expression for a discrete harmonic function (or for its generating function). The most remarkable features of the above result
are the following:
\begin{itemize}
\item it deals with a random walk with large steps, thus going beyond the results of~\cite{Ra-14}  which only apply to walks with steps in $\{-1,0,1\}^2$,
\item the generating function~\eqref{eq:expression_H} is rational, and moreover its denominator factors as a product of two univariate polynomials: this implies  that $V(a,b)$ admits  a polynomial-exponential expression.
\end{itemize}
 Let us give two examples, with $z_r=0$ unless $r=p$, for some fixed $p$. Then the conditions~\eqref{eq:dist} and~\eqref{eq:drift} force
\[
  z_p= \frac{2}{(p+1)(p+2)} \qquad \hbox{and} \qquad z= \frac p{p+2},
\]
so that $\sigma^2=p/3$. If $p=1$, then $z=z_1=1/3$, and 
\[
  \mathcal V(u,v) =  \frac{6\sqrt 3\ (1-uv)}{(1-u)^3(1-v)^3},
\]
which gives 
\beq\label{Vab1}
  V(a,b)=3\sqrt 3 \ (a+1)(b+1)(a+b+2).
  \eeq
If $p=2$, then $z_2={1}/{6}$, $z=1/2$ and
\[
  \mathcal V(u,v) =  \frac{2\sqrt 6\ (1-uv)}{(1-u)^3(1-v)^3(1+u/3)},
\]
which gives
\beq
\label{eq:expression_V(a,b)_p=2}
V(a,b)=\frac {3\sqrt 6}4 (b+1) \left(
  (a+1)(a+b+2)+\frac a 2 +\frac b 4 + \frac 5 8
  -\frac {2b+1} 8 \left(-\frac{1}{3}\right)^{a+1}
\right).
\eeq
\\
{\bf 4.}  The form of the estimate in Proposition~\ref{prop:asympz_enum} follows from a general result of Denisov and
Wachtel~\cite[Thm.~1]{denisov2015random}, which implies that
for a zero-drift quadrant walk,  $\mathbb{P}(\tau^{(a,b)}>n)\sim \widetilde V(a,b)n^{-q/2}$ 
for some  discrete harmonic function $\widetilde V(a,b)$.
The exponent $q$ is ${\pi}/\arccos(-\rho)$, where $\rho$ is
 the \emm correlation factor,
\[
\rho:=\frac{\EE(XY)}{\sqrt{\EE(X^2)\EE(Y^2)}}.
\]
From Lemma~\ref{lem:covari} we see that $\rho=-1/2$ and $q=3$ for any
tandem step distribution. However, the results of~\cite{denisov2015random} do not seem
to yield any explicit expression for $\widetilde V(a,b)$ (nor of its \gf).\medskip
\\
{\bf 5.} Note that $\Lambda(u)$ has nonnegative coefficients,
  and $\Lambda(1)=\sigma^2$. Hence, as $a,b\to\infty$, $V(a,b)$ admits the asymptotic estimate
\beq\label{eq:asymptotic_V(a,b)}
V(a,b)\sim \frac{2}{\sigma}\cdot\frac1{\Lambda(1)}[u^a][v^b]\frac{1-uv}{(1-u)^3(1-v)^3}=\frac1{\sigma^3}(a+1)(b+1)(a+b+2)\sim\frac{ab(a+b)}{\sigma^3}.
\eeq
The asymptotic behaviour of $V(a,b)$ can be interpreted as a scaling limit of the discrete harmonic function; the limit function $V_\infty(a,b):=ab(a+b)$
is continuous harmonic in the classical sense:
\begin{equation*}
     \EE(X^2)\frac{\partial^{2}V_\infty}{\partial a^2}+2\EE(XY)\frac{\partial^{2}V_\infty}{\partial a\partial b}+\EE(Y^2)\frac{\partial^{2}V_\infty}{\partial b^2}=2\sigma^2\left(\frac{\partial^{2}V_\infty}{\partial a^2}-\frac{\partial^{2}V_\infty}{\partial a\partial b}+\frac{\partial^{2}V_\infty}{\partial b^2}\right)=0,
\end{equation*}
where we have used~\eqref{eq:covariance_matrix}. 
In Section~\ref{sec:proba-harmonic}, we discuss 
the connections between $V$ and  $V_\infty$ further.
\end{remarks}

We now turn to the probability that  $S^{(a,b)}(n)$ reaches the point
$(c,d)$ at time $n$, without having ever left the quadrant.
Recall that this is only possible if $2n\equiv (c-d)-(a-b)$~mod~$\iota$, where  the \emph{periodicity index}
$\iota$ is defined by~\eqref{def:iota}. 

\begin{Proposition}\label{prop:local_limit}
Let $a,b,c,d\in\mathbb{N}$.  
Under the zero-drift assumption~\eqref{eq:drift}, for 
$2n\equiv (c-d)-(a-b)$~{\em mod}~$\iota$ we have
\[
\PP\left(S^{(a,b)}(n)=(c,d),\ \tau^{(a,b)}> n\right) \sim
\frac{\iota}{4\sqrt{3}\pi\sigma^2}\frac{V(a,b)V(d,c)}{n^{4}}
\qquad \hbox{as } n\to\infty,
\]
where $\sigma^2$ is given by~\eqref{eq:sigma}, 
$\iota$ is the periodicity index, and 
 $V(a,b)$ is the discrete harmonic function of Proposition~\ref{prop:asympz_enum}. 
\end{Proposition}
Again, the form of the above estimate resembles a general formula of
Denisov and Wachtel~\cite[Thm.~6]{denisov2015random}. However, our
result is more precise  because all constants are explicit, and moreover
Theorem~6 in~\cite{denisov2015random} requires a \emm strong
aperiodicity, assumption, which does not hold in general for tandem walks (Lemma~\ref{lem:per}). We will explain how to adapt the main arguments of~\cite{denisov2015random} to our periodic walk in Section~\ref{subsec:computing_V(a,b)}. This is also briefly discussed in~\cite[p.~3]{duraj-wachtel}.

%==========================================================
\subsection{Asymptotic enumeration of weighted tandem walks}
\label{sec:weighted}
%==============================================================
Our aim here is to derive the asymptotic result of Theorem~\ref{thm:asympt}
from the above probabilistic results, and to make all constants in
this proposition explicit. For convenience, we change the
  notation, and replace
the weights $z_0, \ldots, z_p$ of Theorem~\ref{thm:asympt}
by $w_0,\ldots,w_p$ (SE steps have weight~1).
The \emph{weight} of a $p$-tandem walk having $n_r$
face steps of level $r$ for $0\leq r\leq p$ 
is defined as $\prod_{r=0}^p w_r^{n_r}$; furthermore
we denote the weighted ``number`` of walks of length $n$ staying in the quadrant, 
starting at $(a,b)$ and ending at $(c,d)$, by $q_n(a,b;c,d)$. 
With the notation
introduced at the beginning of Section~\ref{sec:counting},
\[
q_n(a,b;c,d)= [t^n x^c y^d]Q^{(a,b)}(x,y)= [t^n ]Q^{(a,b)}_{c,d}.
\]
By a suitable normalization, we will now 
relate $q_n(a,b;c,d)$ to a probability of the form $P(S_{n}^{(a,b)}=(c,d),\ \tau^{(a,b)}>n)$, 
as considered in Proposition~\ref{prop:local_limit}. 
For two positive parameters~$\alpha,\gamma$ {(to be fixed later)}, we let 
\beq\label{eq:normalize}
z= \alpha^2/\gamma,\ \ z_r=w_r\alpha^{-r}/\gamma\quad \mathrm{for }\ 0\leq r\leq p.
\eeq
As in Section~\ref{sec:weighted_asymp}, we want these values
   to describe the step distribution of a random walk with zero drift.
Then  the probability of a step $(i,j)$ will be its weight, multiplied by $\alpha^{i-j}/\gamma$. 
The zero-drift condition~\eqref{eq:drift} is satisfied if and
only if 
\beq\label{eq:alpha}
\alpha^{2}=\sum_{r=1}^p\binom{r+1}{2}w_r\alpha^{-r}. 
\eeq
This equation {in $\alpha$} has a unique solution;   
indeed, as $\alpha$ increases from $0$ to $+\infty$, the left-hand side
increases from $0$ to $+\infty$, while the right-hand side decreases from $+\infty$ to $0$. 
Once $\alpha$ has been adjusted,  the normalization condition~\eqref{eq:dist} forces 
\beq\label{eq:gamma} 
\gamma=\alpha^{2}+\sum_{r=0}^p(r+1)w_r\alpha^{-r}=\sum_{r=0}^p\binom{r+2}{2}w_r\alpha^{-r}. 
\eeq
With this choice of $\alpha$ and $\gamma$,
the values $z,z_0,\ldots,z_p$ define {indeed} a probability distribution on
tandem steps, having zero drift. Now the probability of a quadrant walk
of length $n$ going  from $(a,b)$ to $(c,d)$ is equal to its weight,
multiplied by $\alpha^{(c-a)-(d-b)}/\gamma^n$. 
Hence,
\[
\mathbb{P}\big(S^{(a,b)}(n)=(c,d),\ \tau^{(a,b)}>
n\big)=\frac{\alpha^{(c-a)-(d-b)}}{\gamma^n}\, q_n(a,b;c,d).
\]
Theorem~\ref{thm:asympt} now follows from Proposition~\ref{prop:local_limit}.

\begin{Theorem*}
Let $a,b,c,d$ be nonnegative integers and let $w_0,\ldots,w_p$ be
nonnegative weights with $w_p>0$. Define $\alpha >0$ and $\gamma$ by
\eqref{eq:alpha} and~\eqref{eq:gamma}. Let
\[
D=\{r\in\llbracket 0, p\rrbracket,\ w_r>0\} \qquad \hbox{ and } \qquad\iota=\gcd(r+2,\
r\in D\}.
\] 
Then, as $n\to\infty$ conditioned on $c-d\equiv a-b +2n$~{\em mod}~$\iota$, we have 
\beq\label{eq:asympt}
q_n(a,b;c,d)\sim \kappa\ \!\gamma^nn^{-4},
\eeq
where \[
\kappa:=\frac{\iota}{4\sqrt{3}\pi\sigma^2}V(a,b)V(d,c)\alpha^{(d-b)-(c-a)}, 
\]
with  $V(\cdot, \cdot)$  the harmonic function of
 Proposition~\ref{prop:asympz_enum} and $\sigma^2$ given
 by~\eqref{eq:sigma}, both taken with $z_r=w_r\alpha^{-r}/\gamma$.
\end{Theorem*}

We can now go back to the number of bipolar orientations with
prescribed face degrees.
\noindent\begin{proof}[Proof of Corollary~\ref{coro:bipolar}]
It follows from  the KMSW  bijection that
 $B_n^{(\Omega)}(b,c)=q_n(0,b;c,0)$, taken for $w_r=1$ if $r+2\in
 \Omega$, and $w_r=0$ otherwise.  With the notation of  the above
 proposition, this gives $D=\Omega-2$,
\[
1= \sum_{s\in \Omega} {\binom {s-1} 2} \alpha^{-s}
\qquad\hbox{and}\qquad \gamma= \sum_{s\in \Omega} {\binom s 2}
  \alpha^{-s+2},
\]
which fits with the values of $\alpha$ and $\gamma$ given in
  Corollary~\ref{coro:bipolar}. One easily checks that the
  value~\eqref{eq:sigma} of $\sigma^2$ is also in agreement with the corollary. 
Then we only need to determine the values $V(0,b)$ and $V(0,c)$. By Proposition~\ref{prop:asympz_enum},
\[
\mathcal V(0,v)=\sum_{i\ge 0} V(0,i)v^i= \frac 2 \sigma \frac 1 {(1-v)^3 \Lambda(0)},
\]
and, by~\eqref{eq:drift} and then~\eqref{eq:normalize},
\[
\Lambda(0)= \sum_{r=1}^ pz_r{\binom {r+1} 2}=z=\frac{\alpha^2}{\gamma}.
\] 
Hence
\[
  V(0,i)=\frac{1}{\sigma\Lambda(0)}(i+1)(i+2)=
   \frac \gamma{\sigma \alpha^2}  (i+1)(i+2).
\]
The proof is then completed by putting all pieces together.
\end{proof}

%=============================================================
 \subsection{The probability to stay in the quadrant till time
   $\boldsymbol n$: proof}
\label{subsec:computing_V(a,b)}
%==========================================================
We return to the probabilistic setting of
Proposition~\ref{prop:asympz_enum}, where we consider a zero-drift random
tandem walk with step distribution given by~\eqref{distri}.  The
probability that the walk $S^{(a,b)}$ remains in the quadrant till
time $n$ is closely related to the coefficient of $t^n$ in the series
$Q^{(a,b)}(1,1)$ given in Theorem~\ref{thm:alg}.  However,  we
need to incorporate 
a positive weight~$z$ for each SE step: this amounts to replacing $t$ by $zt$
and $z_r$ by $z_r/z$ for $0\leq r\leq p$. Let us emphasize the
dependence of $Q^{(a,b)}$ in the variables $t$ and $z_r$ by writing
\[
  Q^{(a,b)}(x,y)\equiv Q^{(a,b)}(t, \gz; x,y),
\]
with $\gz=(z_0, \ldots, z_p)$.
Then 
\[
\PP\left(\tau^{(a,b)}>n\right)=[t^n]Q^{(a,b)}(tz, \gz/z; 1,1).
\]
  It now follows from Theorem~\ref{thm:alg} that
\beq\label{eq:QabW}
\sum_{n\ge 0}\PP\left(\tau^{(a,b)}>n\right)t^n
=\frac{\Wb}{t}\sum_{i=0}^a \Ab_i\sum_{j=0}^b\Wb^j,
\end{equation}
where $\Wb=W(tz, \gz/z)$ (see~\eqref{W-eq})
is the unique series in $t$ satisfying  $\Wb=t\phi(\Wb)$, with 
\beq\label{phi-def}
\phi(w)=z+\sum_{r=0}^pz_r(w+\cdots+w^{r+1}),
\eeq
and $\Ab_i= \frac 1 z A_i(tz, \gz/z)$ is a polynomial in $\Wb$:
\[
  \Ab_i= [u^i]\frac1
  {z -u \Wb \sum_{i+k<r\le p} u^i \Wb ^k z_r}.
\]

Starting from~\eqref{eq:QabW}, we observe that
Proposition~\ref{prop:asympz_enum} follows by linearity if we can prove
the following lemma.
\begin{Lemma}
  For $i, j \ge 0$, we have, as $n$ tends to infinity,
\[ 
[t^n] \frac{\Wb}{t}\Ab_i\Wb^j\sim \frac 1  {4\sqrt{\pi}}\, U(i,j) n^{-3/2},
\] 
where
\[
\sum_{i,j\ge 0}U(i,j) u^i v^j = \frac 2 \sigma \cdot
\frac{1-uv}{(1-u)^2(1-v)^2 \Lambda(u)},
\]
with $\sigma$ and $\Lambda(u)$ as in Proposition~\ref{prop:asympz_enum}.
\end{Lemma}
\begin{proof}
  We will prove this lemma   using Flajolet and Odlyzko's 
\emm singularity analysis,~\cite{flajolet-odlyzko,flajolet-sedgewick}.

Our first task  is to determine
the dominant singularities, and the singular behaviour of $\Wb\equiv \Wb(t)$. The equation $\Wb=t\phi(\Wb)$,
with $\phi$ defined by~\eqref{phi-def}, fits in the \emm smooth
aperiodic inverse function schema, of~\cite[Thm.~VII.2,
p.~453]{flajolet-sedgewick} (see also Thm.~IV.6 on p.~404 in the same
reference; aperiodicity comes for instance from the term $z_pw$ in
 the transformation~$\phi$). Consequently, $\Wb$ has a unique
singularity  $\tsing >0$ on its circle of convergence, and $\tsing$ is
the unique positive solution of $\phi(\tsing)= \tsing
\phi'(\tsing)$. Recall that $z$ and the $z_i$'s satisfy the
normalization condition~\eqref{eq:dist}, and that we  assume that
the zero-drift condition~\eqref{eq:drift} holds. This means that
$\phi(1)=\phi'(1)=1$, so that the radius of $\Wb$ is
$\tsing=1$. Moreover, $\Wb(\tsing)=t_c=1$ as well. Still using the above
cited results
of~\cite{flajolet-sedgewick}, we conclude that $\Wb$ admits  a square-root singular expansion around $t_c=1$, namely
\[
\Wb=1-d\sqrt{1-t}+O(1-t),
\]
where $d=\sqrt{2/\phi''(1)}=1/\sigma$, 
with $\sigma^2$ given by~\eqref{eq:sigma}. Consequently, for any
$k\geq 1$, the series $\Wb^k$ also has a unique singularity on its
circle of convergence, and
\[
\Wb^k=1-\frac k{\sigma}\ \sqrt{1-t}+O(1-t),
\]  
and more generally, for any  polynomial $P(w)$ having
nonnegative coefficients,
\beq\label{Pgen}
P(\Wb)=1-\frac1{\sigma}P'(1)\sqrt{1-t}+O(1-t).
\eeq
Recall that  the series $\Ab_i \Wb^{j+1}$ can be expressed as a polynomial in
  $\Wb$. More precisely, $\Ab_i \Wb^{j+1}= P_{i,j}(\Wb)$, where
\[
  P_{i,j}(w):=[u^i v^j] \left(A(u,w) B(v,w)\right),
\]
with
\[
A(u,w):=\frac{1}{z-uw\sum_{i+k<r\le p}u^i w^k z_r}\qquad
\hbox{and } \qquad B(v,w):=\frac{w}{1-vw}.
\]
It then follows from~\eqref{Pgen}  that
\[
  \Ab_i \Wb^{j+1}  =1-\frac1{\sigma}\ \! P_{i,j}'(1) \sqrt{1-t}+O(1-t),
\]
with
\beq\label{Pij}
  P_{i,j}'(1) =[u^i v^j]\left. \frac{\partial \left(A(u,w) B(v,w)\right)}{\partial     w} \right|_{w=1}.
\eeq
Since  $\tsing=1$, we have the same singular expansion for $\Ab_i \Wb^{j+1}/t$, and thus,  using the transfer theorem of~\cite[Cor.~VI.1,
p.~392]{flajolet-sedgewick}, we find
\beq\label{ij-asympt}
[t^n]\frac 1 t \Ab_i \Wb^{j+1}\sim
\frac1{2\sigma\sqrt{\pi}}\,  P_{i,j}'(1) n^{-3/2}
\sim \frac1{4\sqrt{\pi}}\, U(i,j)n^{-3/2}, 
\eeq  
where we define
\beq\label{Uijdef}
U(i,j):=2\,P_{i,j}'(1)/\sigma.
\eeq
It remains to express $P_{i,j}'(1)$, that is, the derivative
in~\eqref{Pij}. Clearly, we have
\[
B(v,1)=\frac{1}{1-v}\qquad \hbox{and} \qquad
B'_2(v,1)=\frac{1}{(1-v)^2}.
\]
Moreover,
\[
  A(u,1)=\frac{1}  {z-u\sum_{0\le i<r\le p}u^i z_r (r-i)}.
\]
Upon rewriting $z$ as in~\eqref{eq:drift}, this gives
\[
   A(u,1)=\frac{1}{\sum_{0\le i<r\le p}(1-u^ {i+1}) z_r(r-i)}
 =\frac 1{(1-u)\Lambda(u)} ,
  \]
where $\Lambda(u)$ is defined as in Proposition~\ref{prop:asympz_enum}. 
Finally
\[
  A'_2(u,1)=A(u,1)^2\, u \left(\sum_{i+k <r \le p} u^i(k+1) z_r\right)
    =A(u,1)^2u \Lambda(u)=\frac{u }{(1-u)^2\Lambda(u)}.
\]
Hence, getting back to~\eqref{Pij} and~\eqref{Uijdef}, we have
\begin{align*}
  P'_{i,j}(1)= \frac \sigma 2 \cdot U(i,j)
  &= [u^i v^j]   \left( A(u,1) B'_2(v,1) + A'_2(u,1) B(v,1)\right)\\
&  = [u^i v^j]
  \frac{1-uv}{(1-u)^2(1-v)^2 \Lambda(u)},
\end{align*}
as claimed in the lemma.
\end{proof}

%===========================================================================
\subsection{The local limit theorem: proof}
\label{subsec:LLT}
%============================================================================
We now  explain how to go from Proposition~\ref{prop:asympz_enum} to
Proposition~\ref{prop:local_limit}. We follow the steps
in~\cite{denisov2015random}, with which some familiarity is assumed,
and explain how to deal with periodicity issues. We refer to Spitzer's
book~\cite{Spitzer} for several classical results on random walks.

Let  $S(n)$ denote the point reached after $n$ steps of the random tandem walk (starting at the origin). Recall that $\Sab(n):=(a,b)+S(n)$ denotes
the point attained after $n$ steps, when starting from $(a,b)$.

\subsubsection{Local limit theorem for unconstrained walks} 
Lemma~\ref{lem:covari} gives  the covariance matrix of the step distribution. By the central limit theorem (in its
vectorial formulation), the random variable
$\frac1{\sigma\sqrt{n}}S(n)$ converges in law  
to the random variable on $\mathbb{R}^2$ of 
density 
$$
f(x,y)=\frac1 {2\pi \sqrt 3}
\exp\big(-\tfrac1{3}(x^2+y^2+xy)\big).
$$ 
 Obviously, the same limit holds for the random variable $\frac1{\sigma\sqrt{n}}\Sab(n)$,
for any fixed starting point $(a,b)$.  
Let us now state the corresponding Gnedenko local limit theorem.
Let $\Rab(n)$ be the sublattice 
\beq\label{Rab}
\Rab(n)=\{ (i,j)\in\mathbb{Z}^2: i-j\equiv 2n+(a-b) \!\mod \iota\}.
\eeq
Note that it only depends on $n$ through $n$~mod~$\iota$. According to Lemma~\ref{lem:per},  any point reachable from $(a,b)$ in $n$ steps lies in  $\Rab(n)$, 
and any  point in $\Rab(n)$ can be reached in $n$ steps starting from $(a,b)$, for $n$ large enough.
 Then, with $(a,b)$ fixed and $n$ tending to infinity,
\beq\label{eq:local_limit}
\underset{(i,j)\in \Rab(n)}{\sup}\left\vert n\cdot
\boldP\bigl(\Sab(n)
=(i,j)\bigr)-\frac{\iota}{\sigma^2}\,f\!\left(\frac{i}{\sigma\sqrt{n}},\frac{j}{\sigma\sqrt{n}}\right) \right\vert\to 0.
\eeq
The proof is classically done by a saddle-point argument,
after diagonalization of the covariance matrix 
(we refer to Proposition~P9 in~\cite[pp.~75--77]{Spitzer}, and the
remark that follows it); 
moreover, as discussed in Example~2 in~\cite[pp.~78--79]{Spitzer},   
 a periodicity $\iota\geq 2$ of the step set results in $\iota$
saddle-points, each giving the same asymptotic contribution, with the effect that the asymptotic constant is multiplied by~$\iota$
for points of the reachable sublattice.

\subsubsection{Local limit theorem for  walks confined in the quadrant} 
We begin with the following central limit theorem under the quadrant constraint.
 
\begin{Proposition}\label{prop:Sabconve}
The random variable $\frac1{\sigma\sqrt{n}}\Sab(n)$
conditioned on $\{\tau^{(a,b)}>n\}$  converges in law
to the random variable on $\mathbb{R}_+^2$ of density
\beq\label{eq:g}
g(x,y)=\frac1{\sqrt{3\pi}}\,xy(x+y)\exp\big(-\tfrac1{3}(x^2+y^2+xy)\big).
\eeq
\end{Proposition}

 There is a natural link between the function  $xy(x+y)$
   occurring in the density above, and the limit $V_\infty$ of the
   discrete harmonic function (see~\eqref{eq:asymptotic_V(a,b)}), which we discuss further    in Section~\ref{sec:proba-harmonic}.

\begin{proof} We use a normalization that transforms the random walk
  $S^{(a,b)}(n)$ into a walk with uncorrelated  $x$- and
  $y$-projections. This also transforms the quarter plane into a
  different cone, to which we then apply 
  \cite[Thm.~3]{denisov2015random}. This classical argument has  been used recently
  in a similar context in~\cite[Sec.~1.5]{denisov2015random}
  and~\cite[Thm.~4]{BoRaSa14}.

Let us now give details. Let $\cL$ be the linear mapping of matrix
\begin{equation*}
  \frac{\sqrt{2}}{\sigma\sqrt{3}}\left(
    \begin{array}{rr} 1&\frac{1}{2}\smallskip\\0&\frac{\sqrt{3}}{2}
    \end{array}\right).
\end{equation*}
 Then it is easy to check that if $(X,Y)$ is the step distribution of
our random tandem walk, given by~\eqref{distri}, then the covariance matrix
of $\cL(X,Y)$ is the identity.
Moreover,~$\cL$ maps the quadrant $\{re^{i\theta}: r\geq0 \text{ and }\theta\in
[0,\pi/2]\}$ to the cone $K_{\pi/3}=\{re^{i\theta}: r\geq0 \text{ and
}\theta\in [0,\pi/3]\}$. Let $\tiSab:=\cL(\Sab)$ denote the transformed walk.
By~\cite[Thm.~3]{denisov2015random}, 
$\tiSab(n)/\sqrt n$,  conditioned on $\{\tau^{(a,b)}>n\}$,
  converges in law to the random   variable $\widetilde{S}$
  having the following density (expressed in polar coordinates): 
\begin{equation*}
     \rho(r,\theta) = H_{0}\cdot \mathbbm{1}_{\theta\in[0,\pi/3]}\cdot  u(r,\theta)
  \cdot r   \exp(-r^2/2),
\end{equation*}
where $H_0$ is the normalizing constant and the function $u(r,\theta)$ is given by \cite[Eq.~(3)]{denisov2015random}. This function  may be
interpreted as the unique harmonic function that is nonnegative
% positive
in the cone $K_{\pi/3}$ and vanishes on the boundary.

Let us now show that $u(r,\theta)=r^3 \sin(3\theta)$. By Eq.~(3) in~\cite{denisov2015random}, we need to solve the eigenvalue problem stated in Eq.~(2) of that same paper. In dimension $d=2$, the Laplace--Beltrami
operator $L_{\mathbb S^{d-1}}$ involved in this problem  is simply
$\frac{\partial^2}{\partial\theta^2}$, and the eigenvalue problem
becomes $\frac{\partial^2m_j}{\partial\theta^2}=-\lambda_jm_j$. This
is easily solved in
$m_j(\theta)=a\cos(\sqrt{\lambda_j}\theta)+b\sin(\sqrt{\lambda_j}\theta)$
for arbitrary constants $a$ and $b$. The boundary conditions
$m_j(0)=m_j(\pi/3)=0$ yield $a=0$ and force $\lambda_j=(3j)^2$ (as in~\cite{denisov2015random}, we take $0<\lambda_1 < \lambda_2 \le \cdots$). In
particular, $\lambda_1=9$, the function $m_1$ reads $m_1(\theta)=b\sin(3\theta)$, which together with [28,
Eq.~(3)] proves that (up to a multiplicative constant)
$u(r,\theta)=r^3 \sin(3\theta)$. The value of the normalizing constant
$H_0$ in the above expression for $\rho(r, \theta)$ is then  found to be $1/{\sqrt{2\pi}}$, since $\rho$ must be a density.

In $(x,y)$-coordinates, we have
\[
  u(r,\theta)=r^3\sin(3\theta)  =y(3x^2-y^2),
\] 
hence the density of $\widetilde S$, expressed in Cartesian coordinates, is
\[
\widetilde{g}(x,y)=\frac{1}{\sqrt{2\pi}}\cdot\mathbbm{1}_{(x,y)\in K_{\pi/3}}\cdot y(3x^2-y^2)\cdot\exp\bigl(-\tfrac{1}{2}(x^2+y^2)\bigr).
\]
Getting back to $S^{(a,b)}(n)$ and  the quadrant,
we conclude that $\frac1{\sigma\sqrt{n}}\Sab(n)$ conditioned on $\{\tau^{a,b}>n\}$ converges to the random variable $\frac1{\sigma}\cL^{-1}(\widetilde{S})$, and by
a routine change of variable we find that it has density $g(x,y)$.
\end{proof} 

Let us now state the corresponding  Gnedenko
 local limit theorem, that is, the counterpart of~\eqref{eq:local_limit}. In the aperiodic case this would be Theorem~5  in~\cite{denisov2015random}.

 \begin{Proposition}\label{prop:local-cone}
Let $\Rpab(n):=\Rab(n)\cap \mathbb{N}^2$.   For $a,b$ fixed and $n$ tending to infinity, we have
\[
\underset{(i,j)\in \Rpab(n)}{\sup}\left| n^{5/2}\cdot\boldP\big(\Sab(n)=(i,j),\ \tau^{(a,b)}>n\big)-\frac{\iota V(a,b)}{4\sqrt{\pi}\sigma^2}\,g\!\left(\frac{i}{\sigma\sqrt{n}},\frac{j}{\sigma\sqrt{n}}\right) \right|\to 0.
\]
\end{Proposition}
\begin{proof}[Proof (sketch)]
 Given the estimate of $\PP(\tau^{(a,b)} >n)$ given in  Proposition~\ref{prop:asympz_enum}, the above proposition is equivalent to
\[
\underset{(i,j)\in \Rpab(n)}{\sup}\left|
  n\cdot\boldP\big(\Sab(n)=(i,j)\ |\ \tau^{(a,b)}>n\big)-
\frac{\iota}{\sigma^2}\, g\!\left(\frac{i}{\sigma\sqrt{n}},\frac{j}{\sigma\sqrt{n}}\right) \right|\to 0.
\]
The proof mimics the proof
  of~\cite[Thm.~5]{denisov2015random}. In the typical case where $i$ and $j$ are of the order of $\sqrt n$, it  proceeds
by a well chosen splitting of quadrant walks from $(a,b)$ to $(i,j)$
into two parts;
% [authors] We fear that we have lost precision. What we wanted to say is that there are two parts in the walk, and that we use one argument for the first part, and the other argument for the second part.
% Christian: [...] So, if you want to say that, to the first part you apply <argument 1>,
% and to the second part you apply <argument 2>, then say it!
% I leave it to you to do the appropriate reformulation.
% [authors] done
to the first part, we apply  
Proposition~\ref{prop:Sabconve} (convergence in law of the constrained  
walk to the density~\eqref{eq:g}), and to the second part, we apply
the domain-unconstrained local limit theorem~\eqref{eq:local_limit}.
%applies Proposition~\ref{prop:Sabconve} (convergence in law of the constrained walk to the density~\eqref{eq:g}),    and  the domain-unconstrained local limit theorem~\eqref{eq:local_limit}.
The  factor  $\iota$  propagates from~\eqref{eq:local_limit} to the final result.
\end{proof}

\subsubsection{Conclusion}
We now derive  Proposition~\ref{prop:local_limit}
from Proposition~\ref{prop:local-cone}, as Theorem~6 is derived
from Theorem~5 in~\cite{denisov2015random}. For $n\geq 0$ let 
$n_1:=\lfloor n/2\rfloor$ and $n_2:=\lceil n/2 \rceil$, so that $n=n_1+n_2$. 
For two points $(a,b)$ and $(i,j)$ in the quadrant, let $\cE^{(a,b)}_{i,j}(n)$ 
denote the event $\{\Sab(n)=(i,j),\ \tau^{(a,b)}>n\}$. 
Observe that the step distribution~\eqref{distri} is such that the laws
of $(-X,-Y)$ and $(Y,X)$ coincide. In other words,  time-reversal  has
the same effect as exchanging the roles of $x$ and $y$. Hence, for any
quadrant points $(a,b)$, $(c,d)$ and $(i,j)$, we have
\begin{align*}
 \PP\big(\cE^{(a,b)}_{c,d}(n),\
  \Sab(n_1)=(i,j)\big)&=\PP\big(\cE^{(a,b)}_{i,j}(n_1)\big)\cdot\PP\big(\cE^{(i,j)}_{c,d}(n_2)\big)
  \\
  &=
\PP\big(\cE^{(a,b)}_{i,j}(n_1)\big)\cdot\PP\big(\cE^{(d,c)}_{j,i}(n_2)\big).
\end{align*}
We now assume that $(c,d)\in R^{(a,b)}_\ge(n)$, and sum the above identity
  over all $(i,j)$ in $R^{(a,b)}_\ge(n_1)$. In the first line below, we
  use Proposition~\ref{prop:local-cone}  to estimate the
  probability of the two $\cE$-events, and in the second line,  we use
  the symmetry of $g(x,y)$ in $x$ and $y$. This yields
\begin{align}
\PP\left(\cE^{(a,b)}_{c,d}(n)\right)&\sim \underset{(i,j)\in
                                     \Rpab({n_1})}{\sum}\frac{\iota
                                     V(a,b)}{4\sqrt{\pi}n_1^{5/2}\sigma^2}g\left(\frac{i}{\sigma\sqrt{n_1}},\frac{j}{\sigma\sqrt{n_1}}
                                     \right)\nonumber
\cdot \frac{\iota
  V(d,c)}{4\sqrt{\pi}n_2^{5/2}\sigma^2}\, g\left(\frac{j}{\sigma\sqrt{n_2}},\frac{i}{\sigma\sqrt{n_2}}\right)\\[.2cm]
&\sim V(a,b)V(d,c)\frac{2\iota^2}{\pi n^5\sigma^4}\underset{(i,j)\in \Rpab({n_1})}{\sum}g\left(\frac{i}{\sigma\sqrt{n/2}},\frac{j}{\sigma\sqrt{n/2}}\right)^2,\label{Est}
\end{align}
since $n_1\sim n_2 \sim n/2$.
Then we classically approximate the sum by an integral, using the
description~\eqref{Rab} of $\Rab(n)$,
$$
\frac{\iota}{(\sigma\sqrt{n/2})^2}\underset{(i,j)\in \Rpab(n_1)}{\sum}g\left(\frac{i}{\sigma\sqrt{n/2}},\frac{j}{\sigma\sqrt{n/2}}\right)^2\sim \int\!\!\!\int_{(x,y)\in\mathbb{R}_+^2}g(x,y)^2 \dd x\ \!\dd y=\frac1{4\sqrt{3}}.
$$
Returning to~\eqref{Est}, this gives
$$
\PP\left(\cE^{(a,b)}_{c,d}(n)\right)= \PP\left(\Sab(n)\!=\!(c,d),\ \tau^{(a,b)}\!>\!n\right)\sim \frac{\iota}{4\sqrt{3}\pi\sigma^2}\frac{V(a,b)V(d,c)}{n^4},
$$ 
as stated in Proposition~\ref{prop:local_limit}. \qed

%%%%%%%%%%%%%%%%%%%%%%%%%%%%%%%%%%%%%%%%%%%%%%%%%%%%%%%%%%%%%%
\section{Complements and final comments}
\label{sec:final}
%%%%%%%%%%%%%%%%%%%%%%%%%%%%%%%%%%%%%%%%%%%%%%%%%%%%%%%%%%%%%%%%ù

%=================================================
\subsection{Link with non-intersecting triples of directed walks}
 \label{sec:Baxter}
 %=================================================

It is  known that plane bipolar orientations can be 
encoded by certain non-intersecting
 triples of \emm directed, walks, that is, walks on the square
 lattice consisting of north and east steps (Figure~\ref{fig:triple}, right)~\cite{AlPo15,bonichon-mbm-fusy,fusy-bipolar,felsner2011bijections}. 
We explain here how such a bijection can be obtained from 
the KMSW  bijection with tandem walks.  In fact, our construction
is more general, as it extends to \emm marked, bipolar orientations. 
In particular, it gives a closed form expression for the number of marked bipolar orientations with prescribed signature,  number of (plain) edges, and  number of inner faces (see Proposition~\ref{prop:count_marked} and the paragraph just following it). However, this bijection does not allow us to record face degrees.

Consider a tandem walk $w$ of length $n$ going from $(a,b)$ to $(c,d)$,
staying in the quadrant, with successive steps $s_1,\ldots,s_n$.
We associate with $w$ a triple of directed walks $D_1,D_2,D_3$ 
as follows. We initialize $D_1,D_2,D_3$ to be the empty walks 
starting (and ending) at $(0,-a-1)$, $(0,0)$, and $(0,b+1)$,  respectively. Then we  let $D_1,D_2,D_3$ grow by reading the successive steps of $w$. 
To be precise, for $m$ from $1$ to $n$:
\begin{itemize}
\item
If $s_m$ is a SE  step, we add a north step to $D_2$, and leave
$D_1$ and $D_3$ unchanged.
\item
If $s_m$ is a face step $(-i,j)$ we append an east step to $D_2$,
and we append the walk $EN^i$ to $D_1$ and append the walk $N^jE$
to $D_3$. 
\end{itemize}

\begin{figure}[b]
\begin{center}
\includegraphics[width=8cm]{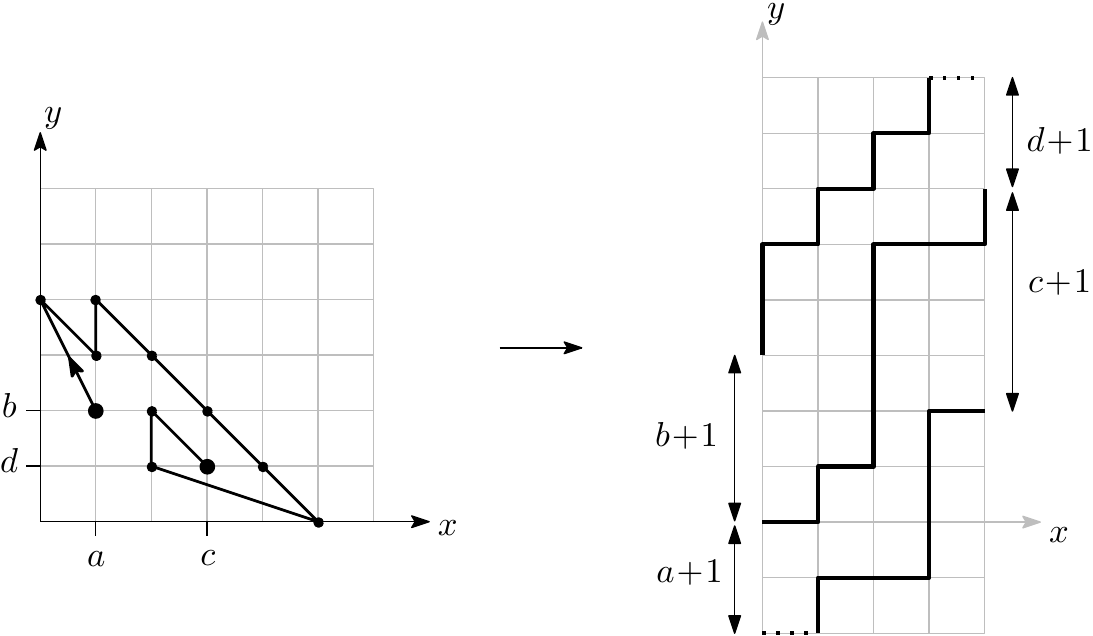}
\end{center}
\caption{A tandem walk in the quadrant, and 
the corresponding non-intersecting triple of directed lattice walks.}
\label{fig:triple}
\end{figure}

\noindent See Figure~\ref{fig:triple} for an example.
One can see that at stage $m$  the walks $D_1,D_2,D_3$
have the same number of east steps (which is 
the number of face steps among 
$s_1,\ldots,s_m$), and $D_2$ has exactly $m$ steps. 
In addition, if we denote by $(i_m,j_m)$ 
the point of $w$ reached after $m$ steps (that is,
  $(i_m,j_m)=s_1+\cdots + s_m$),  
and by $y_1^{{(m)}},y_2^{{(m)}},y_3^{{(m)}}$  the ordinates of the endpoints of $D_1,D_2,D_3$ at stage $m$, then $y_2^{{(m)}}-y_1^{{(m)}}=i_m+1$ and
$y_3^{{(m)}}-y_2^{{(m)}}=j_m+1$. In particular,
  saying that $w$ stays in the quadrant is equivalent to saying that
  the three walks $D_1, D_2, D_3$ do not intersect.
By construction, the first step of $D_1$ (respectively the last
step of $D_3$) is necessarily an east step, hence can be deleted
without loss of information. We denote by $(\hat D_1, D_2, \hat
  D_3)$ the resulting triple of walks. This gives the following proposition.

\begin{Proposition}\label{prop:count_marked}
The above mapping is a bijection $\Psi$ between tandem walks
of length $n$ with~$k$ face steps, staying in the quadrant,  
starting at $(a,b)$ and ending at $(c,d)$, 
and non-intersecting triples of directed walks from $A_1=(1,-a-1)$,
$A_2=(0,0)$, $A_3=(0,b+1)$ to $B_1=(k,n-k-c-1)$, $B_2=(k,n-k)$, $B_3=(k-1,n-k+d+1)$. 

Therefore, by the Lindstr\"om--Gessel--Viennot lemma~\cite{gessel-viennot}, 
the number of such walks is given by 
\[
q_{n,k}(a,b,c,d)=\left|
\begin{array}{ccc}
\binom{n+a-c-1}{k-1}&\binom{n+a}{k-1}&\binom{n+a+d}{k-2}\\[.2cm]
\binom{n-c-1}{k}&\binom{n}{k}&\binom{n+d}{k-1}\\[.2cm]
\binom{n-b-c-2}{k}&\binom{n-b-1}{k}&\binom{n-b+d-1}{k-1}
 \end{array} \right|.
\] 
\end{Proposition}

As discussed at the beginning of Section~\ref{sec:counting}, the number
\begin{multline*}
\tq_{n,k}(a,b,c,d):=q_{n,k}(a,b,c,d)-q_{n,k}(a-1,b,c-1,d)\\
-q_{n,k}(a,b-1,c,d-1)+q_{n,k}(a-1,b-1,c-1,d-1) 
\end{multline*}
counts marked bipolar orientations
with $n+1$ plain edges, $k$ inner faces, and signature $(a,b;c,d)$.

Coming back to the involutions $\rho$ and $\sigma$ of
Definition~\ref{def:symmetries}, we see that $\Psi$ behaves simply with respect to $\rho$ (it amounts to rotate $\Psi(w)$ by a half-turn), but that the
transformation induced by $\sigma$ is not simple. Accordingly, the
underlying  symmetry $\tq_{n,k}(a,b,c,d)=\tq_{n,k}(d,b,c,a)$
is not clear from the above determinant.

When  $a=d=0$, we have $q_{n,k}(0,b,c,0)=\tq_{n,k}(0,b,c,0)$, and the
map $\Psi$, composed with (the reverse of) the KMSW bijection $\Phi$,
sends bipolar orientations with $n+1$ edges, $k$ inner faces, 
left outer boundary of length $b+1$, right outer boundary of length $c+1$, 
bijectively onto non-intersecting triples of directed walks joining 
$A_1=(1,-1)$, $A_2=(0,0)$, $A_3=(0,b+1)$ 
to $B_1=(k,n-k-c-1)$, $B_2=(k,n-k)$, $B_3=(k-1,n-k+1)$.   
(Another bijection between these two families can also be deduced
from the correspondence between plane bipolar orientations
and twin pairs of binary trees given in~\cite{felsner2011bijections},
and from the encoding of twin binary trees given
in~\cite{dulucq1998baxter}, upon taking the mirror encoding for the second
binary tree.)

The further specialization $a=b=c=d=0$ gives a bijection for bipolar
orientations of a digon having $n+1$ edges and $k$ inner faces.
Upon deleting  the two outer edges (for $n\geq 3$), we obtain a  bipolar orientation having $n-1$ edges and $k-2$ inner faces. In the corresponding tandem walk, the
  first step is necessarily of the form $(0,j)$, and the last one of
  the form $(-i,0)$. Hence, in  the associated non-intersecting 
triple $\hat D_1,D_2,\hat D_3$, the first steps of $\hat D_1$ and $D_2$, 
and the last steps of $D_2$ and $\hat D_3$, are always east, hence  these
four steps can be deleted.
We thus recover the fact that plane bipolar orientations
with $n-1$ edges and $k-2$ inner faces are in bijection with non-intersecting
triples of lattice walks from $(2,-1),(1,0),(0,1)$ to $(k,n-k-1),(k-1,n-k),(k-2,n-k+1)$, which are counted by the Baxter 
summand (see~\eqref{baxter-bip}): 
\[
q_{n,k}(0,0,0,0)=
\frac{2}{n^2(n-1)}\binom{n}{k-2}\binom{n}{k-1}\binom{n}{k}.
\]

%=============================================================
\subsection{Random generation of tandem walks in the quadrant} 
%=============================================================

For $p\geq 1$, let $z$ and $z_0,\ldots,z_p$ be step probabilities 
satisfying~\eqref{eq:dist} and~\eqref{eq:drift}.
We let $\cQ_n$ denote the family of $p$-tandem walks of length $n$ in the quadrant starting at the origin,  and  $\cE_n$ the subfamily of those that end at the
origin (\emm excursions,). We consider the problem of generating a random walk in $\cQ_n$ (or $\cE_n$) such that each walk $w$ occurs with 
probability proportional to $z^k\prod_r z_r^{n_r}$,
with $k$ the number of SE steps in $w$ and $n_r$ the number of face steps of level $r$. 
With $\boldsymbol z=(z,z_0,\ldots,z_p)$, we refer to such random walks as \emm $\boldsymbol z$-distributed,.

Regarding $\cQ_n$, the bijection of {Proposition~\ref{prop:invol}}
reduces this problem to the random generation of {$\boldsymbol z$-distributed}
 tandem walks in the upper half-plane, starting at the origin and 
ending on the $x$-axis. Projected to the $y$-axis, these walks correspond to a model of  critical Galton--Watson trees, and can thus be randomly generated in linear time~\cite{devroye2012simulating}.

We thus focus on  $\cE_n$. We begin with the special case
  $p=1$ and $z_0=0$. Then $z=z_1=1/3$, and the $\boldsymbol z$-distribution is simply   uniform.
This case is particularly simple due to the existence of closed
  form expressions. Indeed, the number   of walks of length $n=3m+2i+j$ starting at the origin and ending at $(i,j)$ is (see \cite[Prop.~9]{mbm-mishna})
\[
  q_n(0,0;i,j)=\frac{(i+1)(j+1)(i+j+2)(3m+2i+j)!}{m!\,(m+i+1)!\,(m+i+j+2)!}.
  \]
This makes it easy to draw a uniform  excursion in $\cE_n$ step by step in time $O(n)$ (in reverse order, that is, from  the last one to the first one). This
is how we generated the tandem walk encoding the random bipolar orientation with
triangular faces of Figure~\ref{fig:random}.
  
For the general case $p\geq 1$, 
 we propose an {almost linear} algorithm under two relaxations:
\begin{enumerate}
\item 
[(i)] the length is not exactly prescribed, but lies in a linear-size 
window $\llbracket 2n, 3n\rrbracket$,
\item  [(ii)] the distribution conditioned to a given size $m\in\llbracket 2n,
3n\rrbracket$ coincides with
  the desired   distribution only   asymptotically (meaning that the
total variation distance between the actual distribution and the
$\boldsymbol z$-distribution is $o(1)$).
\end{enumerate}

To obtain a random excursion $w$ we generate two random walks $w_1,w_2$ in the quadrant, of respective lengths $n$ and $2n$, using the above random generators for $\cQ_n$ and $\cQ_{2n}$.
Let $(a,b)$ be the ending point of $w_1$.
If $w_2$ does not visit $(b,a)$ between times $n$ and $2n$, then we declare
a failure situation and restart generating $w_2$ until we obtain 
a walk visiting $(b,a)$ between times $n$ and $2n$.  
Then let $n'\in\llbracket n, 2n\rrbracket$ be the index of the last
visit of $w_2$ to $(b,a)$. Let $w_3$ be the
    prefix of length $n'$ of $w_2$, and let $\ow_3$ be obtained by
    reversing time in $w_3$, and applying an $x/y$-symmetry. That is,
  if the $r$th step of $w_3$ is $(i,j)$ then 
the $(n'-r+1)$th step of $\ow_3$ is $(-j,-i)$. As already used in
Section~\ref{subsec:LLT}, $\ow_3$ is also a tandem walk.
Then  the concatenation $w$ of $w_1$ and $\ow_3$ is a random excursion
in $\cE_m$, where $m=n+n' \in \llbracket 2n, 3n\rrbracket$.

A random excursion in $\cE_m$ is called  
{\emph{$n$-twisted}} if it is {$\boldsymbol z$-distributed} \emm once conditioned on the point, $(X_n,Y_n)$ \emm
visited after $n$ steps,. Clearly
the random excursion $w$ constructed by the above procedure
is $n$-twisted.  Moreover, by Proposition~\ref{prop:Sabconve}, as $n\to\infty$ the rescaled
random point $\frac1{\sigma\sqrt{n}}(X_n,Y_n)$ converges to the law of density 
\[
g(x,y)=\frac1{\sqrt{3\pi}}\,xy(x+y)\exp\big(-\tfrac1{3}(x^2+y^2+xy)\big),
\]
and a local limit statement also holds, following from Proposition~\ref{prop:local-cone}. Let us compare this behaviour to the limit density of 
  $\frac1{\sigma\sqrt{n}}(X_n,Y_n)$ in a $\boldsymbol z$-distributed excursion of $\cE_m$.  Let $q_n(i,j):= \PP( S^{(0,0)}(n)=(i,j), \tau^{(0,0)}>n)$ denote the $\boldsymbol z$-weighted number of quadrant walks 
of length $n$ starting at the origin and ending at $(i,j)$.
Then the probability that a $\boldsymbol z$-distributed excursion of $\cE_m$ is at $(i,j)$ after $n$ steps is proportional to $q_n(i,j)q_{m-n}(j,i)$. This, combined with Proposition~\ref{prop:local-cone}, implies that  the rescaled random point $\frac1{\sigma\sqrt{n}}(X_n,Y_n)$ in a  $\boldsymbol z$-distributed excursion of $\cE_m$  asymptotically follows the law of density 
\[
c_\alpha\,  g(x,y)g(x/\sqrt{\alpha},y/\sqrt{\alpha}),
\]
where $\alpha:=(m-n)/n\in[1,2]$, for some constant $c_\alpha$. Moreover a local limit statement also holds. 

We would now like to twist the way we produce the excursion  $w$ so that the distribution of its $n$th point approaches the above distribution.
This is classically done by adding a rejection step (see, e.g., \cite[II.3]{Devr86}). To be precise, let $g_0\approx 0.267$
be the maximal value of $g(x,y)$ (attained at $(x,y)=(\sqrt{6}/2,\sqrt{6}/2)$ and at two other points). Then we repeat calling the random sampler above, at each attempt producing 
an $n$-twisted random excursion of length $m\in \llbracket 2n, 3n\rrbracket$ and then flipping a coin with success probability
$\frac1{g_0}g(\frac{X_n}{\sigma\sqrt{\alpha
    n}},\frac{Y_n}{\sigma\sqrt{\alpha n}})$, where
{$\alpha=(m-n)/n$, and $(X_n,Y_n)$ is the point reached after $n$ steps.} 
 We return the excursion for the first successful attempt. With this additional rejection step we obtain a random sampler that is asymptotically $\boldsymbol z$-distributed.

 Let us discuss (heuristically)   the time complexity of the sampler. The coin-flipping   probability of success is $\Theta(1)$ (uniformly over
 $\alpha\in[1,2]$), hence the time complexity is of the same order as the one of the  $\boldsymbol z$-twisted random excursion sampler.  
The generation of $w_1$ takes time $O(n)$ as already mentioned. Moreover it is well known~\cite{erdos51} that, in   
the simple random 2-dimensional walk of length $n$, the number of
distinct points visited by the walk is of the order of
$n/\log(n)$. Hence we can expect that the number of attempts needed
to generate $w_2$ should be of the order of $\log(n)$. We can thus expect the overall time complexity of the sampler to be of the order of $n\log(n)$. 

%=============================================================
\subsection{Some remarks on the discrete harmonic function  $\boldsymbol{V(a,b)}$}
%=============================================================
The harmonic function $V(a,b)$ of $p$-tandem walks in the
quadrant, given  by~\eqref{eq:expression_H}, is ubiquitous throughout
Section~\ref{sec:asymptotic}, as it expresses the dependence of
various asymptotic behaviours in terms of the starting point $(a,b)$ of
the random walk; see Propositions~\ref{prop:asympz_enum}
and~\ref{prop:local_limit}.
The notion of discrete harmonic function is  intrinsically
interesting, as it is related to many probabilistic problems (Doob's $h$-transform and (non-)uniqueness problem, to quote a few of them). In this section we present some key features of $V(a,b)$.

%=============================================================
\subsubsection{The discrete harmonic function and
   Tutte's   invariants}
\label{sec:Tutte}
% ==============================================================

Recall that the values $V(a,b)$ must be related by the
identity~\eqref{eq:recurrence_relation_V}, which follows from a first step
  decomposition of random quadrant walks. This translates into a functional
  equation for the associated \gf\ $\mathcal V(u,v)$ defined
  by~\eqref{eq:expression_H},
  \beq\label{eqfunc-h}
  \left(\oS(u,v)-1\right)\mathcal V(u,v)=z\bu v\mathcal V(0,v)+
 \sum_{0<j\le r\le p} z_r u^{r-j} \bv^j \sum_{k=0}^ {j-1} v^k\mathcal
 V_k(u),
\eeq
where
\[
  \oS(u,v)= z\bu v + \sum_{0\le j \le r \le p} z_r u^{r-j} \bv^j,
\]
and
$\mathcal V_k(u)$ is the coefficient of $v^k$ in $\mathcal
V(u,v)$. Observe that $\oS(u,v)$ is closely related to the step
polynomial $S(x,y)$ defined by~\eqref{S:def}. In fact, $\oS(u,v)$  is
$S(\bu, \bv)$, taken at
$t=1$ and with a weight $z$ on the first step. This comes from the
fact that the  recurrence relation~\eqref{eq:recurrence_relation_V}
satisfied by the harmonic function is, in some sense, dual to the
recurrence corresponding to the enumeration of walks.

Let $U_0, U_1, \ldots, U_p$ be the $p+1$  solutions to $\oS(u,v)=1$
(when solved for $u$). They are algebraic functions of $v$, $z$ and
the $z_r$'s. Upon writing
\begin{equation}
\label{eq:S-weighted}
\oS(u,v)=z\bar{u}v+\sum_{r=0}^{p}z_{r}\, \frac {u^{r+1}-\bv^{r+1}}{u-\bv},
\end{equation}
we see that the equation $\oS(u,v)=1$ implies 
\beq\label{I0}
  I_0(u)=I_0(\bv), \qquad \hbox{with} \qquad I_0(u)= u+z\bu
  -\sum_{r=0}^p z_r u^{r+1}.
\eeq
In particular, all the functions $I_0(U_i)$ are equal: using the
terminology of Tutte~\cite{tutte-chromatic-revisited}, we say that
the rational function $I_0(u)$ is an \emm
invariant,. We refer to~\cite{BeBMRa-16,FrRa17}   for recent applications of invariants to quadrant problems.

It is shown in~\cite{Ra-14} how, in the case of small steps, the notion
of invariants can be used to determine discrete harmonic functions in
the quadrant. Let us illustrate this for the case $p=1$
of tandem
walks (we take moreover $z_0=0$, so that $z=z_1=1/3$ as discussed in
the examples following Proposition~\ref{prop:asympz_enum}). In this case, the
functional equation~\eqref{eqfunc-h} reads
\[
 3 uv\left( \oS(u,v)-1\right) \mathcal V(u,v)= v^2 \mathcal V(0,v)+
  u \mathcal V(u,0).
\]
In~\cite{Ra-14}, it is explained how one can give an analytic meaning to
$\mathcal V(u,v)$ (even without knowing that it is a simple rational
function), such that $u$ can be specialized to $U_0$ and $U_1=U_p$ in
the above equation. Then the left hand-side vanishes, and it follows that
\[
  0= v^2 \mathcal V(0,v)+ U_0 \mathcal V(U_0,0)= v^2 \mathcal V(0,v)+
  U_1 \mathcal V(U_1,0),
\]
so that the function $I(u):=u\mathcal V(u,0)$ is another invariant (if
$p=1$). Remarkably,  this property, together with some properties peculiar to harmonic functions, eventually characterizes $I(u)$, and gives an expression for
it in terms of the rational invariant $I_0(u)$ defined by~\eqref{I0}:
\[
  I(u)= u\mathcal V(u,0)= \frac{C}{I_0(u)-I_0(1)},
  \]
  for some normalizing constant $C$ that can be determined (and is
  $2\sqrt 3$ here).

We now   return to general values of $p$. No counterpart of the
analytic framework of~\cite{Ra-14} has been developed for large steps,
and the role of invariants remains to be worked out in this setting. However, it follows from the exact value of $\mathcal V(u,v)$, given in
Proposition~\ref{prop:asympz_enum}, that the above identity still holds.

\begin{Proposition}
  For $p$-tandem walks in the quadrant, the \gf\ $\mathcal V(u,v)$ of Proposition~\ref{prop:asympz_enum}  satisfies
  \[
    u\mathcal V(u,0)= \frac{2}\sigma \cdot \frac 1{I_0(u)-I_0(1)},
    \]
    where $I_0(u)$ is the rational invariant defined by~\eqref{I0} and
    $\sigma$ is given by~\eqref{eq:sigma}. In particular, $u\mathcal V(u,0)$ is
      also an invariant.
  \end{Proposition}
 The proof is a simple calculation, using the normalizing
 identity~\eqref{eq:dist} and the zero-drift identity~\eqref{eq:drift}. We leave it to the reader.\qed

%=============================================================
\subsubsection{Probabilistic features}
\label{sec:proba-harmonic}
%=============================================================

The aim of this  section is twofold: we first comment on the links between the discrete harmonic function $V(a,b)$ and its continuous counterpart
$V_\infty(a,b)$ (see \eqref{eq:asymptotic_V(a,b)}); then we give an estimation of $S^{(a,b)}(\tau^{(a,b)})$, the location of the random walk at its first exit (Proposition~\ref{prop:estimate_exit}).

The asymptotic result \eqref{eq:asymptotic_V(a,b)} as both $a$ and $b$ tend to infinity,
\begin{equation*}
     V(a,b)\sim \frac{V_\infty(a,b)}{\sigma^3},
\end{equation*}
can be derived from purely probabilistic arguments: it is proved in
\cite[Lem.~13]{denisov2015random} that, when the covariance
  matrix is the identity, the discrete harmonic function $V(a,b)$ is
asymptotically equivalent to the continuous harmonic function
$V_\infty(a,b)=ab(a+b)$.
The above factor $\frac{1}{\sigma^3}$ is thus due to a different
normalization.
It is remarkable that this factor captures all
the  dependence in the model  in the parameters $z$ and $z_r$.
The function $V_\infty(a,b)=ab(a+b)$ can thus be viewed as a
universal harmonic function for our class of models. This
  universality is also visible in the expression for the covariance
  matrix in Lemma~\ref{lem:covari}.

In fact we  can go further and state an exact formula relating $V$ and
$V_\infty$.

\begin{Proposition}
  \label{prop:estimate_exit}
  Let us define the following shifted version of $V_\infty$:
  \[
    V_\infty^{\rm s}(a,b):=V_\infty(a+1,b+1)=(a+1)(b+1)(a+b+2).\] Then
\begin{equation}
\label{eq:estimate_exit}
V^{\rm s}_\infty (a,b)-\sigma^3V(a,b)=     \EE \left(V^{\rm s}_\infty (S^{(a,b)}(\tau^{(a,b)}))\right),
\end{equation}
where $S^{(a,b)}(\tau^{(a,b)})$ is the position where the random
tandem walk started at $(a,b)$ leaves the quadrant for the first time.
\end{Proposition}

\begin{remarks}
{\bf 1.} If $p=1$ then $S^{(a,b)}(\tau^{(a,b)}))$ is necessarily
of the form $(-1,j)$ or $(i,-1)$, for some integers $i,j\geq
{0}$. Since $V^{\rm s}_\infty$ vanishes at these points,
the expectation in \eqref{eq:estimate_exit} is zero, which is
consistent with~\eqref{Vab1}.

\medskip

\noindent{\bf 2.} If $p\geq2$,  there are again two possibilities. The random walk may exit through the horizontal boundary, in which case $S^{(a,b)}(\tau^{(a,b)})$ is still of the form $(i,-1)$, where $V^{\rm s}_\infty$
vanishes. It may also exit through the vertical boundary, and then
$S^{(a,b)}(\tau^{(a,b)})=(-i,j)$, for $1\le i \le p$
  and $j\ge 0$ and $V^{\rm s}_\infty{(-i,j)=(1-i)(1+j)(j-i+2)}$ does not
  {necessarily} vanish.
  Hence   we can interpret~\eqref{eq:estimate_exit}
   as an estimation of the {square of the}  vertical
  coordinate, when an exit occurs {through the vertical boundary.}
\end{remarks}

\begin{proof}[Proof of Proposition~\ref{prop:estimate_exit}]
  Let us define the function $f$ by
  \beq
\label{eq:formula_(4)_DeWa15}
     f(a,b)=\EE \left(V^{\rm s}_\infty ((a,b)+(X,Y))\right)-V^{\rm s}_\infty (a,b).
\eeq
Then it follows from~\cite[Eq.~(5)]{denisov2015random} that
\[
     \sigma^3 V(a,b)=V^{\rm s}_\infty (a,b)-\EE \left(V^{\rm s}_\infty (S^{(a,b)}(\tau^{(a,b)}))\right)+\EE \left(\sum_{m=0}^{\tau^{(a,b)}-1}f(S^{(a,b)}(m))\right).
\]
There are two slight differences between this formula
and Eq.~(5) in~\cite{denisov2015random}: first, for the same reasons
as above, we need the factor $\sigma^3$; second, the right-hand side of our
formula involves $V^{\rm s}_\infty$ rather than  $V_\infty$ because
the boundary axes of~\cite{denisov2015random} are shifted by $(1,1)$
(more precisely, the walk is killed on the coordinate axes).

The function $f$ defined by~\eqref{eq:formula_(4)_DeWa15} can be
  computed explicitly, and is found to be $0$. This completes the
  proof of the  proposition. We compute $f$ below  in an
  independent lemma, which shows again the rich structure of   tandem walks.
\end{proof}

\begin{Lemma}
\label{lem:u_martingale}
The functions $V_\infty$ (and $V_\infty^{\rm s}$) are discrete harmonic functions for the random walk in $\mathbb{Z}^2$ with increments $(X,Y)$ satisfying~\eqref{distri} (with no killing), in the sense that for all $(a,b)\in\mathbb{Z}^2$,
\begin{align*}
\label{eq:global_harmonic}
V_\infty(a,b)&=\EE \left(V_\infty ((a,b)+(X,Y))\right)
\\
&=zV_\infty(a+1,b-1)+\sum_{r=0}^{p}z_{r} \sum_{i+j=r}V_\infty(a-i,b+j),
\end{align*}
and similarly for $V_\infty^{\rm s}$.
\end{Lemma}

\begin{proof}
  Recall that $V_\infty(a,b)=ab(a+b)$. Hence
  \begin{equation*}
  \EE \left(V_\infty ((a,b)+(X,Y))\right)=
  z(a+1)(b-1)(a+b)+\sum_{r=0}^{p}z_{r} \sum_{i+j=r}(a-i)(b+j)(a-i+b+j).
\end{equation*}
We write $j=r-i$, and sum over $i=0, \ldots, r$ for $r$
  fixed. Upon putting apart the cubic terms, this transforms the above expression into 
\begin{equation*}
  zab(a+b)+\sum_{r=0}^{p}z_{r} (r+1)ab(a+b)+
  z(b-a-1)(a+b)+
    \sum_{r=0}^{p}z_{r}\frac{r(r+1)}{2}(a+b)(a-b+1).
\end{equation*}
Because of the drift condition~\eqref{eq:drift}, the sum of the
third and fourth term vanishes. Moreover, the sum of the first two
terms equals $ab(a+b)=V_\infty(a,b)$, due to the normalization
condition~\eqref{eq:dist}. This completes the proof for
$V_\infty$. Finally, since $V_\infty^{\rm    s}(a,b)=V_\infty(a+1,b+1)$, this function is also harmonic.
\end{proof}

%==============================================================
\subsection{More combinatorics for 1D walks} 
%===============================================
In this section, by a \emm 1D walk, we mean any walk with steps in $\{-1\}\cup\mathbb{N}=\{-1,0,1, \ldots\}$. Such walks are classically represented as directed 2D walks, upon drawing each step $i$ as 
$(1,i)$. 
For $\cF$ a family of 1D walks, we 
denote by $F$ the associated \gf \ in $\mathbb{Q}[w_{-1},w_0,w_1,\ldots][[t]]$, 
where the variable $t$ records the length, and $w_{i}$ the number of steps~$i$, for $i\geq -1$. For such walks, we give here a combinatorial proof of two enumerative results that we have used in the paper. The first one deals with  nonnegative walks going from height~$0$ to height $a$, and this result has been used in Section~\ref{sec:prop_alg_combi}.
The only published proofs that we know for it are algebraic~\cite{banderier-flajolet,bousquet-petkovsek-recurrences,gessel-factorization}. The second one deals with nonnegative walks going from height $k$ to height $0$, and has been established algebraically when proving Corollary~\ref{cor:Qthird}. (In practise we translate these walks so that they start at $0$, end at $-k$, and never visit any vertex of height less than $-k$.) 

We begin with preliminary arguments that have already been used in Section~\ref{sec:bij_proofs}. For $a\geq 0$ we let $\cH_a$ be the family of 1D walks going from height $0$ to height $a$ with nonnegative height all along.  
Furthermore, for $k\geq 0$ we let $\cD_k$ be the family of 1D walks going from height $0$ to height $-k$ and with height at least $-k$ all along (note that $\cD_0=\cH_0$), 
and we let $\cL_k$ be the subfamily of those that reach height $-k$ only at their ending point.  
By considering the first step $i$ of a walk, we see that $Y:= L_1$ is the unique power series in $t$ satisfying
\[ 
Y=t\sum_{i\geq-1}w_iY^{i+1}.
\] 
Moreover, for $k\geq 1$ a decomposition of walks of $\mathcal L_k$ at their first visits to heights $-1,\ldots,-k+1$ shows that $L_k=Y^{k}$. Since, clearly, $L_k=tw_{-1}D_{k-1}$, this gives, for $k\geq 0$, 
\[
D_k=\frac1{tw_{-1}}{Y^{k+1}}.
\]
This is the analogue of~\eqref{HY}.

%=====================================================================
\subsubsection{Nonnegative walks from $\boldsymbol 0$ to $\boldsymbol a$}
\label{sec:app1}
%=====================================================================

The classical \emm kernel method,~\cite{banderier-flajolet,bousquet-petkovsek-recurrences} shows that 
\[
\sum_{a\ge 0} H_a u^a= \frac{1-\bu Y}{1-tS(u)},
\]
where $\bu=1/u$ and $S(u)=\sum_{i\ge -1}w_i u^i$ is the \gf\ of the
steps. As argued  in the proof of
  Lemma~\ref{lem:interpretation}, this can be rewritten as
\begin{equation}\label{eq:Ha}
H_a=\frac{Y}{tw_{-1}}[u^a]\frac{1}{1-\frac{uY}{w_{-1}}\sum_{j,k\geq 0}w_{j+k+1}u^kY^j}.
\end{equation}
We now prove this identity combinatorially via standard path decompositions.
As we will show,  it implies the expression~\eqref{1D-res-bis} for the series $H^{0\rightarrow a}$ counting tandem walks in the upper half-plane ending at ordinate $a$.

In a 1D walk a \emph{record} is a point whose height is strictly smaller than the height
of all subsequent points (by convention the endpoint is 
always considered as a record).  For a walk in $\cH_a$ 
the heights of the successive records form an increasing sequence 
 $0,i_1,i_1+i_2,\ldots,i_1+\cdots+i_e$, where  $i_1,\ldots,i_e$ are positive
and $i_1+\cdots+i_e=a$; the sequence $(i_1,\ldots,i_e)$ is called
the \emph{record-sequence} of the walk (note that the record-sequence is empty if and only if $a=0$).   
For $k\geq 1$ we let $U_k$ be the \gf\ of those walks
in $\cH_k$ whose
only records are the endpoints. This means that all points, except the starting point, have height at least $k$.  Clearly, the \gf\ of walks
in $\cH_a$ that have record-sequence $(i_1,\ldots,i_e)$ is then equal to $H_0\prod_{s=1}^eU_{i_s}$.
Hence, by summing over all possible record-sequences, and using $H_0=\frac1{tw_{-1}}Y$, we find:
\[
H_a=\frac{Y}{tw_{-1}}[u^a]\frac1{1-\sum_{k\geq 1}U_ku^k}.
\] 
We finally express the series $U_k$ for $k\geq 1$. A  walk  counted by $U_k$ consists of a first step  $h>k$, followed by a (translated) walk of $\mathcal D_{h-k}$. The associated \gf\ is thus $tw_h D_{h-k}=\frac{w_h}{w_{-1}}Y^{h-k+1}$. Hence we have  
\[
U_k=\frac{Y}{w_{-1}}\sum_{h\geq k\ge 1}w_{h}Y^{h-k},
\]
so that
\[
\sum_{k\geq 1}U_ku^k=\frac{uY}{w_{-1}}\sum_{j,k\geq 0}w_{j+k+1}u^kY^j,
\]
which gives~\eqref{eq:Ha}. 

This result can now be applied to tandem walks.  
With the notation of Section~\ref{sec:prop_alg_combi},  a tandem walk in $\cH^{0\rightarrow a}$, once projected on the vertical axis, becomes a 1D walk in $\cH_a$: every SE step projects onto a $-1$ step, and for $r\geq 0$, a step $r$ in the 1D walk can arise from any step $(-i,r)$ with $i\geq 0$.
Hence, under the specialization $\{w_{-1}=1, w_s=\sum_{r\geq s}z_r\ \mathrm{for}\ s\geq 0\}$,
we have $Y=W$ and $H_a=H^{0\rightarrow a}$. We thus recover~\eqref{1D-res-bis}.   

%=====================================================================
\subsubsection{Nonnegative walks from $\boldsymbol k$ to $\boldsymbol 0$}
%=====================================================================
\label{sec:proof_eq_const}
  
Let $k\geq 1$. In this section we give a combinatorial proof of the following constant term expression for $L_k$:
\begin{equation}\label{eq:const_Y}
L_k=Y^{k}=-t[y^0]y^{1+k}\frac{S'(y)}{K(y)},
\end{equation}
where $S(y)=\sum_{i\geq -1}w_iy^i$ and 
$K(y):=1-tS(y)$. Note that it directly yields the identity~\eqref{eq11} 
used to prove Corollary~\ref{cor:Qthird}, by taking the specialization for tandem walks: 
$\{w_{-1}=x, w_s=\sum_{r\geq s}\bx^{r-s}z_r\ \mathrm{for}\ s\geq 0\}$. 

\begin{figure}
\begin{center}
\includegraphics[width=\linewidth]{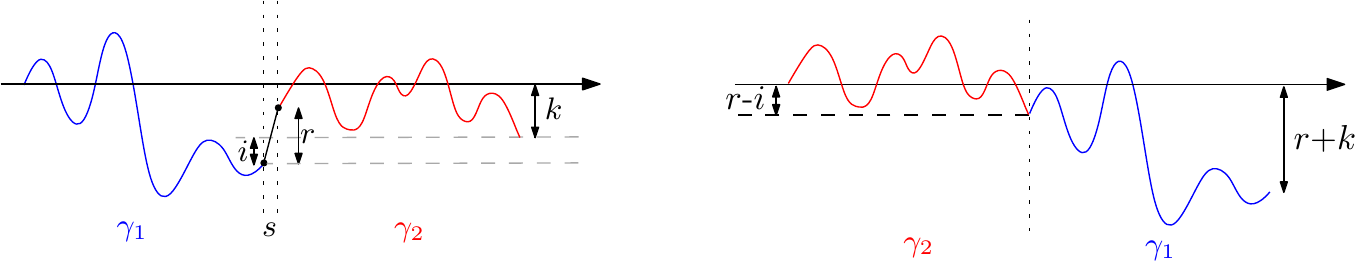}
\end{center}
\caption{The bijection between $\cN_{k,r,i}$ (for any fixed $i\le r$) and $\cW_{k+r}$.} 
\label{fig:luka}
\end{figure}

Our combinatorial argument is borrowed from~\cite{renault2008lost} (see also the interesting discussion therein regarding the 
fact that, for walks with steps $\pm 1$, the argument is originally 
due to D\'esir\'e Andr\'e, to whom the reflection principle was wrongly attributed).

Let $\cW_k$ be the family of 1D walks ending at height $-k$, and 
let $\cN_k=\cW_k\backslash \cL_k$ be the subset of those walks that visit height $-k$ before their final point. 
For $\gamma\in\cN_k$, we define the \emph{marked step} of $\gamma$ as the last step that starts at height $\leq -k$. For $0\le i \le r$, 
let $\cN_{k,r,i}$ be the subfamily of $\cN_k$ where the marked step   is a step $+r$ that 
starts at height $-k-i$ (if $i=r$ then the marked step is always the last step).We refer to Figure~\ref{fig:luka} for an illustration.
We claim that $\cN_{k,r,i}$ and $\cW_{k+r}$ are in bijection.
To $\gamma\in\cN_{k,r,i}$,  
written as $\gamma_1s\gamma_2$ with $s$ the marked step,  we associate
 $\gamma_2\gamma_1\in\cW_{k+r}$. Conversely, for $\gamma\in\cW_{k+r}$, we let $\gamma_2$ be the prefix of 
$\gamma$ ending at the first visit to height $i-r$ (if $i=r$ then $\gamma_2$ is empty), and we take $\gamma_1$ as the corresponding suffix of $\gamma$; 
we associate to $\gamma$ the walk $\gamma_1s\gamma_2$ where $s$ is a $+r$ step.   
In terms of generating functions this gives
$N_{k,r,i}=tw_rW_{k+r}$. Hence
$$
N_k=\sum_{r\geq 0}\sum_{i=0}^{r}N_{k,r,i}=\sum_{r\geq 0}\sum_{i=0}^{r}tw_rW_{k+r}=t\sum_{r\geq 0}(r+1)w_rW_{k+r}.
$$
Moreover, for all $j\geq 1$, we clearly have
$$
W_j(t)=[y^{-j}]\frac{1}{K(y)}=[y^0]\frac{y^j}{K(y)}.
$$
Since $L_k=W_k-N_k$, for $k\ge 1$ we obtain
\begin{align*}
L_k&=[y^0]\frac{y^k}{K(y)}\cdot\Big(1-t\sum_{r\geq 0}(r+1)w_ry^{r}\Big)\\
&=[y^0]\frac{y^k}{K(y)}\cdot\Big(K(t,y)+tw_{-1}y^{-1}-t\sum_{r\geq 0}rw_ry^{r}\Big)\\ 
&=t[y^0]\frac{y^k}{K(y)}\cdot\Big(w_{-1}y^{-1}-\sum_{r\geq 0}rw_ry^{r}\Big) .
\end{align*}
Since $yS'(y)=-y^{-1}w_{-1}+\sum_{r\geq 0}rw_ry^{r}$, this gives the expression~\eqref{eq:const_Y} for $L_{k}$.

\bigskip
\noindent{\bf Acknowledgements.} We thank warmly both
% the two
anonymous referees for their thorough reading, their useful comments and suggestions. We are grateful to Fr\'ed\'eric Chyzak and Marni Mishna for interesting discussions, to Pierre Lairez who computed recurrence relations from our constant term expressions in Section~\ref{sec:main_result_start0b}, and to J\'er\'emie Bettinelli who produced the geometric representation of the random bipolar orientation shown in Figure~\ref{fig:random}.

%%%%%%%%%%%%%%%%%%%%%%%%%%%%%%%%%%%%%%%%%%%%%%%%%
\bibliographystyle{plain}

\end{document}